\documentclass[a4paper,reqno]{amsart}

\usepackage[T1]{fontenc}
\usepackage[utf8x]{inputenc}
\usepackage[english]{babel}
\usepackage{yfonts}
\usepackage{dsfont}
\usepackage{amscd,amssymb,amsmath,amsthm,amsfonts}
\usepackage{mathtools}
\usepackage{accents}
\usepackage{tensor}
\usepackage{graphicx}
\usepackage{tikz}
\usetikzlibrary{calc,matrix,arrows,decorations.pathmorphing}
\usepackage{calc}
\usepackage[cal=boondox,scr=boondoxo]{mathalfa}
\usepackage{enumitem}
\usepackage{marginnote}
\usepackage{booktabs}
\usepackage{scalerel}[2016/12/29]
\usepackage{url}
\usepackage{contour}
\usepackage{ulem}

\usepackage{hyperref}
\hypersetup{colorlinks=true,pageanchor=false,
linkcolor=blue,citecolor=red,urlcolor=red}
\usepackage[all]{hypcap}

\allowdisplaybreaks

\newtheorem{theorem}{Theorem}[section]
\newtheorem{corollary}[theorem]{Corollary}
\newtheorem{proposition}[theorem]{Proposition}
\newtheorem{lemma}[theorem]{Lemma}

\theoremstyle{definition}
\newtheorem{definition}[theorem]{Definition}

\theoremstyle{remark}
\newtheorem{remark}[theorem]{Remark}

\newtheorem{conjecture}[theorem]{Conjecture}

\newcommand{\N}{\mathbb{N}}
\newcommand{\Z}{\mathbb{Z}}

\newcommand{\R}{\mathbb{R}}

\newcommand{\rmC}{\mathrm{C}}

\newcommand{\rmT}{\mathrm{T}}

\newcommand{\bbC}{\mathbb{C}}

\newcommand{\bbQ}{\mathbb{Q}}

\newcommand{\calC}{\mathcal{C}}

\newcommand{\calE}{\mathcal{E}}

\newcommand{\calH}{\mathcal{H}}

\newcommand{\calL}{\mathcal{L}}

\newcommand{\frakg}{\mathfrak{g}}
\newcommand{\frakh}{\mathfrak{h}}

\newcommand{\ad}{\mathrm{ad}}

\newcommand{\id}{\mathrm{id}}

\newcommand{\lev}{\smash{\stackrel{\leftarrow}{\mathrm{ev}}}}
\newcommand{\lcoev}{\smash{\stackrel{\longleftarrow}{\mathrm{coev}}}}
\newcommand{\rev}{\smash{\stackrel{\rightarrow}{\mathrm{ev}}}}
\newcommand{\rcoev}{\smash{\stackrel{\longrightarrow}{\mathrm{coev}}}}

\DeclareMathOperator{\Forall}{\forall}
\DeclareMathOperator{\ttimes}{\widetilde{\times}}

\DeclareMathOperator{\disjun}{\sqcup}
\DeclareMathOperator{\bcs}{\natural}

\newcommand{\leqs}{\leqslant}
\newcommand{\geqs}{\geqslant}

\newcommand{\mods}[1]{\operatorname{\mathnormal{#1}-mod}}

\newcommand{\fsl}{\mathfrak{sl}}
\newcommand{\SL}{\mathrm{SL}}

\newcommand{\Vect}{\mathrm{Vect}}

\newcommand{\op}{\mathrm{op}}

\DeclareRobustCommand{\one}{\mathbin{\text{\includegraphics[height=\heightof{$\mathbf{1}$}]{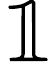}}}}

\newcommand{\sqbinom}[2]{\left[ \begin{matrix} #1 \\ #2 \end{matrix} \right]}

\makeatletter
\newcommand{\subalign}[1]{
  \vcenter{
    \Let@ \restore@math@cr \default@tag
    \baselineskip\fontdimen10 \scriptfont\tw@
    \advance\baselineskip\fontdimen12 \scriptfont\tw@
    \lineskip\thr@@\fontdimen8 \scriptfont\thr@@
    \lineskiplimit\lineskip
    \ialign{\hfil$\m@th\scriptstyle##$&$\m@th\scriptstyle{}##$\crcr
      #1\crcr
    }
  }
}
\makeatother

\def\clap#1{\hbox to 0pt{\hss#1\hss}}

\def\mathrlap{\mathpalette\mathrlapinternal}

\def\mathrlapinternal#1#2{%
\rlap{$\mathsurround=0pt#1{#2}$}}

\newcommand{\sig}{n}
\newcommand{\Cob}{\mathrm{Cob}}
\newcommand{\conn}{\mathrm{conn}}
\newcommand{\FRCob}{3\mathrm{Cob}^\sigma}
\newcommand{\RCob}{3\mathrm{Cob}}
\newcommand{\Chb}{\mathrm{Chb}}
\newcommand{\RHB}{4\mathrm{HB}}
\newcommand{\bfCob}{\mathbf{Cob}}

\newcommand{\ETTH}{\mathrm{TTan}^\sigma}
\newcommand{\TTH}{\mathrm{TTan}}
\newcommand{\KTan}{\mathrm{KTan}}
\newcommand{\Algf}{4\mathrm{Alg}}
\newcommand{\tAlgt}{3\mathrm{Alg}^\sigma}
\newcommand{\Algt}{3\mathrm{Alg}}
\newcommand{\HAf}{\calH_4}
\newcommand{\tHAt}{\calH_3^\sigma}
\newcommand{\HAt}{\calH_3}
\newcommand{\KLF}{J_4}
\newcommand{\KLFts}{J_3^\sigma}
\newcommand{\KLFt}{J_3}
\newcommand{\FrgtT}{F_\rmT}
\newcommand{\FrgtC}{F_\rmC}
\newcommand{\tQtnt}{Q^\sigma}
\newcommand{\Qtnt}{Q}

\newcommand{\pic}[2][0]{\raisebox{-0.5\height + 2.5pt + #1pt}{\includegraphics{#2.pdf}}}
\newcommand{\toppic}[1]{\raisebox{-\height}{\includegraphics{#1.pdf}}}

\newcommand\arxiv[2]{\href{https://arXiv.org/abs/#1}{\texttt{arXiv:\allowbreak #1} #2}}
\newcommand\doi[2]{\href{https://doi.org/#1}{#2}}

\contourlength{1pt}

\DeclareRobustCommand{\myuline}[1]{
 \ifmmode \text{\uline{$\phantom{#1}$}\llap{\contour{white}{$#1$}}}
 \else \uline{\phantom{#1}}\llap{\contour{white}{#1}} \fi
}

\numberwithin{equation}{section}
  
\begin{document}

\raggedbottom

\title{Kerler--Lyubashenko Functors on 4-Dimensional 2-Handlebodies}

\author[A. Beliakova]{Anna Beliakova} 
\address{Institute of Mathematics, University of Zurich, Winterthurerstrasse 190, CH-8057 Zurich, Switzerland} 
\email{anna.beliakova@math.uzh.ch}

\author[M. De Renzi]{Marco De Renzi} 
\address{Institute of Mathematics, University of Zurich, Winterthurerstrasse 190, CH-8057 Zurich, Switzerland} 
\email{marco.derenzi@math.uzh.ch}

\begin{abstract}
 We construct a braided monoidal functor $\KLF$ from Bobtcheva and Piergallini's category $\RHB$ of connected 4-di\-men\-sion\-al 2-han\-dle\-bod\-ies (up to 2-de\-for\-ma\-tions) to an arbitrary unimodular ribbon category $\calC$, which is not required to be semisimple. The main example of target category is provided by $\mods{H}$, the category of left modules over a unimodular ribbon Hopf algebra $H$. The source category $\RHB$ is freely generated, as a braided monoidal category, by a BPH algebra (short for Bobtcheva--Piergallini Hopf algebra), and this is sent by the Kerler--Lyubashenko functor $\KLF$ to the end $\int_{X \in \calC} X \otimes X^*$ in $\calC$, which is given by the adjoint representation in the case of $\mods{H}$. When $\calC$ is factorizable, we show that the construction only depends on the boundary and signature of handlebodies, and thus projects to a functor $\KLFts$ defined on Kerler's category $\FRCob$ of connected framed 3-di\-men\-sion\-al cobordisms. When $H^*$ is not semisimple and $H$ is not factorizable, our functor $\KLF$ has the potential of detecting diffeomorphisms that are not 2-de\-for\-ma\-tions.
\end{abstract}

\maketitle
\setcounter{tocdepth}{3}

\section{Introduction}

After the failure of the smooth \textit{h}-cobordism theorem in dimension~4, one of the main goals of 4-man\-i\-fold topology has become the detection of exotic smooth structures. Two smooth manifolds form an exotic pair if they are homeomorphic but not diffeomorphic. Many examples have been constructed, see for instance \cite{F82,A91,AP07,AR14} and references therein. Some of the simplest exotic pairs are given by Mazur manifolds \cite{HMP19}, which are contractible 4-man\-i\-folds obtained from the 4-ball $D^4$ by attaching a single 1-han\-dle and a single 2-han\-dle. Mazur manifolds are particularly simple instances of 4-di\-men\-sion\-al 2-han\-dle\-bod\-ies, which are 4-man\-i\-folds obtained from the 4-ball $D^4$ by attaching a finite number of 1-han\-dles and 2-han\-dles. There is another natural equivalence relation on the class of 4-di\-men\-sion\-al 2-han\-dle\-bod\-ies, known as \textit{$2$-e\-quiv\-a\-lence}. It is generated by \textit{$2$-de\-for\-ma\-tions}, which are those diffeomorphisms implemented by handle moves that never step outside of the class of 4-di\-men\-sion\-al 2-han\-dle\-bod\-ies. In other words, when considering 4-di\-men\-sion\-al 2-han\-dle\-bod\-ies up to 2-e\-quiv\-a\-lence, creation and removal of canceling pairs of handles of index 2/3 and 3/4 is forbidden. Whether 2-de\-for\-ma\-tions form a proper subclass of diffeomorphisms is still an open question, mainly due to the lack of sensitive invariants of 4-di\-men\-sion\-al 2-han\-dle\-bod\-ies up to 2-e\-quiv\-a\-lence. This paper delivers new families of computable invariants that could potentially detect  the difference between diffeomorphisms and $2$-de\-for\-ma\-tions.

\subsection{Main results}\label{S:main_results} We construct linear representations of a certain topological category whose morphisms are 2-de\-for\-ma\-tion classes of 4-di\-men\-sion\-al 2-han\-dle\-bod\-ies. More precisely, we define a braided monoidal functor $\KLF$ with source the category $\RHB$ of connected 4-di\-men\-sion\-al 2-han\-dle\-bod\-ies defined by Bobtcheva and Piergallini \cite{BP11}, and recalled in Section~\ref{S:relative_handlebody_category}, and with target a unimodular ribbon category $\calC$. Ribbon categories were introduced by Turaev in \cite{T94}, and have been successfully used for the construction of link and 3-man\-i\-fold invariants. A ribbon category is unimodular if it is finite and the projective cover of its tensor unit is self-dual, see Section~\ref{S:unimodular_factorizable_categories}. Our main result can then be stated as follows.

\begin{theorem}\label{T:main_4-dim}
 If $\calC$ is a unimodular ribbon category, then there exists a braided monoidal functor 
 \[
  \KLF : \RHB \to \calC
 \]
 sending the generating Hopf algebra of $\RHB$ (the solid torus) to the end 
 \[
  \calE = \int_{X \in \calC} X \otimes X^* \in \calC.
 \]
\end{theorem}

The technology used for our construction was developed over a period of twenty years by Majid \cite{M91}, Lyubashenko \cite{L94}, Crane and Yetter \cite{CY94}, Kerler \cite{K01}, Habiro \cite{H05}, Asaeda \cite{A11}, and Bobtcheva and Piergallini \cite{BP11}. Let us explain what is new in our approach.

The first attempts at constructing a topological invariant of $4$-manifolds using the algebraic input of a ribbon category were made by Turaev using the notion of \textit{shadows}, based on Matveev--Piergallini spines \cite{T94}. Unfortunately, the resulting state sum invariant reduces to the Turaev--Viro invariant of the 3-di\-men\-sion\-al boundary. Note that Turaev--Viro invariants, originally defined for quantum $\fsl_2$ in \cite{TV92}, and later generalized by Barrett and Westbury to arbitrary spherical fusion categories \cite{BW93}, are known to be determined by the corresponding Witten--Reshetikhin--Turaev (WRT) invariants whenever the underlying spherical fusion category is also a modular category \cite{BD93,T94}. For closed 4-man\-i\-folds, similar invariants due to Broda \cite{B93} and to Crane and Yetter \cite{CY93} were proven to boil down to homological invariants like the signature and the Euler characteristic \cite{R93,CKY93}. It should be noted that all these constructions require the use of semisimple categories.

One still can expect non-semisimple ribbon categories to provide combinatorial relatives of Donaldson--Floer invariants, an idea notably promoted by Crane and Frenkel in \cite{CF94}. The first work in this direction, due to Bobtcheva and Messia, was an elaborate extension of the Hennings--Kauffman--Radford (HKR) non-semisimple invariants of closed 3-man\-i\-folds to 4-di\-men\-sion\-al 2-han\-dle\-bod\-ies (up to 2-de\-for\-ma\-tions) \cite{BM02}. The construction relies on the presentation of 4-di\-men\-sion\-al 2-han\-dle\-bod\-ies in terms of Kirby tangles, with dotted and undotted components corresponding to 1-han\-dles and 2-han\-dles respectively. Hennings' evaluation based on the integral of a ribbon Hopf algebra is generalized by assigning special central elements to the two types of components. Unfortunately, no interesting candidates for these central elements were found, except for the known ones reproducing WRT and HKR invariants, and hence all the examples discussed by Bobtcheva and Messia essentially reduce to invariants of the 3-di\-men\-sion\-al boundary, together with the signature.

Our next result provides a conceptual explanation for this phenomenon. Indeed, all the constructions mentioned above were tested on factorizable ribbon categories, which are finite ribbon categories whose braiding satisfies the non-degeneracy condition recalled in Section~\ref{S:unimodular_factorizable_categories}. However, in the factorizable case, the functor $\KLF : \RHB \to \calC$ of Theorem~\ref{T:main_4-dim} factors through a boundary functor $\partial : \RHB \to \FRCob$ with target the category of connected framed 3-di\-men\-sion\-al cobordisms\footnote{All cobordisms considered in this paper are actually \textit{relative} cobordisms, see Section~\ref{S:signature_defects}.} of Section~\ref{S:framed_relative_cobordism_category}. More precisely, the boundary functor $\partial$ sends every 4-di\-men\-sional 2-han\-dle\-bod\-y $W$ to the framed 3-di\-men\-sional cobordism determined by its outgoing boundary $\partial_+ W$ and by its signature $\sigma(W)$, see Sections~\ref{S:relative_handlebodies} and \ref{S:signature_defects}. Furthermore, under additional assumptions on the ribbon structure of $\calC$, the construction is independent of the signature too, in the sense that it factors through a forgetful functor $\FrgtC : \FRCob \to \RCob$ with target the category of connected 3-di\-men\-sion\-al cobordisms without framings. Thus we have the following result.

\begin{theorem}\label{T:main_3-dim}
 If $\calC$ is a factorizable ribbon category, then $\KLF : \RHB \to \calC$ factors through a unique braided monoidal functor 
 \[
  \KLFts : \FRCob \to \calC,
 \]
 meaning there exists a unique commutative diagram of the form
 \begin{center}
  \begin{tikzpicture}[descr/.style={fill=white}]
   \node (P0) at (180:2.5) {$\RHB$};
   \node (P1) at (0,0) {$\FRCob$};
   \node (P2) at (270:2.5) {$\calC$};
   \draw
   (P0) edge[->] node[left,xshift=-5pt] {$\KLF$} (P2)
   (P0) edge[->>] node[above] {$\partial$} (P1)
   (P1) edge[->] node[right] {$\KLFts$} (P2);
  \end{tikzpicture}
 \end{center}
 Furthermore, if there exists an integral $\lambda : \calE \to \one$ satisfying $\lambda \circ v_+ = \lambda \circ v_- = \id_{\one}$ for the ribbon and inverse ribbon elements $v_+,v_- : \one \to \calE$ of Section~\ref{S:end}, then $\KLFts : \FRCob \to \calC$ factors through a unique braided monoidal functor 
 \[
  \KLFt : \RCob \to \calC,
 \]
 meaning there exists a unique commutative diagram of the form
 \begin{center}
  \begin{tikzpicture}[descr/.style={fill=white}]
   \node (P0) at (180:2.5) {$\RHB$};
   \node (P1) at (0,0) {$\FRCob$};
   \node (P2) at (270:2.5) {$\calC$};
   \node (P3) at (0:2.5) {$\RCob$};
   \draw
   (P0) edge[->] node[left,xshift=-5pt] {$\KLF$} (P2)
   (P0) edge[->>] node[above] {$\partial$} (P1)
   (P1) edge[->] node[descr] {$\KLFts$} (P2)
   (P1) edge[->>] node[above] {$\FrgtC$} (P3)
   (P3) edge[->] node[right,xshift=5pt] {$\KLFt$} (P2);
  \end{tikzpicture}
 \end{center}
\end{theorem}

Theorem~\ref{T:main_3-dim} implies that factorizable ribbon categories are not sensitive to $4$-di\-men\-sion\-al topology beyond the signature. Interesting examples of non-factorizable unimodular ribbon categories are provided by categories of representations of non-factorizable unimodular ribbon Hopf algebras. Already for the small quantum group $u_q \fsl_2$ at a root of unity ${q = \smash{e^{\frac{2 \pi i}{r}}}}$ of order $r \equiv 0 \mod 8$, Corollary~\ref{C:CP2_S2_times_S2} and Remark~\ref{R:boundary&signature} show that our invariant can distinguish the trivial bundle $S^2 \times S^2$ from the twisted bundle $S^2 \ttimes S^2$ (once a $4$-handle has been removed from both) despite them having the same boundary, signature, and Euler characteristic.

It turns out that there is another condition we should pay attention to, apart from factorizability. Indeed, we say two 4-di\-men\-sion\-al 2-han\-dle\-bod\-ies are \textit{$2$-ex\-ot\-ic} if they are diffeomorphic but not 2-e\-quiv\-a\-lent. As we mentioned earlier, whether 2-ex\-ot\-ic pairs exist at all is still an open question, but the answer is widely believed to be positive. However, any 2-de\-for\-ma\-tion invariant of 4-di\-men\-sion\-al 2-han\-dle\-bod\-ies that is multiplicative under boundary connected sum (like $\KLF$) will be unable to detect 2-ex\-ot\-ic pairs unless it vanishes against $S^2 \times D^2$. This is due to the following instability of 2-ex\-ot\-ic phenomena: any two diffeomorphic 4-di\-men\-sion\-al 2-han\-dle\-bod\-ies become 2-e\-quiv\-a\-lent after a finite number of boundary connected sums with $S^2 \times D^2$, see Lemma~\ref{L:instability}. Luckily, there is a simple condition that ensures the vanishing of $\KLF(S^2 \times D^2)$, which can be easily stated in the case $\calC = \mods{H}$ for a unimodular ribbon Hopf algebra $H$: the scalar $\KLF(S^2 \times D^2)$ is invertible if and only if $H$ is cosemisimple (in the sense that $H^*$ is semisimple), and similarly the scalar $\KLF(S^1 \times D^3)$ is invertible if and only if $H$ is semisimple, see Corollary~\ref{C:co_semisimple}. In Section~\ref{S:explicit_formulas}, we consider a large family of unimodular ribbon Hopf algebras (including the small quantum group $u_q \fsl_2$ discussed earlier) that are neither semisimple nor cosemisimple, with many being also non-factorizable. In particular, all the corresponding invariants have the potential of detecting 2-ex\-ot\-ic pairs.

\subsection{Potential applications}

Our construction might shed light on a deep open problem arising in combinatorial group theory called the \textit{Andrews--Curtis Conjecture} \cite{AC65}. Indeed, a slight generalization\footnote{In the original Andrews--Curtis Conjecture, move $(iv)$ is not considered, and the empty presentation is replaced by a trivial presentation where all relators coincide with generators (bijectively).} of the original statement can be formulated as follows.

\begin{conjecture}[Generalized Andrews--Curtis]\label{C:AC}
 Every balanced\footnote{A balanced presentation of a group is a finite presentation featuring the same number of generators as relators.} presentation of the trivial group can be reduced to the empty presentation by a sequence of \textit{Andrews--Curtis moves}, which are operations of the following type:
 \begin{enumerate}
  \item Replacement of a relator with its conjugate by a generator;
  \item Replacement of a relator with its inverse;
  \item Replacement of a relator with its product with another relator;
  \item Creation/removal of an identical pair of generator and relator.
 \end{enumerate}
\end{conjecture}

This problem is related to the analogue of the Poincaré Conjecture for 2-e\-quiv\-a\-lence classes of 4-di\-men\-sion\-al 2-han\-dle\-bod\-ies, which can be formulated as follows.

\begin{conjecture}\label{C:4-dim_AC}
 Every 4-di\-men\-sion\-al 2-han\-dle\-body that is diffeomorphic to $D^4$ is 2-equivalent to $D^4$. 
\end{conjecture}

The relation between the two statements follows from the observation that every 4-di\-men\-sion\-al 2-han\-dle\-body is the 4-di\-men\-sion\-al thickening of a 2-di\-men\-sion\-al CW-complex, in the sense that it deformation-retracts onto the complex determined by the cores of its handles. Every such 2-di\-men\-sion\-al CW-complex naturally yields a finite presentation of its fundamental group, with generators corresponding to 1-cells and relators corresponding to 2-cells, and conversely every finitely presented group can be realized as the fundamental group of a 2-di\-men\-sion\-al CW-complex. Under this correspondence, 2-de\-for\-ma\-tions are translated to sequences of Andrews--Curtis moves\footnote{Moves $(i)$ and $(ii)$ correspond to isotopies, move $(iii)$ corresponds to a 2-handle slide, and move $(iv)$ corresponds to the creation/removal of canceling $1/2$-handle pair.}. Therefore, a proof of Conjecture~\ref{C:4-dim_AC} would immediately imply Conjecture~\ref{C:AC}. However, since 4-di\-men\-sion\-al thickenings are not unique, the other implication is not clear. This means that, as both conjectures are expected to be false, a counterexample to Conjecture~\ref{C:4-dim_AC} would not immediately disprove Conjecture~\ref{C:AC}, although it would provide strong evidence against it.

The functor $\KLF$ could be tested on potential counterexamples to Conjecture~\ref{C:4-dim_AC}, like the family of $4$-balls $\{ \Delta_n \mid n \in \N \}$ defined by
\[
 \Delta_n := \pic{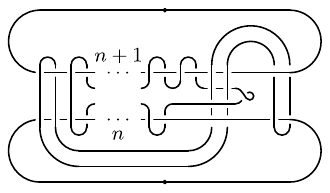}
\]
These $4$-man\-i\-folds were introduced by Gompf in \cite[Figure~4]{G91}, and they were all shown to be diffeomorphic to $D^4$. However, Gompf's diffeomorphism involves the use of a canceling $2/3$-handle pair, so it remains unknown whether these $4$-balls are also 2-e\-quiv\-a\-lent to $D^4$ or not, see \cite[Conjecture~B]{G91}. Their fundamental groups admit the following balanced presentations:
\[
 \pi_1(\Delta_n) = \langle x,y \mid xyx=yxy, x^n=y^{n+1} \rangle.
\]
These are expected to provide counterexamples to the Andrews--Curtis Conjecture for $n \geqs 3$, see \cite[Conjecture~A]{G91}. Should any of them be validated, the following famous conjecture (which appears as Problem~1.82 in Kirby's list \cite{K97}) would be disproved, thanks to \cite{GST10}.

\begin{conjecture}[Generalized Property R Conjecture]
 If surgery along a framed $n$-component link $L \subset S^3$ yields $(S^1 \times S^2)^{\# n}$, then $L$ is obtained from the $0$-framed unlink by a sequence of handle slides.
\end{conjecture}

A more elaborate family of potential counterexamples to Conjecture~\ref{C:4-dim_AC} is provided Akbulut--Kirby $4$-balls, depicted in \cite[Figure~28]{AK85} in a special case, and in \cite[Figure~36]{G91} in general. These $4$-man\-i\-folds were all shown to be diffeomorphic to $D^4$, first by Gompf in a special case \cite{G91}, and later by Akbulut in full generality \cite{A09}. However, whether they are also $2$-e\-quiv\-a\-lent to $D^4$ remains an open question.

\subsection{Overview of the construction} As suggested by our previous discussion, this paper focuses on three closely related categories, and it is convenient to describe each of them in three different ways. The first incarnation is topological:
\begin{itemize}
 \item $\RHB$ is the category of connected 4-di\-men\-sion\-al 2-han\-dle\-bod\-ies up to 2-de\-for\-ma\-tion, which was introduced by Bobtcheva and Piergallini in \cite{BP11}, see Sections~\ref{S:relative_handlebodies} and \ref{S:relative_handlebody_category}.
 \item $\FRCob$ is the category of connected framed 3-di\-men\-sion\-al cobordisms up to diffeomorphism, which was first studied by Kerler in \cite{K01}, see Sections~\ref{S:signature_defects} and \ref{S:framed_relative_cobordism_category}.
 \item $\RCob$ is the category of connected 3-di\-men\-sion\-al cobordisms up to diffeomorphism, which was first considered by Kerler in \cite{K96}, see Section~\ref{S:framed_relative_cobordism_category}.
\end{itemize}
The reason for considering $\FRCob$, alongside $\RCob$, is a recurring phenomenon in quantum topology, as many of the relevant algebraic structures giving rise to quantum constructions actually produce functors which require framings on 3-di\-men\-sion\-al cobordisms. We also stress the fact that, although $\RHB$ is referred to as the category of relative handlebody cobordisms in \cite{BP11}, it is \textit{not} actually a (relative) cobordism category, in the sense that composition in $\RHB$ is not the composition of cobordisms, see Section~\ref{S:relative_handlebody_category} for more details.

As it often happens in this context, these topological categories are conveniently described in diagrammatic terms. This leads us to consider three diagrammatic categories which are equivalent to the previous topological ones:
\begin{itemize}
 \item $\KTan$ is the category of Kirby tangles, as considered by Bobtcheva and Piergallini in \cite{BP11}, see Section~\ref{S:Kirby_tangles}.
 \item $\ETTH$ is the category of signed top tangles in handlebodies, which is introduced here for the first time, see Section~\ref{S:signed_top_tangles}.
 \item $\TTH$ is the category of top tangles in handlebodies, which is a generalization of the category defined by Habiro in \cite{H05}, see Section~\ref{S:signed_top_tangles}.
\end{itemize}

All these topological and diagrammatic categories admit explicit algebraic presentations. This is the result of a deep work initiated by Crane and Yetter \cite{CY94}, developed by Kerler \cite{K01}, Habiro \cite{H05}, and Asaeda \cite{A11}, and fully completed by Bobtcheva and Piergallini \cite{BP11}. In particular, this means we can consider three algebraically presented categories which are equivalent to the above ones:
\begin{itemize}
 \item $\Algf$ is the free braided monoidal category generated by a BPH algebra (short for Bobtcheva--Piergallini Hopf algebra), see Section~\ref{S:pre-modular_Hopf_agebras}.
 \item $\tAlgt$ is the free braided monoidal category generated by a factorizable BPH algebra, see Section~\ref{S:modular_Hopf_algebras}.
 \item $\Algt$ is the free braided monoidal category generated by an anomaly-free factorizable BPH algebra, see Section~\ref{S:modular_Hopf_algebras}.
\end{itemize} 
All these definitions are due to Bobtcheva and Piergallini, although we introduce small changes to their terminology. In short, BPH algebras in braided monoidal categories are objects equipped with several structure morphisms satisfying a list of axioms. As explained in Section~\ref{S:Kirby_top_functors}, a diagrammatic representation of these generating morphisms for the categories above is summarized by the following table:
\begin{center}
 \begin{tabular}{ c c c c c c }
  \\
  {} & $\Algf$ & $\Algt$ & {} & $\Algf$ & $\Algt$ \\
  {} & $\KTan$ & $\TTH$ & {} & $\KTan$ & $\TTH$ \\
  {} & $\RHB$ & $\RCob$ & {} & $\RHB$ & $\RCob$ \\
  \toppic{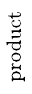} & 
  \toppic{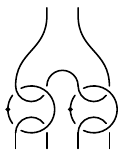} &
  \toppic{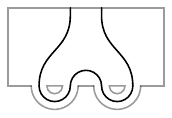} &
  \toppic{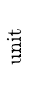} & 
  \toppic{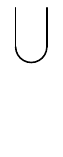} &
  \toppic{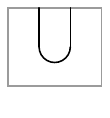}
  \\[5pt]
  \toppic{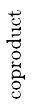} &
  \toppic{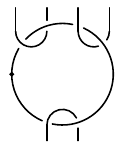} & 
  \toppic{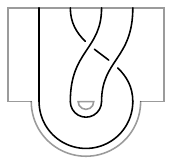} &
  \toppic{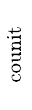} & 
  \toppic{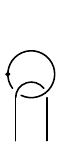} &
  \toppic{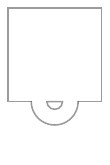}
  \\[5pt]
  \toppic{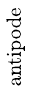} &
  \toppic{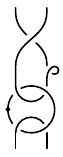} & 
  \toppic{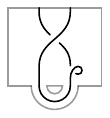} &
  \toppic{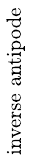} &
  \toppic{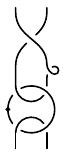} & 
  \toppic{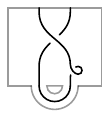}
  \\[5pt]
  \toppic{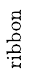} &
  \toppic{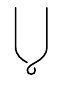} & 
  \toppic{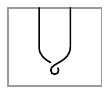} &
  \toppic{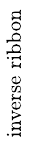} &
  \toppic{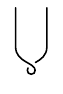} & 
  \toppic{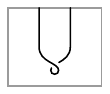}
  \\[5pt]
  \toppic{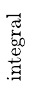} &
  \toppic{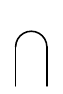} & 
  \toppic{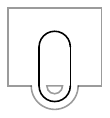} &
  \toppic{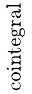} &
  \toppic{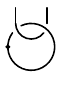} & 
  \toppic{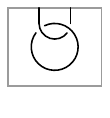} 
  \\
 \end{tabular}
\end{center}

\newpage

Furthermore, all these different categories are related by the following diagram of functors, where horizontal arrows denote projections, and vertical arrows denote equivalences:
\begin{center} 
 \begin{tikzpicture}[descr/.style={fill=white}]
  \node (P0) at (135:{2.5*sqrt(2)}) {$\Algf$};
  \node (P1) at (90:{2.5}) {$\tAlgt$};
  \node (P2) at (45:{2.5*sqrt(2)}) {$\Algt$};
  \node (P3) at (180:{2.5}) {$\KTan$};
  \node (P4) at (0,0) {$\ETTH$};
  \node (P5) at (0:{2.5}) {$\TTH$};
  \node (P6) at (225:{2.5*sqrt(2)}) {$\RHB$};
  \node (P7) at (270:{2.5}) {$\FRCob$};
  \node (P8) at (315:{2.5*sqrt(2)}) {$\RCob$};
  \draw
  (P0) edge[->>] node[above] {$\tQtnt$} (P1)
  (P1) edge[->>] node[above] {$\Qtnt$} (P2)
  (P0) edge[->] node[left] {$K$} node[right] {\rotatebox{270}{$\cong$}} (P3)
  (P1) edge[->] node[left] {$T^\sigma$} node[right] {\rotatebox{270}{$\cong$}} (P4)
  (P2) edge[->] node[left] {$T$} node[right] {\rotatebox{270}{$\cong$}} (P5)  
  (P3) edge[->>] node[above] {$E$} (P4)
  (P4) edge[->>] node[above] {$\FrgtT$} (P5)
  (P3) edge[->] node[left] {$D$} node[right] {\rotatebox{270}{$\cong$}} (P6)
  (P4) edge[->] node[left] {$\chi^\sigma$} node[right] {\rotatebox{270}{$\cong$}} (P7)
  (P5) edge[->] node[left] {$\chi$} node[right] {\rotatebox{270}{$\cong$}} (P8)
  (P6) edge[->>] node[above] {$\partial$} (P7)
  (P7) edge[->>] node[above] {$\FrgtC$} (P8);
 \end{tikzpicture}
\end{center}
The missing definitions can be found in Sections~\ref{S:boundary_forgetful_functors}, \ref{S:handle_surgery_functors}, and \ref{S:Kirby_top_functors}.

The Kerler--Lyubashenko functor $\KLF : \RHB \to \calC$ appearing in Theorem~\ref{T:main_4-dim} is based on the explicit algebraic presentation of $\RHB$ obtained by Bobtcheva and Piergallini. The strategy is essentially to prove that the end
\[
 \calE = \int_{X \in \calC} X \otimes X^*
\]
of a unimodular ribbon category $\calC$ supports the structure of a BPH algebra in $\calC$, and can therefore be chosen as the image of the generating object of $\Algf$. Furthermore, increasingly strong additional assumptions on $\calC$ ensure that $\KLF : \RHB \to \calC$ factors through successive quotients, as proved in Section~\ref{S:proofs}.

We end this introduction with a couple of comments. First, we restate Bobtcheva and Piergallini's deep result in terms of braided monoidal functors of the form $V : \FRCob \to \calC$ for a braided monoidal category $\calC$, which we refer to as \textit{$\calC$-valued braided TFTs} (short for \textit{Topological Field Theories}). Let us highlight a couple of key differences with respect to Atiyah's well-known axioms: 
\begin{enumerate}
 \item The monoidal structure on $\FRCob$ is induced by boundary connected sum, instead of disjoint union.
 \item The braiding on $\FRCob$ is induced by permutation of 1-han\-dles, and therefore is not symmetric.
\end{enumerate}
We say a $\calC$-valued braided TFT is \textit{anomaly-free} if it factors through the forgetful functor $\FrgtC : \FRCob \to \RCob$.

\begin{theorem}[Bobtcheva--Piergallini]
 If $\calC$ is a braided monoidal category, then (a\-nom\-a\-ly-free) $\calC$-valued braided TFTs are classified by (a\-nom\-a\-ly-free) factorizable BPH algebras in $\calC$.
\end{theorem}

This statement is a further step towards the algebraization of 3-di\-men\-sion\-al TFTs, after the complete classification of 2-di\-men\-sion\-al TFTs in terms of commutative Frobenius algebras.

Finally, let us briefly mention how our work relates to Reutter's proof that semisimple $4$-di\-men\-sion\-al TFTs cannot distinguish homeomorphic closed $4$-man\-i\-folds and homotopy equivalent simply connected ones \cite{R20}. The reason why our construction is not automatically subject to this kind of constraints is that we step outside of the framework in which Reutter draws his conclusions in several ways: indeed, we use $4$-dimensional $2$-handlebodies instead of arbitrary $4$-manifolds, we consider $2$-deformations instead of diffeormorphisms, and in addition to our previous remark concerning braided monoidal structures, our functors do not seem to satisfy any obvious reformulation of Reutter's semisimplicity assumption, since the (braided) Hopf algebra associated with the solid torus is clearly not semisimple.

\subsection*{Acknowledgments} We would like to express our gratitude to Maria Stamatova for bringing to our attention the work of Bobtcheva and Piergallini, without which the paper would not have been written. Our work was supported by the National Center of Competence in Research (NCCR) SwissMAP, and by Grant n.~200020\_207374 of the Swiss National Science Foundation (SNSF).

\section{Topological preliminaries}

In this section we recall some basic facts about handle decompositions and signature defects for the convenience of the reader.

\subsection*{Conventions}

All manifolds considered in this paper are compact, smooth, and oriented, that is, orientable and equipped with a fixed orientation, and all diffeomorphisms are orientation-preserving. The interval $[0,1]$ will be denoted $I$, the unit $n$-disc will be denoted $D^n$, and the unit $n$-sphere will be denoted $S^n$. All corners that appear as a result of taking products between manifolds with boundary and of gluing manifolds along submanifolds of their boundaries can be smoothed canonically, up to diffeomorphism, and we will tacitly do so without comment. Furthermore, we fix throughout the paper an algebraically closed field $\Bbbk$, and the term linear will always stand for $\Bbbk$-linear, unless stated otherwise.

\subsection{Handlebodies and deformations}\label{S:relative_handlebodies}

Let $n$ and $k \leqs n$ be natural numbers. Recall that an \textit{$n$-di\-men\-sion\-al $k$-han\-dle} is a copy $D_{(k,n)}$ of the $n$-manifold $D^k \times D^{n-k}$. Its \textit{core} is the $k$-sub\-man\-i\-fold $D^k \times \{ 0 \}$, while its \textit{cocore} is the $(n-k)$-sub\-man\-i\-fold $\{ 0 \} \times D^{n-k}$. It is often useful to think about $D_{(k,n)}$ as the $n$-di\-men\-sion\-al thickening of a $k$-cell. Indeed, the boundary $\partial D_{(k,n)}$ can be decomposed into the \textit{attaching tube} $A_{(k,n)} := (\partial D^k) \times D^{n-k}$ and the \textit{belt tube} $B_{(k,n)} := D^k \times (\partial D^{n-k})$. These provide tubular neighborhoods, inside $\partial D^{(k,n)}$, for the \textit{attaching sphere} $(\partial D^k) \times \{ 0 \}$ and the \textit{belt sphere} $\{ 0 \} \times (\partial D^{n-k})$. An $n$-di\-men\-sion\-al $k$-han\-dle $D_{(k,n)}$ can be attached to the boundary of an $n$-dimensional manifold $X$ using an embedding $\varphi : A_{(k,n)} \hookrightarrow \partial X$. See \cite[Figure~4.1]{GS99} for an example of the attachment of a copy of $D_{(1,2)}$.

If $Y$ is an $(n-1)$-man\-i\-fold, an \textit{$n$-di\-men\-sion\-al $k$-han\-dle\-bod\-y built on $Y$}, sometimes simply called a \textit{handlebody}\footnote{In \cite{BP11} the term \textit{relative} handlebody is used, but since all handlebodies considered here are of this type, we can adopt a simplified terminology without the risk of confusion.},  is an $n$-man\-i\-fold $X$ equipped with a filtration
\[
 Y \times I = X^{-1} \subseteq X^0 \subseteq \ldots \subseteq X^{k-1} \subseteq X^k = X
\]
of $n$-man\-i\-folds such that, if the \textit{incoming boundary} of $X^i$ is defined to be
\[
 \partial_- X^i := Y \times \{ 0 \}
\]
and the \textit{outgoing boundary} of $X^i$ is defined to be
\[
 \partial_+ X^i := (\partial X^i) \smallsetminus \left( (Y \times \{ 0 \}) \cup ((\partial Y) \times \ring{I}) \right)
\]
for every $-1 \leqs i \leqs k$, then $X^i$ is obtained from $X^{i-1}$ by attaching a finite number of $n$-di\-men\-sion\-al $i$-han\-dles to $\partial_+ X^{i-1}$, see for instance \cite[Definition~1.2.2]{BP11} for a precise definition. Similarly, a \textit{dual} $n$-di\-men\-sion\-al $k$-han\-dle\-bod\-y built on $Y$ is an $n$-man\-i\-fold $X$ equipped with a filtration
\[
 Y \times I = X^{-1} \subseteq X^0 \subseteq \ldots \subseteq X^{k-1} \subseteq X^k = X
\]
of $n$-man\-i\-folds such that, if the \textit{outgoing boundary} of $X^i$ is defined to be
\[
 \partial_+ X^i := Y \times \{ 1 \}
\]
and the \textit{incoming boundary} of $X^i$ is defined to be
\[
 \partial_- X^i := (\partial X^i) \smallsetminus \left( (Y \times \{ 1 \}) \cup ((\partial Y) \times \ring{I}) \right)
\]
for every $-1 \leqs i \leqs k$, then $X^i$ is obtained from $X^{i-1}$ by attaching a finite number of $n$-di\-men\-sion\-al $i$-han\-dles to $\partial_- X^{i-1}$. Most of the time, we will abusively denote handlebodies and dual handlebodies simply by their underlying manifold, without mention to the accompanying filtration.

Now let $\ell \leqs n$ be a natural number. An \textit{$\ell$-de\-for\-ma\-tion} of a $n$-di\-men\-sion\-al $k$-han\-dle\-bod\-y $X$ is a finite sequence of moves of the following type:
\begin{enumerate}
 \item isotopy of the attaching map of an $i$-han\-dle inside $\partial_+ X^{i-1}$ for $0 \leqs i \leqs \ell$;
 \item handle slide of an $i$-han\-dle over another $i$-han\-dle for $0 \leqs i \leqs \ell$;
 \item creation/removal of a canceling pair of $(i-1)/i$-han\-dles for $1 \leqs i \leqs \ell$.
\end{enumerate}
Two $n$-di\-men\-sion\-al $k$-handlebodies built on the same $(n-1)$-di\-men\-sion\-al manifold are said to be \textit{$\ell$-e\-quiv\-a\-lent} if they are related by an $\ell$-de\-for\-ma\-tion. In particular, if two $n$-di\-men\-sion\-al $k$-han\-dle\-bod\-ies are $k$-e\-quiv\-a\-lent, it means they can be deformed into one another without ever stepping outside of the class of $n$-di\-men\-sion\-al $k$-han\-dle\-bod\-ies. Two $n$-di\-men\-sion\-al $(n-1)$-han\-dle\-bod\-ies are diffeomorphic if and only if they are $(n-1)$-e\-quiv\-a\-lent, see for instance \cite[Section~1.2]{BP11}. It remains an open question to understand whether every pair of diffeomorphic $n$-di\-men\-sion\-al $(n-1)$-han\-dle\-bod\-ies are also $(n-2)$-e\-quiv\-a\-lent. Up to 1-de\-for\-ma\-tion, every connected handlebody with connected incoming boundary can be assumed to have no additional 0-han\-dle, see for instance \cite[Proposition~1.2.4]{BP11}.

\subsection{Cobordisms and signature defects}\label{S:signature_defects}

If $Z$ is a closed $(n-1)$-man\-i\-fold, an \textit{$n$-man\-i\-fold with $Z$-pa\-ram\-e\-trized boundary}, sometimes simply called a \textit{manifold with parametrized boundary}, is a pair $(Y,\varphi)$, where $Y$ is an $n$-man\-i\-fold and $\varphi : Z \to \partial Y$ is a diffeomorphism. A \textit{relative cobordism between $n$-man\-i\-folds with $Z$-pa\-ram\-e\-trized boundary} $(Y,\varphi)$ and $(Y',\varphi')$, sometimes simply called a \textit{cobordism}, is an $(n+1)$-man\-i\-fold with $(Y,Y')$-pa\-ram\-e\-trized boundary $(X,\psi)$, where 
\[
 (Y,Y') := (-Y) \cup_{(\varphi \times \{ 0 \})} (Z \times I) \cup_{(\varphi' \times \{ 1 \})} Y',
\]
where $-Y$ denotes $Y$ with the opposite orientation. Most of the time, we will abusively denote manifolds with parametrized boundary and cobordisms by their underlying manifold, without explicitly specifying boundary identifications. Notice that the notion of relative cobordism between man\-i\-folds with $Z$-pa\-ram\-e\-trized boundary given here recovers the standard notion of cobordism when $Z = \varnothing$.

If $n = 4k+2$, then a \textit{Lagrangian subspace} $\calL \subset H_{2k+1}(Y;\R)$ for the $(2k+1)$th homology of an $n$-man\-i\-fold $Y$ is a maximal isotropic subspace with respect to the transverse intersection pairing $\pitchfork_Y : H_{2k+1}(Y;\R) \times H_{2k+1}(Y;\R) \to \R$, which is an antisymmetric bilinear form on $H_{2k+1}(Y;\R)$, and which is non-de\-gen\-er\-ate if and only if $\partial Y = \varnothing$. Every relative cobordism $(X,\psi)$ between $n$-man\-i\-folds with $Z$-pa\-ram\-e\-trized boundary $(Y,\varphi)$ and $(Y',\varphi')$ can be used to push forward to the $(2k+1)$th homology of $Y'$ any Lagrangian subspace $\calL \subset H_{2k+1}(Y;\R)$, by setting
\[
 X_*(\calL) := \{ y' \in H_{2k+1}(Y';\R) \mid i_*(\psi_*(y')) \in i_*(\psi_* (\calL)) \},
\]
where $i : \partial X \hookrightarrow X$ denotes inclusion, and similarly to pull back to the $(2k+1)$th homology of $Y$ any Lagrangian subspace $\calL' \subset H_{2k+1}(Y';\R)$, by setting
\[
 X^*(\calL') := \{ y \in H_{2k+1}(Y;\R) \mid i_*(\psi_*(y)) \in i_*(\psi_* (\calL')) \}.
\]
If $\calL,\calL',\calL'' \in H_{2k+1}(Y;\R)$ are Lagrangian subspaces, then their \textit{Maslov index}
\[
 \mu(\calL,\calL',\calL'') \in \Z
\]
is an integer which is defined in \cite[Section~IV.3.5]{T94} (see \cite[Appendix~C.3]{D17} for the version that is needed here whenever $\partial Y \neq \varnothing$), and which is completely antisymmetric in all its entries.

If $n = 4k$, then the \textit{signature} $\sigma(W) \in \Z$ of an $n$-man\-i\-fold $W$ is the signature of the transverse intersection pairing $\pitchfork_W : H_{2k}(W;\R) \times H_{2k}(W;\R) \to \R$, which is a symmetric bilinear form on $H_{2k}(W;\R)$. The famous non-additivity theorem of Wall \cite{W69} states that whenever $W = W_- \cup_{M_0} W_+$ for some $n$-di\-men\-sion\-al submanifolds $W_-$ and $W_+$ and some $(n-1)$-di\-men\-sion\-al submanifold $M_0$ satisfying $\partial W_- = (-M_-) \cup_\varSigma M_0$ and $\partial W_+ = (-M_0) \cup_\varSigma M_+$, where $M_-$ and $M_+$ are $(n-1)$-di\-men\-sion\-al submanifolds and $\varSigma$ is an $(n-2)$-di\-men\-sion\-al submanifold satisfying $\partial W = (-M_-) \cup M_+$ and $(-M_-) \cap M_+ = \partial M_- = \partial M_0 = \partial M_+ = \varSigma$, then
\[
 \sigma(W) = \sigma(W_-) + \sigma(W_+) + \mu(\calL_-,\calL_0,\calL_+),
\]
where ${\calL_0 := \ker(i_{\varSigma_0 *})}$ and ${\calL_{\pm} := \ker(i_{\varSigma_{\pm} *})}$ for the homomorphisms
\[
 i_{\varSigma_0 *} : H_{2k-1}(\varSigma) \rightarrow H_{2k-1}(M_0), \quad
 i_{\varSigma_{\pm} *} : H_{2k-1}(\varSigma) \rightarrow H_{2k-1}(M_{\pm}).
\]
induced by inclusion.

\section{Handlebody and cobordism categories}\label{S:relative_handlebodies_and_cobordisms}

In this section, we recall the definition of the three topological categories we will focus on: Bobtcheva and Piergallini's category $\RHB$ of connected 4-di\-men\-sion\-al 2-han\-dle\-bod\-ies, and Kerler's categories $\FRCob$ and $\RCob$ of connected (framed)  3-di\-men\-sion\-al cobordisms.

\subsection{4-dimensional 2-handlebody category}\label{S:relative_handlebody_category}

For every $m \in \N$, we specify a standard connected 3-di\-men\-sion\-al 1-han\-dle\-bod\-y $H_m \subset \R^3$ built on the square $I^2$ and featuring precisely $m$ 1-handles, which we represent graphically through the projection to $\R \times \{ 0 \} \times \R$ as
\begin{align*}
 \\*[-7.5pt]
 \pic{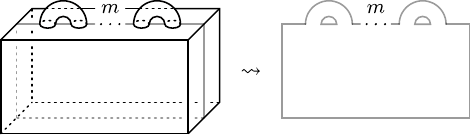} \\*[-7.5pt]
\end{align*}
and we denote by $\varSigma_{m,1}$ the surface $\partial_+ H_m$. Similarly, we denote by $H_m^* \subset \R^3$ the standard connected dual handlebody obtained from $H_m$ by reflection with respect the plane $\R^2 \times \{ \frac 12 \}$, which we also represent graphically as
\begin{align*}
 \\*[-7.5pt]
 \pic{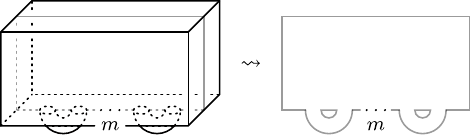} \\*[-7.5pt]
\end{align*}
and we denote by $\varSigma_{m,1}^*$ the surface $\partial_- H_m^*$.

The \textit{category $\RHB$ of connected $4$-di\-men\-sion\-al $2$-han\-dle\-bod\-ies}, first introduced in \cite{BP11} under the notation $\Chb^{3+1}_1$, is defined as follows:
\begin{itemize}
 \item Objects of $\RHB$ are natural numbers $m \in \N$.
 \item Morphisms of $\RHB$ from $m$ to $m'$ are 2-de\-for\-ma\-tion classes of connected 4-di\-men\-sion\-al 2-han\-dle\-bod\-ies $W$ built on the 3-manifold
  \[
   X(H_m,H_{m'}) := H_m^* \cup_{(I^2 \times \{ 0 \})} I^3 \cup_{(I^2 \times \{ 1 \})} H_{m'}
  \]
  obtained from the standard connected handlebodies introduced earlier, which we represent schematically as
  \[
   \pic{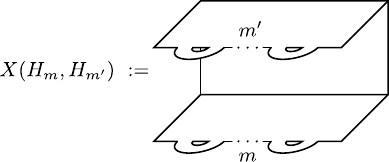}
  \]
  The $-1$st manifold in the filtration of $W$ is represented schematically as
  \[
   \pic{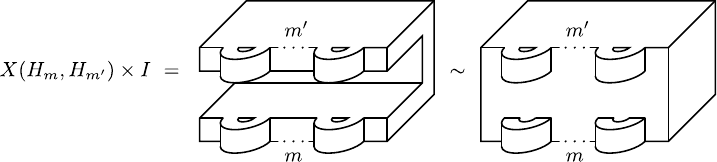}
  \]
  As observed earlier, morphisms of $\RHB$ are diffeomorphic if and only if they are 3-e\-quiv\-a\-lent, but whether the same holds for 2-e\-quiv\-a\-lence or not is still an open question.
 \item The identity $\id_m \in \RHB(m,m)$ of an object $m \in \RHB$ is the 2-de\-for\-ma\-tion class of the handlebody $H_m \times I$, which can be built on the 3-manifold $X(H_m,H_m)$ by attaching 2-han\-dles like
 \[
  \pic{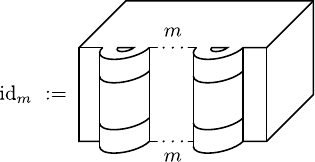}
 \]
 \item The composition 
  \[
   W' \circ W \in \RHB(m,m'')
  \]
  of morphisms $W \in \RHB(m,m')$, $W' \in \RHB(m',m'')$ is given by 
  \[
   W \cup_{H_{m'}} W'.
  \]
\end{itemize}
The category $\RHB$ can be given the structure of a monoidal category:
\begin{itemize}
 \item The tensor product 
  \[
   m \otimes m' \in \RHB
  \]
  of objects $m,m' \in \RHB$ is given by the sum $m + m'$. 
 \item The tensor product 
  \[
   W \otimes W' \in \RHB(m \otimes m',m'' \otimes m''')
  \]
  of morphisms $W \in \RHB(m,m'')$, $W' \in \RHB(m',m''')$ is given by 
  \[
   W \bcs W',
  \]
  where $W \bcs W'$ is obtained by gluing horizontally $W$ to $W'$, identifying the right side of $W$ with the left side of $W'$ as prescribed by the identity map, and then by shrinking the result in order to get a handlebody built on $X(H_{m + m'},H_{m'' + m'''})$.
\end{itemize}

We stress the fact that, although $\RHB$ is called the category of relative handlebody cobordisms in \cite{BP11} and in \cite{B20}, it is not a (relative) cobordism category. Indeed, despite the fact that every handlebody naturally admits the structure of a cobordism, the composition of 4-di\-men\-sion\-al 2-han\-dle\-bod\-ies in $\RHB$ is not the composition of the corresponding cobordisms structures. Furthermore, 4-di\-men\-sion\-al 2-han\-dle\-bod\-ies in $\RHB$ are considered only up to 2-de\-for\-ma\-tion, which is expected to determine a different equivalence relation than the one determined by diffeomorphism.

\subsection{Framed 3-dimensional cobordism category}\label{S:framed_relative_cobordism_category}

For each standard connected surface $\varSigma_{m,1}$ introduced earlier, we specify a standard Lagrangian subspace $\calL_m \subset H_1(\varSigma_{m,1};\R)$ defined as the subspace generated by $\{ \beta_1,\ldots,\beta_m,\partial \}$, where $\beta_i$ denotes the homology class of the belt sphere of the $i$th 1-han\-dle of $\varSigma_{m,1}$ for every integer $1 \leqs i \leqs m$, and $\partial$ denotes the homology class of the boundary of $\varSigma_{m,1}$, as represented here below
\[
 \pic{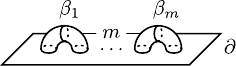}
\]

The \textit{category $\FRCob$ of connected framed $3$-di\-men\-sion\-al cobordisms}, first considered in \cite{K01} under the notation $\bfCob$, is defined as follows:
\begin{itemize}
 \item Objects of $\FRCob$ are natural numbers $m \in \N$.
 \item Morphisms of $\FRCob$ from $m$ to $m'$ are pairs $(M,\sig)$, where $M$ is the diffeomorphism class of a connected 3-di\-men\-sion\-al cobordism $M$ from $\varSigma_{m,1}$ to $\varSigma_{m',1}$ for the standard connected surfaces introduced earlier, and $\sig \in \Z$ is an integer, called the \textit{signature defect}.
 \item The identity $\id_m \in \FRCob(m,m)$ of an object $m \in \FRCob$ is the pair 
  \[ 
   (\varSigma_{m,1} \times I,0).
  \]
 \item The composition 
  \[
   (M',\sig') \circ (M,\sig) \in \FRCob(m,m'')
  \]
  of morphisms $(M,\sig) \in \FRCob(m,m')$, $(M',\sig') \in \FRCob(m',m'')$ is given by
  \[
   (M \cup_{\varSigma_{m',1}} M', \sig + \sig' - \mu(M_*(\calL_m),\calL_{m'},(M')^*(\calL_{m''}))).
  \]
\end{itemize}
The category $\FRCob$ can be given the structure of a monoidal category:
\begin{itemize}
 \item The tensor product 
  \[
   m \otimes m' \in \FRCob
  \]
  of objects $m,m' \in \FRCob$ is given by the sum $m + m'$. 
 \item The tensor product 
  \[
   (M,\sig) \otimes (M',\sig') \in \FRCob(m \otimes m',m'' \otimes m''')
  \]
  of morphisms $(M,\sig) \in \FRCob(m,m'')$, $(M',\sig') \in \FRCob(m',m''')$ is given by
  \[
   (M \bcs M', \sig+\sig'),
  \]
  where $M \bcs M'$ is obtained by gluing horizontally $M$ to $M'$, identifying the right side of $M$ with the left side of $M'$ as prescribed by the identity map, and then by shrinking the result in order to get a cobordism from $\varSigma_{m + m',1}$ to $\varSigma_{m'' + m''',1}$.
\end{itemize}

Finally, the \textit{category $\RCob$ of connected $3$-di\-men\-sion\-al cobordisms}, first appearing in \cite{K96} under the notation $\Cob_3(1)^\conn$, and in \cite{K01} under the notation $\bfCob_0$, is simply obtained from $\FRCob$ by forgetting signature defects.

\subsection{Boundary and forgetful functors}\label{S:boundary_forgetful_functors} There is a \textit{boundary} functor 
\[
 \partial : \RHB \to \FRCob
\]
that is given by the identity on objects, and that sends every 4-di\-men\-sion\-al 2-han\-dle\-bod\-y $W$ in $\RHB(m,m')$ to the pair $(\partial_+ W, \sigma(W))$ in $\FRCob(m,m')$, where $\partial_+ W$ is the 3-di\-men\-sion\-al cobordism determined by the outgoing boundary of $W$, and $\sigma(W)$ is the signature of $W$. Remark that 2-de\-for\-ma\-tions are diffeomorphisms of a special kind, so clearly 2-e\-quiv\-a\-lent handlebodies have diffeomorphic boundaries. The boundary functor $\partial$ is surjective on objects and full, which means it gives a projection from $\RHB$ to $\FRCob$, see \cite[Section~5.1]{BP11}.

We also have a \textit{forgetful} functor 
\[
 \FrgtC : \FRCob \to \RCob
\]
that is given by the identity on objects, and that sends every pair $(M,\sig)$ in $\FRCob(m,m')$ to the 3-di\-men\-sion\-al cobordism $M$. The forgetful functor $\FrgtC$ is clearly surjective on objects and full, which means it gives a projection from $\FRCob$ to $\RCob$.

\section{Kirby and signed top tangle categories}\label{S:kirby_top_tangles}

The aim of this section is to give a diagrammatic description of the three topological categories of Section~\ref{S:relative_handlebodies_and_cobordisms} by first defining the categories of Kirby tangles $\KTan$ and of (signed) top tangles in handlebodies $\ETTH$ and $\TTH$, then defining functors between them, and finally showing that the following diagram commutes:
\begin{center} 
 \begin{tikzpicture}[descr/.style={fill=white}]
  \node (P3) at (180:{2.5}) {$\KTan$};
  \node (P4) at (0,0) {$\ETTH$};
  \node (P5) at (0:{2.5}) {$\TTH$};
  \node (P6) at (225:{2.5*sqrt(2)}) {$\RHB$};
  \node (P7) at (270:{2.5}) {$\FRCob$};
  \node (P8) at (315:{2.5*sqrt(2)}) {$\RCob$};
  \draw
  (P3) edge[->>] node[above] {$E$} (P4)
  (P4) edge[->>] node[above] {$\FrgtT$} (P5)
  (P3) edge[->] node[left] {$D$} node[right] {\rotatebox{270}{$\cong$}} (P6)
  (P4) edge[->] node[left] {$\chi^\sigma$} node[right] {\rotatebox{270}{$\cong$}} (P7)
  (P5) edge[->] node[left] {$\chi$} node[right] {\rotatebox{270}{$\cong$}} (P8)
  (P6) edge[->>] node[above] {$\partial$} (P7)
  (P7) edge[->>] node[above] {$\FrgtC$} (P8);
 \end{tikzpicture}
\end{center}

\subsection{Kirby tangles}\label{S:Kirby_tangles} First, we recall the definition of the category of Kirby tangles $\KTan$ considered in \cite[Section~2]{BP11}. For all $m,m' \in \N$, a \textit{Kirby $(m,m')$-tan\-gle} $T$, sometimes simply called a \textit{Kirby tangle}, is a framed unoriented tangle in $I^3$ whose set of boundary points is composed of $2m$ points uniformly distributed on the bottom line $I \times \{ \frac 12 \} \times \{ 0 \} \subset I^3$ and $2m'$ points uniformly distributed on the top line $I \times \{ \frac 12 \} \times \{ 1 \} \subset I^3$, and whose components are of two types:
\begin{itemize}
 \item Dotted $0$-framed unknots;
 \item Undotted blackboard-framed components transversal to Seifert discs of dotted unknots, and satisfying the condition that, for all $0 \leqs i \leqs m$ and $0 \leqs i' \leqs m'$, there exists an undotted component joining the $(2i)$th and the $(2i-1)$th incoming boundary points, and another one joining the $(2i')$th and the $(2i'-1)$th outgoing ones.
\end{itemize}
Kirby tangles should be understood as describing 4-di\-men\-sion\-al 2-han\-dle\-bod\-ies. Dotted unknots represent 1-han\-dles, whose attaching tubes, provided by disjoint pairs of 3-balls, are thought to be squeezed onto the corresponding Seifert discs, while undotted knots represent attaching spheres of 2-han\-dles. A \textit{$2$-de\-for\-ma\-tion} of a Kirby tangle $T$ is a finite sequence of isotopies and moves of the following type:
\begin{align*}
 \pic{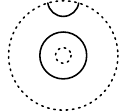} &\leftrightsquigarrow \pic{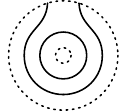} \tag{D$1$}\label{E:rel_D1} \\*
 \pic{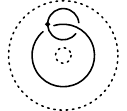} &\leftrightsquigarrow \pic{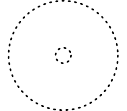} \tag{D$2$}\label{E:rel_D2}
\end{align*}
Move~\eqref{E:rel_D1} corresponds to the slide of a $2$-han\-dle over another $2$-han\-dle, while move~\eqref{E:rel_D2} to the creation/removal of a canceling pair of $1/2$-handles. These operations are performed inside a solid torus $S^1 \times D^2$ embedded into $I^3$, and they leave Kirby tangles unchanged in the complement. Notice that the slide of a $1/2$-han\-dle \textit{under} a $1$-han\-dle can be deduced from these moves, see \cite[Figure~2.2.9]{BP11}.

The \textit{category $\KTan$ of Kirby tangles} is defined as follows:
\begin{itemize}
 \item Objects of $\KTan$ are natural numbers $m \in \N$.
 \item Morphisms of $\KTan$ from $m$ to $m'$ are Kirby $(m,m')$-tan\-gles $T$ modulo $2$-de\-for\-ma\-tions.
 \item The identity $\id_m \in \KTan(m,m)$ of an object $m \in \KTan$ is given by\footnote{We could generalize the definition of Kirby tangles by allowing \textit{through strands}, like in \cite[Section~3.1]{K01}, and by adding \textit{boundary moves} to the set of relations, like in \cite[Equation~(9)]{K01}. In this case, the identity of an object $m \in \KTan$ could be equivalently represented by the standard identity tangle with $2m$ strands.}
  \[
   \id_m := \pic{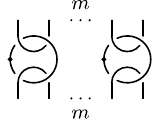}
  \]
 \item The composition 
  \[
   T' \circ T \in \KTan(m,m'')
  \]
  of morphisms $T \in \KTan(m,m')$, $T' \in \KTan(m',m'')$ is given by vertical gluing.
\end{itemize}
The category $\KTan$ can be given the structure of a braided monoidal category:
\begin{itemize}
 \item The tensor product 
  \[
   m \otimes m' \in \KTan
  \]
  of objects $m,m' \in \KTan$ is given by the sum $m + m'$. 
 \item The tensor product 
  \[
   T \otimes T' \in \KTan(m \otimes m',m'' \otimes m''')
  \]
  of morphisms $T \in \KTan(m,m'')$, $T' \in \KTan(m',m''')$ is given by 
  \[
   T \disjun T',
  \]
  where $T \disjun T'$ is obtained by gluing horizontally the two copies of $I^3$, identifying the right side of the left one with the left side of the right one as prescribed by the identity map, and then by shrinking the result into $I^3$.
 \item The braiding $c_{m,m'} \in \KTan(m \otimes m',m' \otimes m)$ of objects $m,m' \in \KTan$ is given by
  \[
   c_{m,m'} := \pic[-12.5]{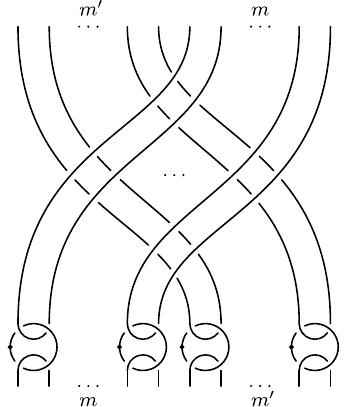}
  \]
\end{itemize}
Note that, up to 2-de\-for\-ma\-tion, we can always pre-compose any Kirby $(m,m')$-tan\-gle $T$ with $\id_m$ and post-compose it with $\id_{m'}$.

\subsection{Signed top tangles in handlebodies}\label{S:signed_top_tangles} In this section, we define a category $\ETTH$, which we call the \textit{category of signed top tangles in handlebodies}. This category is related to the category of bottom tangles in handlebodies defined by Habiro in \cite[Section~14.4]{H05}, which can be realized as a subquotient of $\ETTH$. The change in terminology is a consequence of a change in conventions. Indeed Habiro, like Kerler \cite{K01} and Asaeda \cite{A11}, reads morphisms from top to bottom, while we read them from bottom to top, like Bobtcheva and Piergallini \cite{BP11}. This will exchange the role of bottom tangles and top tangles. For all $m,m' \in \N$, a \textit{top $(m,m')$-tangle} $T$, sometimes simply called a \textit{top tangle}, is an unoriented framed tangle inside the standard connected dual 3-di\-men\-sion\-al 1-han\-dle\-bod\-y $H_m^*$ introduced in Section~\ref{S:relative_handlebodies_and_cobordisms}. Its set of boundary points is composed of $2m'$ points uniformly distributed on the top line $I \times \{ \frac 12 \} \times \{ 1 \} \subset H_m^*$. For every $0 \leqs j \leqs m$, a component of $T$ joins the $(2j)$th and the $(2j-1)$th boundary points. If $T$ is a top tangle and $\sig \in \Z$ is an integer, then a \textit{signed Kirby move} is a modification of the pair $(T,\sig)$ of the following type:
\begin{align*}
 \left( \pic{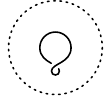},\sig \right) &\leftrightsquigarrow \left( \pic{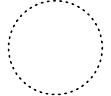},\sig-1 \right) \tag{K$-1$}\label{E:rel_K-1} \\*
 \left( \pic{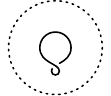},\sig \right) &\leftrightsquigarrow \left( \pic{rel_K1b},\sig+1 \right) \tag{K$+1$}\label{E:rel_K+1} \\
 \left( \pic{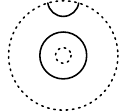},\sig \right) &\leftrightsquigarrow \left( \pic{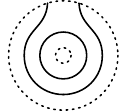},\sig \right) \tag{K$2$}\label{E:rel_K2} \\*
 \left( \pic{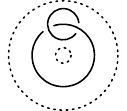},\sig \right) &\leftrightsquigarrow \left( \pic{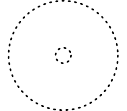},\sig \right) \tag{K$3$}\label{E:rel_K3}
\end{align*}
Moves~\eqref{E:rel_K-1} and \eqref{E:rel_K+1} correspond to (de)stabilizations by $\bbC P^2$ and $\overline{\bbC P^2}$ respectively. These operations are performed inside a ball $D^3$ embedded into $H_m^*$, while the remaining ones are performed inside a solid torus $S^1 \times D^2$ embedded into $H_m^*$.

The \textit{category $\ETTH$ of signed top tangles in handlebodies} is the category defined as follows: 
\begin{itemize}
 \item Objects of $\ETTH$ are natural numbers $m \in \N$.
 \item Morphisms of $\ETTH$ from $m \in \N$ to $m' \in \N$ are pairs $(T,\sig)$, where $T$ is a top $(m,m')$-tangle and $\sig \in \Z$ is an integer called the \textit{signature defect}, modulo isotopies and signed Kirby moves.
 \item The identity $\id_m \in \ETTH(m,m)$ of an object $m \in \ETTH$ is given by
  \[
   \id_m := \left( \pic{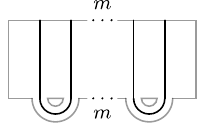}, 0 \right).
  \]
 \item The composition 
  \[
   (T',\sig') \circ (T,\sig) \in \ETTH(m,m'')
  \]
  of morphisms $T \in \ETTH(m,m')$, $T' \in \ETTH(m',m'')$ is given by 
  \[
   (T' \circ T,\sig+\sig'),
  \]
  where $T' \circ T$ is obtained by considering an open tubular neighborhood $N(\tilde{T})$ in $H_m^*$ of the subtangle $\tilde{T}$ of $T$ composed of all arc components, by gluing vertically its complement $H_m^* \smallsetminus N(\tilde{T})$ to $H_{m'}^*$, identifying the top base of $H_m^* \smallsetminus N(\tilde{T})$ with the bottom base of $H_{m'}^*$ as prescribed by the top tangle $\tilde{T}$, and then by shrinking the result into $H_m^*$.
\end{itemize}
The category $\ETTH$ can be given the structure of a braided monoidal category:
\begin{itemize}
 \item The tensor product 
  \[
   m \otimes m' \in \ETTH
  \]
  of objects $m,m' \in \ETTH$ is given by the sum $m + m'$. 
 \item The tensor product 
  \[
   (T,\sig) \otimes (T',\sig') \in \ETTH(m \otimes m',m'' \otimes m''')
  \]
  of morphisms $(T,\sig) \in \ETTH(m,m'')$, $(T',\sig') \in \ETTH(m',m''')$ is given by
  \[
   (T \disjun T',\sig+\sig'),
  \]
  where $T \disjun T'$ is obtained by gluing horizontally $H_m^*$ to $H_{m'}^*$, identifying the right side of $H_m^*$ with the left side of $H_{m'}^*$ as prescribed by the identity map, and then by shrinking the result into $H_{m+m'}^*$.
 \item The braiding $c_{m,m'} \in \ETTH(m \otimes m',m' \otimes m)$ of objects $m,m' \in \ETTH$ is given by
  \[
   c_{m,m'} := \left( \pic{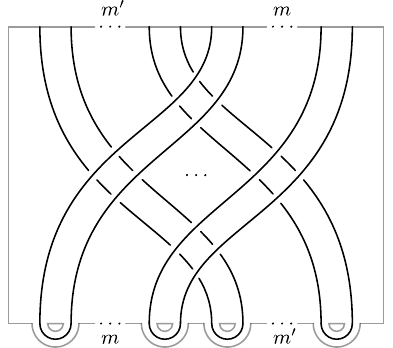},0 \right)
  \]
\end{itemize}

Finally, the \textit{category $\TTH$ of top tangles in handlebodies} is simply obtained from $\ETTH$ by forgetting signature defects.

\subsection{Handle and surgery presentation functors}\label{S:handle_surgery_functors} We consider the \textit{eraser} functor 
\[
 E : \KTan \to \ETTH
\]
given by the identity on objects, and sending every Kirby tangle $T$ in $\KTan(m,m')$ to the pair $(H_m(T),0)$ in $\ETTH(m,m')$, where the top tangle $H_m(T)$ is obtained by erasing dots (an operation which corresponds to 1/2-han\-dle trading) and by replacing the bottom copy of $\id_m$ as shown
\[
 \pic{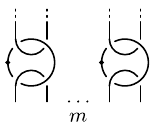} \rightsquigarrow \pic{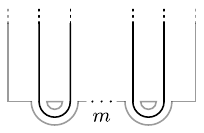}
\]

\begin{proposition}
 The eraser functor $E : \KTan \to \ETTH$ is surjective on objects and full.
\end{proposition}

\begin{proof}
 The fact that the eraser functor $E$ is well-defined follows immediately from the fact that 2-de\-for\-ma\-tions of Kirby tangles \eqref{E:rel_D1}-\eqref{E:rel_D2} project to signed Kirby moves of top tangles \eqref{E:rel_K2}-\eqref{E:rel_K3}. The fact that $E$ is surjective on objects is obvious. The fact that it is also full follows from the fact that if $T$ is a Kirby tangle in $\KTan(m,m')$, then any signed top tangle of the form $(H_m(T),\sig)$ with $\sig \leqs 0$ can be obtained from $(H_m(T \otimes O_-^{\otimes \sig}),0)$ by a sequence of $-\sig$ signed Kirby moves \eqref{E:rel_K-1}, and similarly any signed top tangle of the form $(H_m(T),\sig)$ with $\sig \geqs 0$ can be obtained from $(H_m(T \otimes O_+^{\otimes \sig}),0)$ by a sequence of $\sig$ signed Kirby moves \eqref{E:rel_K+1}, where
 \begin{align*}
  O_- &:= \pic{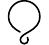} &
  O_+ &:= \pic{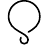} \qedhere
 \end{align*}
\end{proof}

We also have a \textit{forgetful} functor 
\[
 \FrgtT : \ETTH \to \TTH
\]
given by the identity on objects, and sending every pair $(T,\sig)$ in $\ETTH(m,m')$ to the top tangle $T$.

Next, there is a \textit{handle decomposition} functor 
\[
 D : \KTan \to \RHB
\]
given by the identity on objects, and sending every Kirby tangle $T$ in $\KTan(m,m')$ to the handlebody in $\RHB(m,m')$ obtained by attaching 1-handles with cocores determined by Seifert discs of dotted unknots (or equivalently by removing complementary 2-handles with belt spheres determined by dotted unknots, which is the same) and attaching 2-handles along undotted components according to framings. The functor $D$ is an equivalence, see \cite[Proposition~2.3.1]{BP11}.

We consider now the \textit{signed surgery presentation} functor 
\[
 \chi^\sigma : \ETTH \to \FRCob
\]
given by the identity on objects, and sending every pair $(T,\sig)$ in $\ETTH(m,m')$ to the pair $(C(T),\sigma(T \smallsetminus \tilde{T})+\sig)$ in $\FRCob(m,m')$, where $C(T)$ is the cobordism obtained from $H_m$ by carving out an open tubular neighborhood $N(\tilde{T})$ in $H_m$ of the subtangle $\tilde{T}$ of $T$ composed of all arc components, and by performing 2-surgery along the subtangle $T \smallsetminus \tilde{T}$ composed of all circle components, and where $\sigma(T \smallsetminus \tilde{T})$ is the signature of the linking matrix of $(T \smallsetminus \tilde{T}) \subset H_m \subset \R^3$.

\begin{proposition}\label{P:surgery_equivalence}
 The signed surgery presentation functor $\chi^\sigma : \ETTH \to \FRCob$ is an equivalence.
\end{proposition}

\begin{proof}
 The fact that $\chi^\sigma$ is surjective on objects is obvious. Furthermore, every cobordism can be obtained from 
 \[
  C \left( \pic{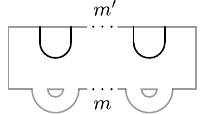} \right)
 \]
 by index 2 surgery along a framed link. This shows that $\chi^\sigma$ is full. Finally, two pairs $(T,\sig)$ and $(T',\sig')$ in $\ETTH$ determine diffeomorphic cobordisms under $\chi^\sigma$ if and only if they are related by a finite sequence of signed Kirby moves~\eqref{E:rel_K-1}-\eqref{E:rel_K3}, as follows from \cite[Theorem~1]{R97}. This shows $\chi^\sigma$ is faithful.
\end{proof}

In particular, we also have a \textit{surgery presentation} functor 
\[ 
 \chi : \TTH \to \RCob
\]
given by the identity on objects, and sending every top tangle $T$ in $\TTH(m,m')$ to the cobordism $C(T)$ in $\RCob(m,m')$. The proof of Proposition~\ref{P:surgery_equivalence} immediately implies the following result.

\begin{proposition}
 The surgery presentation functor $\chi : \TTH \to \RCob$ is an equivalence.
\end{proposition}

\section{Algebraic preliminaries}

In this section we summarize the terminology needed in order to understand the target of the braided monoidal functor of Theorem~\ref{T:main_4-dim}.

\subsection{Unimodular and factorizable ribbon categories}\label{S:unimodular_factorizable_categories} Following \cite[Definition~1.8.5]{EGNO15}, a \textit{finite category} is a linear category $\calC$ over $\Bbbk$ which is equivalent to the category $\mods{A}$ of finite dimensional left $A$-modules for a finite dimensional algebra $A$ over $\Bbbk$. An equivalent and more explicit characterization is provided by \cite[Definition~1.8.6]{EGNO15}. A finite category $\calC$ has enough projectives, which means every object $X \in \calC$ admits a projective cover $P_X \in \calC$ equipped with a projection morphism $\varepsilon_X : P_X \to X$.

A \textit{ribbon category} is a braided rigid monoidal category $\calC$ equipped with a ribbon structure. The monoidal structure, which we always assume to be strict by invoking \cite[Theorem~XI.3.1]{M71}, is given by a tensor product $\otimes : \calC \times \calC \to \calC$ and a tensor unit $\one \in \calC$. Rigidity of $\calC$ ensures the existence of two-sided duals, and this allows us to fix dinatural families of left and right evaluation and coevaluation morphisms $\lev_X : X^* \otimes X \to \one$, $\lcoev_X : \one \to X \otimes X^*$, $\rev_X : X \otimes X^* \to \one$, and $\rcoev_X : \one \to X^* \otimes X$ for every $X \in \calC$. The braided structure is given by a natural family of braiding isomorphisms $c_{X,Y} : X \otimes Y \to Y \otimes X$ for all $X,Y \in \calC$. Finally, the ribbon structure is given by a natural family of twist isomorphisms $\vartheta_X : X \to X$ for every $X \in \calC$. These data must satisfy several axioms, which can be found in \cite[Sections~2.1, 2.10, 8.1, 8.10]{EGNO15}. A ribbon category is always pivotal, meaning it admits a natural family of pivotal isomorphisms $\varphi_X : X \to X^{**}$ for every $X \in \calC$, see \cite[Proposition~8.10.6]{EGNO15}. Since every pivotal category is equivalent to a strict one, as proved in \cite[Theorem~2.2]{NS05}, we will always implicitly assume that $\varphi_X = \id_X$ for every $X \in \calC$.

A rigid category $\calC$ is \textit{unimodular} if it is finite and $P_{\one}^* \cong P_{\one}$, compare with \cite[Definition~6.5.7]{EGNO15}.

An object $X \in \calC$ of a braided monoidal category is said to be transparent if it belongs to the Müger center $M(\calC)$ of $\calC$, which means its double braiding satisfies $c_{Y,X} \circ c_{X,Y} = \id_{X \otimes Y}$ for every object $Y \in \calC$. A unimodular ribbon category $\calC$ is \textit{factorizable} if all its transparent objects are trivial, meaning that they are isomorphic to direct sums of copies of the tensor unit $\one$. A factorizable ribbon category $\calC$ is sometimes also called a \textit{modular} category. The definition given above is equivalent to several others, as proved in \cite[Theorem~1.1]{S16}. In Section~\ref{S:end} we will give yet another equivalent modularity condition.

\subsection{Unimodular and factorizable ribbon Hopf algebras}\label{S:Hopf_algebras} A relevant family of examples of unimodular ribbon categories is provided by categories of representations of unimodular ribbon Hopf algebras. Indeed, let $H$ be a Hopf algebra over a field $\Bbbk$, with unit $\eta : \Bbbk \to H$, product $\mu : H \otimes H \to H$, counit $\varepsilon : H \to \Bbbk$, coproduct $\Delta : H \to H \otimes H$, and antipode $S : H \to H$. For all elements $x,y \in H$, we use the following short notations: for the product, $\mu(x \otimes y) = xy$, for the unit, $\eta(1) = 1$, and for the coproduct, $\Delta(x) = x_{(1)} \otimes x_{(2)}$ (following Sweedler's notation, which hides a sum). Recall that, if $H$ is finite-dimensional, then the antipode $S$ is invertible, as proved in \cite[Theorem~7.1.14]{R12}.

A \textit{ribbon structure} on $H$ is given by an R-matrix $R = R'_i \otimes R''_i \in H \otimes H$ (which hides a sum) and a ribbon element $v_+ \in H$, see \cite[Definitions~VIII.2.2. \& XIV.6.1]{K95}. We denote with $u \in H$ the Drinfeld element and with $M \in H \otimes H$ the M-matrix associated with the R-matrix $R$, which are defined by $u = S(R''_i)R'_i$ and by $M = R''_j R'_i \otimes R'_j R''_i$ respectively (with sums hidden in both notations). We also denote with $v_- \in H$ the inverse of the ribbon element and with $g \in H$ the unique pivotal element compatible with $v_+$, meaning it satisfies $g = uv_-$.

A left integral $\lambda \in H^*$ of $H$ is a linear form on $H$ satisfying $\lambda(x_{(2)}) x_{(1)} = \lambda(x) 1$ for every $x \in H$, and a left cointegral $\Lambda \in H$ of $H$ is an element of $H$ satisfying $x \Lambda = \varepsilon(x) \Lambda$ for every $x \in H$, see \cite[Definition~10.1.1 \& 10.1.2]{R12}. Recall that, if $H$ is finite-dimensional, then a left integral and a left cointegral exist, they are unique up to scalar, and we can lock together their normalizations by requiring
\[
 \lambda(\Lambda) = 1,
\]
as follows from \cite[Theorem~10.2.2]{R12}. Now, recall that a finite-dimensional Hopf algebra $H$ is \textit{unimodular} if $S(\Lambda) = \Lambda$, compare with \cite[Definition~10.2.3]{R12}. If $H$ is unimodular, \cite[Theorem~10.5.4.(e)]{R12} implies that the left integral $\lambda$ satisfies
\[
 \lambda(xy) = \lambda(yS^2(x))
\]
for all $x,y \in H$, see also \cite[Theorem~7.18.12]{EGNO15}. If a ribbon Hopf algebra $H$ is unimodular then, as a consequence of \cite[Proposition~4.1]{H96}, $H$ is also unibalanced, which means that $\lambda(x_{(1)}) x_{(2)} = \lambda(x) g^2$ for every $x \in H$, where $g \in H$ is the pivotal element. Furthermore, the category $\mods{H}$ of finite-di\-men\-sion\-al left $H$-mod\-ules is a unimodular ribbon category thanks to \cite[Lemma~2.5]{L97}.

The Drinfeld map $D : H^* \to H$ of a ribbon Hopf algebra $H$ is the linear map determined by $D(\varphi) := (\varphi \otimes \id_H)(M)$ for every $\varphi \in H^*$, where $M$ is the M-matrix of $H$. By definition, $H$ is \textit{factorizable} if $D$ is a linear isomorphism. This happens if and only if $\mods{H}$ is factorizable, as explained in \cite[Remark~7.7]{FGR17}.

\subsection{Braided Hopf algebras} As we recalled, a Hopf algebra is determined by an object and by several structure morphisms in $\Vect_\Bbbk$. Here, we recall a well-known extension of this definition to an arbitrary braided monoidal category $\calC$.

\begin{definition}\label{D:braided_Hopf}
 A \textit{braided Hopf algebra} in $\calC$ is an object $\calH \in \calC$, together with
 \begin{enumerate}
  \item a \textit{product} $\mu \in \calC(\calH \otimes \calH,\calH)$ and a \textit{unit} $\eta \in \calC(\one,\calH)$, represented as
   \[
    \pic{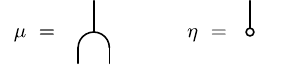}
   \]
  \item a \textit{coproduct} $\Delta \in \calC(\calH,\calH \otimes \calH)$ and a \textit{counit} $\varepsilon \in \calC(\calH,\one)$, represented as
   \[
    \pic{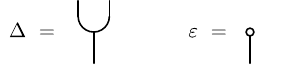}
   \]
  \item an \textit{antipode} and an \textit{inverse antipode} $S,S^{-1} \in \calC(\calH,\calH)$, represented as
   \[
    \pic{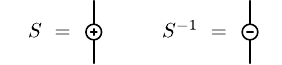}
   \]
 \end{enumerate}
 satisfying the following axioms:
 \begin{enumerate}
  \item $\mu$ and $\eta$ provide an algebra structure
   \[
    \pic{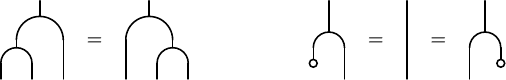}
   \]
  \item $\Delta$ and $\varepsilon$ provide a coalgebra structure
   \[
    \pic{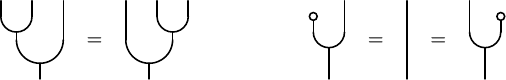}
   \]
  \item algebra and coalgebra structures are compatible
   \[
    \pic{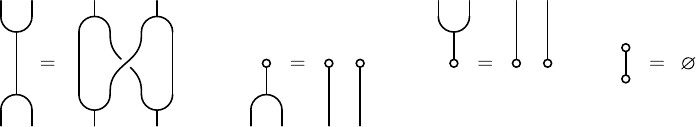}
   \]
  \item $S$ is the convolution-inverse of the identity and it is invertible
   \[
    \pic{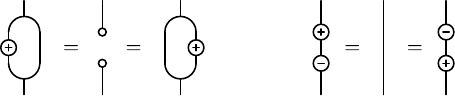}
   \]
 \end{enumerate}
\end{definition}

The antipode $S$ is automatically a bialgebra anti-morphism, meaning it satisfies
\[
 \pic{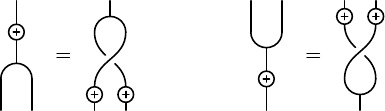}
\]
For a proof of this well-known fact, see for instance \cite[Lemma~1]{K01}. Most of the time, we will abusively denote braided Hopf algebras in $\calC$ simply by their underlying objects, without explicit mention to their structure morphisms.

\section{Bobtcheva--Piergallini Hopf algebras}

The goal of this section is to recall the algebraic presentation of the topological and diagrammatic categories of Sections~\ref{S:relative_handlebodies_and_cobordisms}-\ref{S:kirby_top_tangles}. This is done by defining three categories $\Algf$, $\tAlgt$, and $\Algt$ freely generated by BPH algebras (short for Bobtcheva--Piergallini Hopf algebras), and by extending the commutative diagram we had before as follows:
\begin{center} 
 \begin{tikzpicture}[descr/.style={fill=white}]
  \node (P0) at (135:{2.5*sqrt(2)}) {$\Algf$};
  \node (P1) at (90:{2.5}) {$\tAlgt$};
  \node (P2) at (45:{2.5*sqrt(2)}) {$\Algt$};
  \node (P3) at (180:{2.5}) {$\KTan$};
  \node (P4) at (0,0) {$\ETTH$};
  \node (P5) at (0:{2.5}) {$\TTH$};
  \node (P6) at (225:{2.5*sqrt(2)}) {$\RHB$};
  \node (P7) at (270:{2.5}) {$\FRCob$};
  \node (P8) at (315:{2.5*sqrt(2)}) {$\RCob$};
  \draw
  (P0) edge[->>] node[above] {$\tQtnt$} (P1)
  (P1) edge[->>] node[above] {$\Qtnt$} (P2)
  (P0) edge[->] node[left] {$K$} node[right] {\rotatebox{270}{$\cong$}} (P3)
  (P1) edge[->] node[left] {$T^\sigma$} node[right] {\rotatebox{270}{$\cong$}} (P4)
  (P2) edge[->] node[left] {$T$} node[right] {\rotatebox{270}{$\cong$}} (P5)  
  (P3) edge[->>] node[above] {$E$} (P4)
  (P4) edge[->>] node[above] {$\FrgtT$} (P5)
  (P3) edge[->] node[left] {$D$} node[right] {\rotatebox{270}{$\cong$}} (P6)
  (P4) edge[->] node[left] {$\chi^\sigma$} node[right] {\rotatebox{270}{$\cong$}} (P7)
  (P5) edge[->] node[left] {$\chi$} node[right] {\rotatebox{270}{$\cong$}} (P8)
  (P6) edge[->>] node[above] {$\partial$} (P7)
  (P7) edge[->>] node[above] {$\FrgtC$} (P8);
 \end{tikzpicture}
\end{center}

Our presentation is based on \cite[Section~4]
{K01} and \cite[Sections~4.7 \& 5.5]{BP11}, see also \cite[Sections~2.1 \& 2.4]{B20}.

\subsection{BPH algebras}\label{S:pre-modular_Hopf_agebras}

Let $\calC$ be a braided monoidal category.

\begin{definition}\label{D:ribbon_braided_Hopf}
 A \textit{BP-rib\-bon structure} on a braided Hopf algebra $\calH$ in $\calC$ is given by a \textit{ribbon}\footnote{Notice that the endomorphism appearing on both sides of the left-most equality of axiom~$(i)$ actually corresponds to $v^{-1}$ of \cite[Table~4.7.12]{BP11}. Our choice of signs is motivated by the usual conventions for standard Hopf algebras, as it will become appearent later.} and an \textit{inverse ribbon element} $v_+,v_- \in \calC(\one,\calH)$, represented as
 \[
  \pic{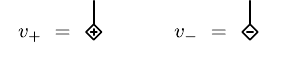}
 \]
 together with their induced \textit{copairings} $w_+, w_- \in \calC(\one,\calH \otimes \calH)$, defined as 
 \[
  \pic{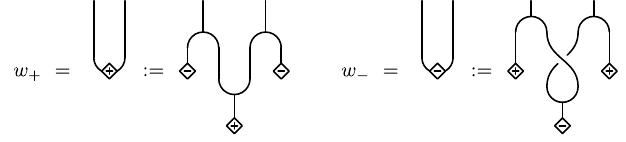}
 \]
 satisfying the following axioms:
 \begin{enumerate}
  \item $v_+$ is central, invertible, normalized, and antipode-invariant
   \[
    \pic{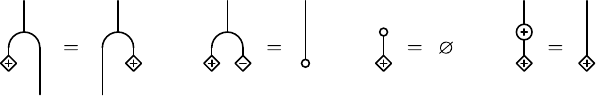}
   \]
  \item $w_+$ is a Hopf copairing
   \[
    \pic{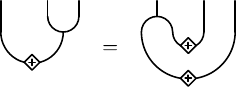}
   \]
  \item $v_+$ is compatible with the coproduct and $w_+$ with the braiding
   \[
    \pic{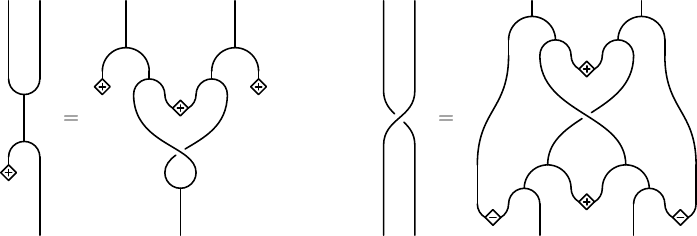}
   \]
 \end{enumerate}
\end{definition}

The inverse ribbon element $v_-$ automatically satisfies
\[
 \pic{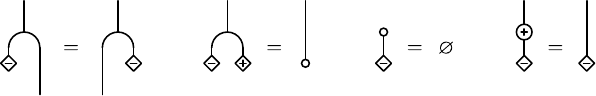}
\]
Moreover, the negative copairing $w_-$ automatically satisfies
\[
 \pic{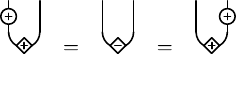}
\]
which means
\[
 \pic{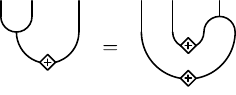}
\]
Indeed, this follows directly from \cite[Lemmas~3 \& 4]{K01}. The compatibility axioms $(iii)$ were first considered in \cite[Equations~(r8) \& (r9)]{BP11}. They can be reformulated as
\[
 \pic{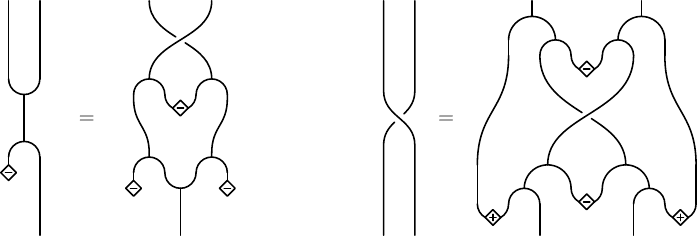}
\]
see \cite[Equations~(p10) \& (p13)]{BP11}. We point out that there are also other possible equivalent forms that the axioms $(iii)$ of Definition~\ref{D:ribbon_braided_Hopf} can take. A convenient one uses the \textit{adjoint morphism} $\alpha \in \calC(\calH \otimes \calH,\calH)$, which is defined as
\[
 \pic{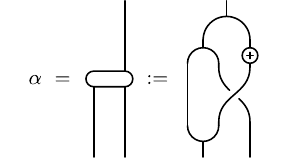}
\]
Then \cite[Proposition~4.4.5]{BP11} implies the equations
\[
 \pic{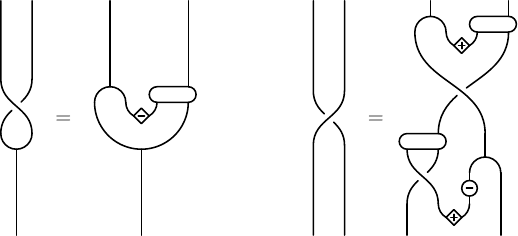}
\]
are equivalent to the axioms $(iii)$ of Definition~\ref{D:ribbon_braided_Hopf}. Remark that BP-rib\-bon Hopf algebras are simply called ribbon Hopf algebras in \cite{BP11}. However, when $\calC = \Vect_\Bbbk$ with the standard symmetric braiding (given by transposition), a ribbon Hopf algebra in the standard sense is not a BP-rib\-bon Hopf algebra. Indeed, the copairing is not a Hopf copairing in general. This is why we change terminology here.

\begin{definition}\label{D:unimodular_braided_Hopf}
 A braided Hopf algebra $\calH \in \calC$ is \textit{BP-u\-ni\-mod\-u\-lar} if it admits an \textit{integral} $\lambda \in \calC(\calH,\one)$ and a \textit{cointegral} $\Lambda \in \calC(\one,\calH)$, represented as
 \[
  \pic{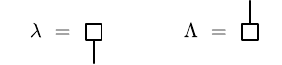}
 \]
 satisfying the following axioms:
 \begin{enumerate}
  \item $\lambda$ is a left coinvariant form which is antipode-invariant 
   \[
    \pic{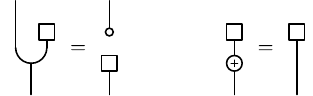}
   \]
  \item $\Lambda$ is a left invariant element which is antipode-invariant 
   \[
    \pic{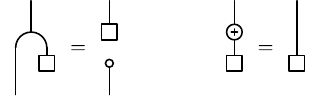}
   \]
  \item $\lambda$ and $\Lambda$ are normalized
   \[
    \pic{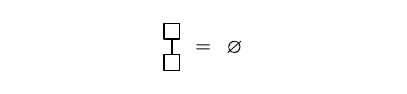}
   \]
 \end{enumerate}
\end{definition}

The integral $\lambda$ and the cointegral $\Lambda$ of a BP-u\-ni\-mod\-u\-lar Hopf algebra are automatically two-sided, meaning they satisfy
\[
 \pic{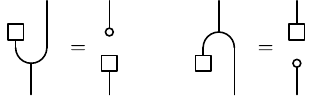}
\]
This is a direct consequence of their antipode-invariance. Remark that BP-u\-ni\-mod\-u\-lar Hopf algebras are simply called unimodular Hopf algebras in \cite{BP11}. However, when $\calC = \Vect_\Bbbk$, a unimodular Hopf algebra in the standard sense is not a BP-u\-ni\-mod\-u\-lar Hopf algebra. Indeed, left integrals are not antipode-invariant in general. This is why we change terminology here.

\begin{definition}\label{D:pre-modular}
 A \textit{BPH algebra} in $\calC$ is a BP-u\-ni\-mod\-u\-lar BP-rib\-bon Hopf algebra in $\calC$.
\end{definition}

We denote by $\HAf$ a BPH algebra, and by $\Algf$ the braided monoidal category freely generated by $\HAf$. This category was first introduced in \cite{BP11} under the name $\calH^r$.

\subsection{Factorizable BPH algebras}\label{S:modular_Hopf_algebras}

\begin{definition}\label{D:modular}
 A BPH algebra $\calH$ in $\calC$ is \textit{factorizable} if the copairing $w_+$ is non-de\-gen\-er\-ate, which means
 \[
  \pic{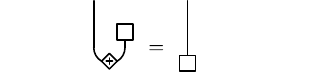}
 \]
\end{definition}

The copairing $w_+$ of a factorizable BPH algebra automatically satisfies also
\[
 \pic{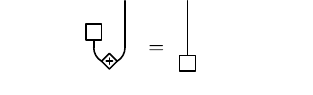}
\]
A direct consequence of the non-de\-gen\-er\-a\-cy of $w_+$ is
\[
 \pic{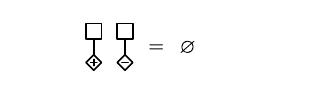}
\]
This condition is called \textit{twist non-de\-gen\-er\-a\-cy} in \cite{DGGPR19}.

We denote by $\tHAt$ a factorizable BPH algebra, and by $\tAlgt$ the braided monoidal category freely generated by $\tHAt$. This category was first introduced in \cite{BP11} under the name $\bar{\calH}^r$.

\begin{definition}\label{D:double_modular}
 A factorizable BPH algebra $\calH$ in $\calC$ is \textit{anomaly-free} if the ribbon element $v_+$ is normalized by
 \[
  \pic{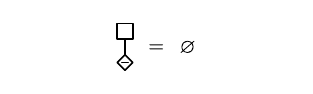}
 \]
\end{definition}

The inverse ribbon element $v_-$ of an anomaly-free factorizable BPH algebra automatically satisfies
\[
 \pic{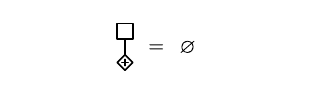}
\]

We denote by $\HAt$ an anomaly-free factorizable BPH algebra, and by $\Algt$ the braided monoidal category freely generated by $\HAt$. This category was first introduced in \cite{BP11} under the name $\bar{\bar{\calH}}^r$.

\subsection{Quotient and algebraic presentation functors}\label{S:Kirby_top_functors} We consider the \textit{quotient} functor 
\[
 \tQtnt : \Algf \to \tAlgt,
\]
which is the braided monoidal functor sending the generating object $\HAf$ of $\Algf$ to the generating object $\tHAt$ of $\tAlgt$, and structure morphisms to structure morphisms. In other words, the functor $\tQtnt$ essentially adds the factorizability relation to the set of axioms satisfied by the BPH algebra $\HAf$.

We also consider the \textit{quotient} functor 
\[
 \Qtnt : \tAlgt \to \Algt,
\]
which is the braided monoidal functor sending the generating object $\tHAt$ of $\tAlgt$ to the generating object $\HAt$ of $\Algt$, and structure morphisms to structure morphisms. Again, the functor $\Qtnt$ essentially adds the anomaly-free relation to the set of axioms satisfied by the factorizable BPH algebra $\tHAt$.

Next, we consider the \textit{Kirby tangle presentation} functor
\[
 K : \Algf \to \KTan,
\]
which is the braided monoidal functor sending the generating object $\HAf$ of $\Algf$ to the object $1 \in \KTan$, and structure morphisms to
\begin{align*}
 \mu &\mapsto \pic{generators_product_KTan} &
 \eta &\mapsto \pic{generators_unit_KTan} \\*[10pt]
 \Delta &\mapsto \pic{generators_coproduct_KTan} &
 \varepsilon &\mapsto \pic{generators_counit_KTan} \\*[10pt]
 S &\mapsto \pic{generators_antipode_KTan} &
 S^{-1} &\mapsto \pic{generators_antipode_inverse_KTan} \\[10pt]
 v_+ &\mapsto \pic{generators_ribbon_KTan} &
 v_- &\mapsto \pic{generators_ribbon_inverse_KTan} \\[10pt]
 \lambda &\mapsto \pic{generators_integral_KTan} &
 \Lambda &\mapsto \pic{generators_cointegral_KTan}
\end{align*}
This automatically implies
\begin{align*}
 w_+ &\mapsto \pic{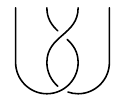} &
 w_- &\mapsto \pic{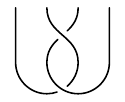}
\end{align*}
The functor $K$ is an equivalence, see \cite[Theorem~4.7.5]{BP11}.

Next, we consider the \textit{signed top tangle presentation} functor
\[
 T^\sigma : \tAlgt \to \ETTH,
\]
which is the braided monoidal functor sending the generating object $\tHAt$ of $\tAlgt$ to the object $1 \in \ETTH$, and structure morphisms to
\begin{align*}
 \mu &\mapsto \left( \pic{generators_product_TTan},0 \right) &
 \eta &\mapsto \left( \pic{generators_unit_TTan},0 \right) \\*[10pt]
 \Delta &\mapsto \left( \pic{generators_coproduct_TTan},0 \right) &
 \varepsilon &\mapsto \left( \pic{generators_counit_TTan},0 \right) \\*[10pt]
 S &\mapsto \left( \pic{generators_antipode_TTan},0 \right) &
 S^{-1} &\mapsto \left( \pic{generators_antipode_inverse_TTan},0 \right) \\[10pt]
 v_+ &\mapsto \left( \pic{generators_ribbon_TTan},0 \right) &
 v_- &\mapsto \left( \pic{generators_ribbon_inverse_TTan},0 \right) \\[10pt]
 \lambda &\mapsto \left( \pic{generators_integral_TTan},0 \right) &
 \Lambda &\mapsto \left( \pic{generators_cointegral_TTan},0 \right)
\end{align*}
This automatically implies
\begin{align*}
 w_+ &\mapsto \left( \pic{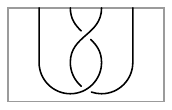},0 \right) &
 w_- &\mapsto \left( \pic{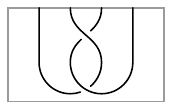},0 \right)
\end{align*}
The functor $T^\sigma$ is an equivalence, see \cite[Theorem~5.5.4]{BP11}.

Similarly, we consider the \textit{top tangle presentation} functor
\[
 T : \Algt \to \TTH,
\]
which is the braided monoidal functor sending the generating object $\HAt$ of $\Algt$ to the object $1 \in \TTH$, and structure morphisms to those prescribed by $T^\sigma$, up to forgetting signature defects. The functor $T$ is an equivalence, see \cite[Theorem~5.5.4]{BP11}.

\section{Kerler--Lyubashenko functors}

In this section we prove our main results. We start by recalling ends and transmutations, before constructing the Kerler--Lyubashenko functor of Theorem~\ref{T:main_4-dim}.

\subsection{End}\label{S:end} Every finite rigid monoidal category $\calC$ admits an \textit{end}
\[
 \calE := \int_{X \in \calC} X \otimes X^* \in \calC,
\]
that is defined as the universal dinatural transformation with target
\begin{align*}
  (\_ \otimes \_^*) : \calC \times \calC^\op & \to \calC \\*
  (X,Y) & \mapsto X \otimes Y^*,
\end{align*}
see \cite[Sections~IX.4-IX.6]{M71} for the terminology. In particular, such an end is given by an object $\calE \in \calC$ together with a dinatural family of structure morphisms $j_X : \calE \to X \otimes X^*$ for every $X \in \calC$, meaning $(f \otimes \id_{X^*}) \circ j_X = (\id_Y \otimes f^*) \circ j_Y$ for every $f : X \to Y$. The end $\calE$ is unique up to isomorphism. Furthermore, if $\calC$ is a ribbon category,
the universal property of the end uniquely determines the following morphisms:
\begin{align*}
 \pic{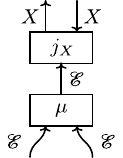} &:= \pic{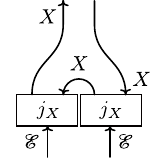} &
 \pic{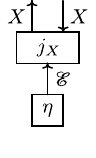} &:= \pic{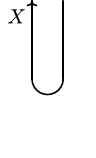} \\*[10pt]
 \pic{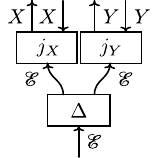} &:= \pic{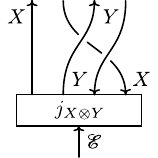} &
 \pic{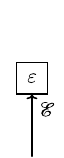} &:= \pic{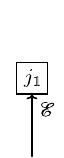} \\*[10pt]
 \pic{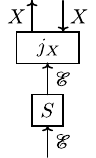} &:= \pic{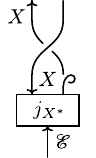} &
 \pic{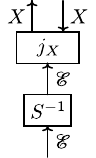} &:= \pic{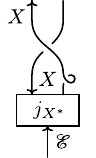} \\[10pt]
 \pic{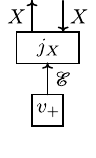} &:= \pic{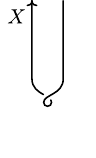} &
 \pic{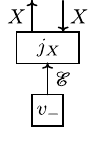} &:= \pic{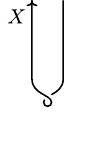} \\*[10pt]
 \pic{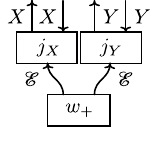} &:= \pic{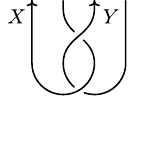} &
 \pic{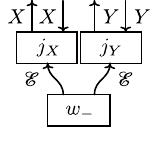} &:= \pic{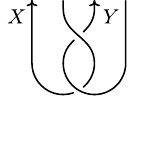} 
\end{align*}
As we will show in Proposition~\ref{P:BP-ribbon}, these morphisms give $\calE$ the structure of a BP-ribbon Hopf algebra.

If $\calC$ is a unimodular ribbon category, then the end $\calE$ also admits an integral $\lambda : \calE \to \one$ and a cointegral $\Lambda : \one \to \calE$, which are unique up to scalar \cite[Propositions~4.2.4 \& 4.2.7]{KL01}. Furthermore, we can always choose a pair satisfying $\lambda \circ \Lambda = \id_{\one}$, whose existence is ensured by \cite[Proposition~4.2.5]{KL01}.

\begin{proposition}\label{P:modularity}
 A unimodular ribbon category $\calC$ is modular if and only if the copairing $w_+ : \one \to \calE \otimes \calE$ of the end $\calE$ is non-degenerate, in the sense that the morphism $D_{\calE^*,\calE} : \calE^* \to \calE$ defined by
 \[
  D_{\calE^*,\calE}:= \pic{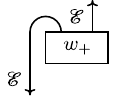}
 \]
 is invertible.
\end{proposition}

A proof of Proposition~\ref{P:modularity} can be found in Appendix~\ref{A:modularity}.

\subsection{Transmutation} Let us fix a unimodular ribbon Hopf algebra $H$, and let $\mods{H}$ denote the unimodular ribbon category of finite-di\-men\-sion\-al left $H$-mod\-ules. We denote by $\ad \in \mods{H}$ the adjoint representation of $H$, which is given by the vector space $H$ equipped with the adjoint left $H$-action
\[
 h \triangleright x := h_{(1)}xS(h_{(2)})
\]
for all $h \in H$ and $x \in \ad$ (recall Sweedler's notation for the coproduct). We have
\[
 \ad = \int_{\mathrlap{V \in \mods{H}}}^{} V \otimes V^*
\]
with structure morphisms given by
\[
 \everymath{\displaystyle}
 \begin{array}{rcl}
  j_V : \ad & \to & V \otimes V^* \\
  x & \mapsto & \sum_{i \in I_V} (x \cdot v_i) \otimes f_i
 \end{array}
\]
for every object $V \in \mods{H}$, where $\{ v_i \in V \mid i \in I_V \}$ and $\{ f_i \in V^* \mid i \in I_V \}$ are dual bases, see \cite[Proposition~A.2]{BD20} for a proof. The fact that the adjoint representation $\ad$ can be given the structure of a braided Hopf algebra in $\mods{H}$ dates back to the work of Majid \cite{M91}, who called the result the \textit{transmutation} of $H$, denoted $\myuline{H}$. Along with the braiding 
\[
 c_{\myuline{H},\myuline{H}}(x \otimes y) = (R'' \triangleright y) \otimes (R' \triangleright x),
\]
structure morphisms of $\myuline{H}$ are given, for all $x,y \in \myuline{H}$, by
\begin{align*}
 \myuline{\mu}(x \otimes y) &= xy, &
 \myuline{\eta}(1) &= 1, \\*
 \myuline{\Delta}(x) &= 
 x_{(1)}S(R''_i) \otimes (R'_i \triangleright x_{(2)}), &
 \myuline{\varepsilon}(x) &= \varepsilon(x), \\*
 \myuline{S}(x) &= 
 R''_i S(R'_i \triangleright x), &
 \myuline{S^{-1}}(x) &= 
 S^{-1}(R'_i \triangleright x) R''_i, \\
 \myuline{v_+}(1) &= v_+, &
 \myuline{v_-}(1) &= v_-, \\*
 \myuline{w_+}(1) &= S(R''_j R'_i) \otimes R'_j R''_i, &
 \myuline{w_-}(1) &= R''_i S^2(R'_j) \otimes R''_j R'_i, \\
 \myuline{\lambda}(x) &= \lambda(x), &
 \myuline{\Lambda}(1) &= \Lambda.
\end{align*}

\subsection{Proofs of Theorems \ref{T:main_4-dim} and \ref{T:main_3-dim}}\label{S:proofs}

Let us restate Theorem \ref{T:main_4-dim}.

\begin{theorem}\label{T:KL_functor_4-dim}
 If $\calC$ is a unimodular ribbon category, then there exists a unique braided monoidal functor
 \[
  \KLF : \Algf \to \calC
 \]
 defined by the assignment $\KLF(\HAf) = \calE$ and sending structure morphisms to structure morphisms.
\end{theorem}

The functor $\KLF : \Algf \to \calC$ is called the \textit{Kerler--Lyubashenko} functor associated with $\calC$. It should be noted that our abusive notation does not distinguish between the functor of Theorem~\ref{T:KL_functor_4-dim} and the one of Theorem~\ref{T:main_4-dim}, but the two are obviously related by the equivalence functor $D \circ K : \RHB \to \Algf$. In order to prove Theorem~\ref{T:KL_functor_4-dim}, we only need to check that the end of a unimodular ribbon category admits the structure of a BP-ribbon Hopf algebra, since we already know it is BP-unimodular.

\begin{proposition}\label{P:BP-ribbon}
 If $\calC$ is a unimodular ribbon category, then the end $\calE$ is a BP-rib\-bon Hopf algebra in $\calC$.
\end{proposition}

\begin{proof}
 The claim essentially follows from \cite[Theorem~5.5.4]{BP11}. Indeed, the structure morphisms for $\calE$ are uniquely determined by the same diagrams which yield a ribbon Hopf algebra structure for the generating object of $\ETTH$. Remark that the proof that this defines a Hopf algebra structure with a Hopf copairing dates back to \cite[Sections~2-3]{L95}. We point out that Lyubashenko actually proves the result for the coend 
 \[
  \calL = \int^{X \in \calC} X \otimes X^* \in \calC,
 \]
 but since he reads morphisms from top to bottom, the same diagrams and graphical proofs also establish our version of the claim. For instance, the fact that the copairing $w_+ : \one \to \calE \otimes \calE$ is a Hopf copairing is witnessed by the equality
 \[
  \pic{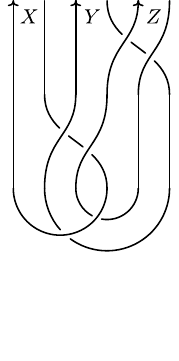} = \pic{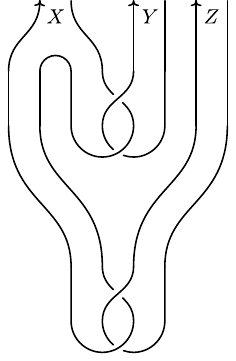}
 \]
 Compare this with \cite[Figure~27]{L95}. Almost all of the remaining relations have been checked in \cite[Section~5.2]{K01} using a graphical calculus that essentially corresponds to the one we considered for $\KTan$. With respect to the ribbon Hopf algebra structure morphisms we introduced earlier, this amounts to checking that
 \begin{align*}
  \pic{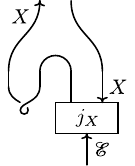} &= \pic{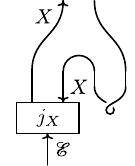} &
  \pic{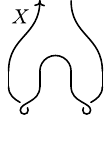} &= \pic{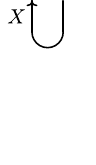} \\*[10pt]
  \pic{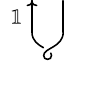} &= \pic{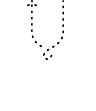} &
  \pic{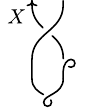} &= \pic{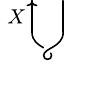}
 \end{align*}
 Remark that the first equality uses the dinaturality of the family of structure morphisms 
 \[
  \{ j_X : \calE \to X \otimes X^* \mid X \in \calC \}.
 \]
 What neither Lyubashenko nor Kerler proved, of course, is that the axioms $(iii)$ of Definition~\ref{D:pre-modular} hold, since they appeared for the first time in the complete algebraic presentation of $\RHB$ found by Bobtcheva and Piergallini in \cite{BP11}. An explicit proof of these identities in our setting amounts to the following: first, we need to check that an isotopy gives
 \[
  \pic{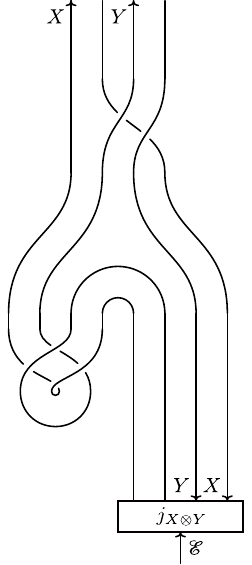} = \pic{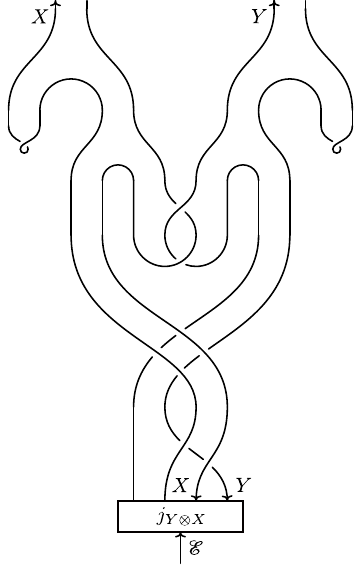}
 \]
 This requires using again the dinaturality of the family of structure morphisms of $\calE$. Next, we need to show that another isotopy gives 
 \[
  \pic{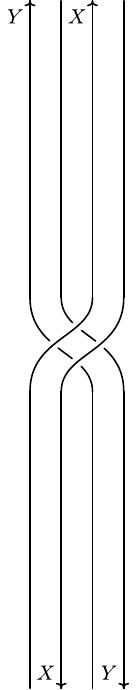} = \pic{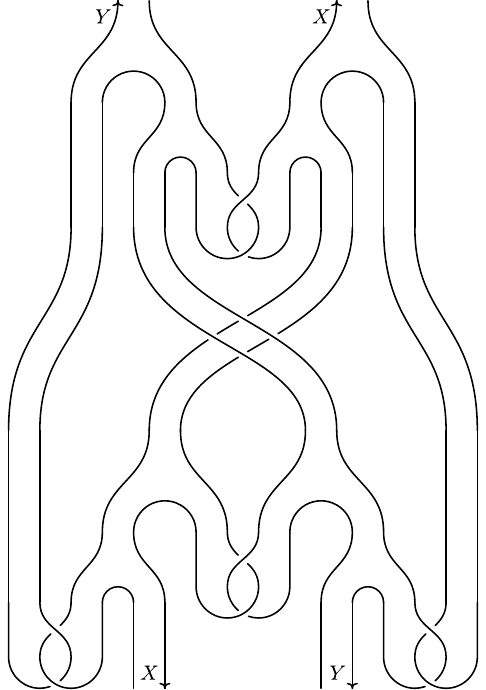}
 \]
 Both claims are easy to verify, and left to the reader.
\end{proof}

\begin{proof}[Proof of Theorem~\ref{T:KL_functor_4-dim}]
 Thanks to Proposition~\ref{P:BP-ribbon}, the end $\calE$ is a BP-rib\-bon Hopf algebra in $\calC$. Furthermore, $\calE$ is a BP-unimodular Hopf algebra in $\calC$ thanks to  \cite[Propositions~4.2.4 \& 4.2.7]{KL01}. This means that the functor
 \[
  \KLF : \Algf \to \calC
 \]
 is well-defined.
\end{proof}

Under stricter assumptions on the unimodular ribbon category $\calC$, the Kerler--Lyubashenko functor $\KLF$ factors through $\tAlgt$, and possibly even $\Algt$.

\begin{theorem}\label{T:KL_functor_3-dim}
 If $\calC$ is a factorizable ribbon category, then there exists a unique braided monoidal functor
 \[
  \KLFts : \tAlgt \to \calC
 \]
 making the following into a commutative diagram:
 \begin{center}
  \begin{tikzpicture}[descr/.style={fill=white}]
   \node (P0) at (180:2.5) {$\Algf$};
   \node (P1) at (0,0) {$\tAlgt$};
   \node (P2) at (270:2.5) {$\calC$};
   \draw
   (P0) edge[->] node[left,xshift=-5pt] {$\KLF$} (P2)
   (P0) edge[->>] node[above] {$\tQtnt$} (P1)
   (P1) edge[->] node[right] {$\KLFts$} (P2);
  \end{tikzpicture}
 \end{center}
 Furthermore, if the specified integral $\lambda : \calE \to \one$ satisfies $\lambda \circ v_+ = \id_{\one}$, then there exists a unique braided monoidal functor
 \[
  \KLFt : \Algt \to \calC
 \]
  making the following into a commutative diagram:
 \begin{center}
  \begin{tikzpicture}[descr/.style={fill=white}]
   \node (P0) at (180:2.5) {$\Algf$};
   \node (P1) at (0,0) {$\tAlgt$};
   \node (P2) at (270:2.5) {$\calC$};
   \node (P3) at (0:2.5) {$\Algt$};
   \draw
   (P0) edge[->] node[left,xshift=-5pt] {$\KLF$} (P2)
   (P0) edge[->>] node[above] {$\tQtnt$} (P1)
   (P1) edge[->] node[descr] {$\KLFts$} (P2)
   (P1) edge[->>] node[above] {$\Qtnt$} (P3)
   (P3) edge[->] node[right,xshift=5pt] {$\KLFt$} (P2);
  \end{tikzpicture}
 \end{center}
\end{theorem}

This is a restatement of Theorem \ref{T:main_3-dim} in Section~\ref{S:main_results}, and we are again adopting an abusive notation. The proof of Theorem~\ref{T:KL_functor_3-dim} follows directly from the following statement, which is essentially a dual version of \cite[Theorem~5]{K96}.

\begin{proposition}\label{P:modular_Hopf_algebra}
 If $\calC$ is a factorizable ribbon category, then there exist an integral $\lambda : \calE \to \one$ and a cointegral $\Lambda : \one \to \calE$ making the end $\calE$ into a factorizable BPH algebra in $\calC$.
\end{proposition}

\begin{proof}
 Thanks to Proposition~\ref{P:modularity}, the copairing $w_+ : \one \to \calE \otimes \calE$ of the end $\calE$ is non-degenerate. Notice that this implies, in particular, that $w_+ : \one \to \calE \otimes \calE$ satisfies \cite[Equation~(24)]{K01}, meaning that
 \begin{align*}
  (\id_{\calE} \otimes f) \circ w_+ &= (\id_{\calE} \otimes g) \circ w_+ & 
  &\Rightarrow & 
  f &= g
 \end{align*}
 for all endomorphisms $f,g : \calE \to \calE$. Then \cite[Lemma~6]{K01} implies that the morphism $(\id_{\calE} \otimes \lambda) \circ w_+$ is a cointegral of $\calE$. But thanks to \cite[Proposition~4.2.4]{KL01}, integrals and cointegrals of $\calE$ are unique up to scalar. Then the claim follows by choosing the normalization
 \begin{align*}
  (\lambda \otimes \lambda) \circ w_+ &= \id_{\one}, &
  \lambda \circ \Lambda &= \id_{\one}.
 \end{align*}
\end{proof}

\section{Computational algorithm for \texorpdfstring{$\mods{H}$}{H-mod}}\label{S:algorithm}

In this section, we present an algorithm for the computation of the functor\footnote{We stress once more that our abusive notation, which overlaps with the one introduced in Theorems~\ref{T:main_4-dim}, \ref{T:main_3-dim}, \ref{T:KL_functor_4-dim}, and \ref{T:KL_functor_3-dim}, is justified by the existence of the equivalence functors recalled in Sections~\ref{S:handle_surgery_functors} and \ref{S:Kirby_top_functors}.} $\KLF : \KTan \to \mods{H}$ associated with a unimodular ribbon Hopf algebra $H$, and of the functor $\KLFts : \ETTH \to \mods{H}$ associated with a factorizable ribbon Hopf algebra $H$. The algorithm is based on the Hennings--Kauffman--Radford one, which has been presented in several slightly different flavors throughout the literature: see \cite[Sections~3 \& 5]{H96} for the original construction, \cite[Section~2]{KR94} for the first redefinition that avoids the use of orientations, \cite[Section~3.3]{K96} for a reformulation based on unoriented top tangles, \cite[Section~4.8]{V03} for a generalized version involving normalizable Kirby elements, \cite[Sections~7.3 \& 12]{H05} for a translation in terms of oriented bottom tangles, and \cite[Section~6.1]{BM02} for the first 4-di\-men\-sion\-al analogue. We will follow more closely the unoriented algorithms, which are based on \textit{singular diagrams} of framed tangles. These are obtained from regular diagrams by discarding framings, and forgetting the difference between overcrossings and undercrossings. On the set of singular diagrams, we consider the equivalence relation generated by all singular versions of the usual local moves corresponding to ambient isotopies of framed tangles, except for the first Reidemeister move. In particular, two equivalent singular diagrams represent homotopic tangles, but not all homotopies are allowed.

\subsection{4-Dimensional case}\label{S:4-dim_case}  Let us start by considering a unimodular ribbon Hopf algebra $H$, and let us compute the intertwiner $\KLF(T) : \ad^{\otimes m} \to \ad^{\otimes m'}$ associated with a Kirby tangle $T \in \KTan(m,m')$. We consider a vector
\[
 x_{1} \otimes \ldots \otimes x_m \in \ad^{\otimes m}
\]
and a regular diagram of $T$. First of all, we insert a bead labeled by $x_i$ on the $i$th incoming strand, just above the $(2i-1)$th incoming boundary point, for every integer $1 \leqs i \leqs m$, as shown:
\[
 \pic{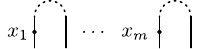}
\]
Next, we pass to the singular version of $T$, while also inserting beads labeled by components of the R-matrix as shown:
\begin{align*}
 \pic{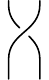} &\mapsto \pic{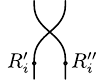} &
 \pic{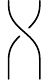} &\mapsto \pic{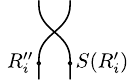}
\end{align*}
Then, we remove dotted components, while also inserting beads labeled by the cointegral as shown:
\begin{align*}
 \pic{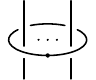} &\mapsto \pic{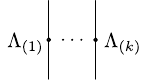}
\end{align*}
When $k = 0$, removing a dotted component costs a multiplicative factor of $\varepsilon(\Lambda)$ in front of $T$. Next, we need to collect all beads sitting on the same strand in one place, which has to be next to the leftmost endpoint, for components which are not closed. As we slide beads past maxima, minima, and crossings, we change their labels according to the rule
\begin{align*}
 \pic{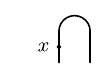} &= \pic{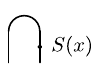} &
 \pic{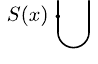} &= \pic{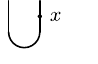} \\*[10pt]
 \pic{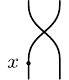} &= \pic{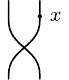} &
 \pic{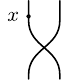} &= \pic{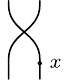}
\end{align*}
Next, we pass from our singular diagram to an equivalent one whose singular crossings all belong to singular versions of twist morphisms, and we replace them with beads labeled by pivotal elements according to the rule
\begin{align*}
 \pic{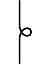} &= \pic{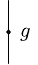} &
 \pic{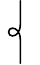} &= \pic{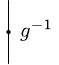}
\end{align*}
This is indeed possible, because we started from a Kirby tangle. Finally, we collect all remaining beads, changing their labels along the way as before, and we multiply everything together according to the rule
\begin{align*}
 \pic{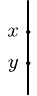} &= \pic{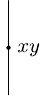}
\end{align*}
In the end, we are left with a planar tangle $B(T)$ carrying at most a single bead on each of its components.

\begin{lemma}\label{L:4-dim_algorithm}
 The intertwiner $\KLF(T) : \ad^{\otimes m} \to \ad^{\otimes m'}$ satisfies
 \[
  \KLF(T)(x_1 \otimes \ldots \otimes x_m) = 
  \left( \prod_{i=1}^m \lambda(y_i x_i) \right) 
  \left( \prod_{j=1}^k \lambda(z_j) \right)
  x'_1 \otimes \ldots \otimes x'_{m'}
 \]
 for every $x_{1} \otimes \ldots \otimes x_m \in \ad^{\otimes m}$, where
 \begin{equation*}
  B(T) = \pic{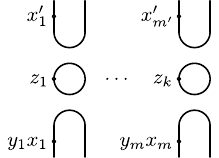}
 \end{equation*}
\end{lemma}

\begin{proof}
 If
 \[
  \KLF(T)(x_1 \otimes \ldots \otimes x_m) = 
  \left( \prod_{i=1}^m \lambda(y_i x_i) \right) 
  \left( \prod_{j=1}^k \lambda(z_j) \right)
  x'_1 \otimes \ldots \otimes x'_{m'}
 \]
 and
 \[
  \KLF(T')(x'_1 \otimes \ldots \otimes x'_{m'}) = 
  \left( \prod_{i=1}^{m'} \lambda(y'_i x'_i) \right) 
  \left( \prod_{j=1}^{k'} \lambda(z'_j) \right)
  x''_1 \otimes \ldots \otimes x''_{m''},
 \]
 then clearly
 \[
  \KLF(T' \circ T)(x_1 \otimes \ldots \otimes x_m) = \KLF(T')(\KLF(T)(x_1 \otimes \ldots \otimes x_m)).
 \]
 Furthermore, we have
 \begin{align*}
  &\pic{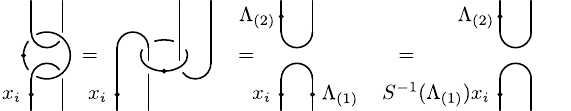}
 \end{align*}
 Then, \cite[Theorem~10.2.2.(c)]{R12} implies
 \[
  \lambda(S^{-1}(\Lambda_{(1)})x_i) \Lambda_{(2)} = \lambda(x_i S(\Lambda_{(1)})) \Lambda_{(2)} = x_i.
 \]
 Therefore
 \begin{align*}
  \KLF(\id_m)(x_1 \otimes \ldots \otimes x_m) 
  &= x_1 \otimes \ldots \otimes x_m,
 \end{align*}
 This means that the algorithm indeed defines a functor with source $\KTan$. Then, it is sufficient to check that
 \[
  \KLF(c_{1,1}) = c_{\myuline{H},\myuline{H}},
 \]
 and that
 \begin{align*}
  \KLF(\mu) &= \myuline{\mu}, &
  \KLF(\eta) &= \myuline{\eta}, \\*
  \KLF(\Delta) &= \myuline{\Delta}, &
  \KLF(\varepsilon) &= \myuline{\varepsilon}, \\*
  \KLF(S) &= \myuline{S}, &
  \KLF(S^{-1}) &= \myuline{S^{-1}}, \\
  \KLF(v_+) &= \myuline{v_+}, &
  \KLF(v_-) &= \myuline{v_-}, \\*
  \KLF(w_+) &= \myuline{w_+}, &
  \KLF(w_-) &= \myuline{w_-}, \\
  \KLF(\lambda) &= \myuline{\lambda}, &
  \KLF(\Lambda) &= \myuline{\Lambda}.
 \end{align*}
 For instance, for the coproduct we need to check that
 \begin{align*}
  \pic{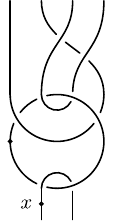} 
  &= \pic{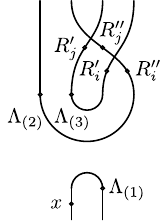} 
  = \pic{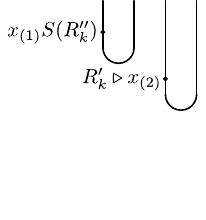} 
 \end{align*}
 where the second equality uses the identity
 \[
  R'_j \otimes R'_i \otimes R''_j R''_i = (R'_k)_{(1)} \otimes (R'_k)_{(2)} \otimes R''_k.
 \]
 The rest is left to the reader.
\end{proof}

The algorithm makes it clear that, when $\calC = \mods{H}$ for a unimodular ribbon Hopf algebra $H$, the invariant determined by the restriction of $\KLF$ to the endomorphism set $\RHB(0,0)$ is a special instance of Bobtcheva and Messia's invariant from \cite[Section~6.1]{BM02}, corresponding to the case $z = 1$ and $w = \Lambda$. In particular, \cite[Corollary~2.15]{BM02} immediately gives us the following result.

\begin{corollary}\label{C:co_semisimple}
 For a unimodular ribbon Hopf algebra $H$, the Kerler--Lyubashenko functor $\KLF : \RHB \to \mods{H}$ satisfies
 \begin{align*}
  \KLF(S^2 \times D^2) &= \lambda(1), &
  \KLF(S^1 \times D^3) &= \varepsilon(\Lambda).
 \end{align*}
 In particular, $\KLF(S^2 \times D^2) \neq 0$ if and only if $H$ is cosemisimple (in the sense that $H^*$ is semisimple), and $\KLF(S^1 \times D^3) \neq 0$ if and only if $H$ is semisimple.
\end{corollary}

\begin{proof}
 The first claim follows immediately from the remark that the endomorphisms $S^2 \times D^2$ and $S^1 \times D^3$ in $\RHB(0,0)$ are represented by an undotted and by a dotted unknot with framing 0, respectively. Then, the second claim follows immediately from \cite[Corollary~10.3.3.(b)]{R12}.
\end{proof}

\subsection{3-Dimensional case}\label{S:3-dim_case} Let us suppose now that $H$ is factorizable, and let us compute the intertwiner $\KLFts(T,\sig) : \ad^{\otimes m} \to \ad^{\otimes m'}$ associated with a signed top tangle $(T,\sig) \in \ETTH(m,m')$. Let us start by considering a vector
\[
 x_{1} \otimes \ldots \otimes x_m \in \ad^{\otimes m}
\]
and a regular diagram of $T$. First of all, if the number of strands of $T$ running along the $i$th 1-han\-dle of $H_m$ is $k_i$, we insert beads labeled by components of the $(k_i-1)$th iterated coproduct $(x_i)_{(1)} \otimes \ldots \otimes (x_i)_{(k_i)}$ for every integer $1 \leqs i \leqs m$, as shown:
\[
 \pic{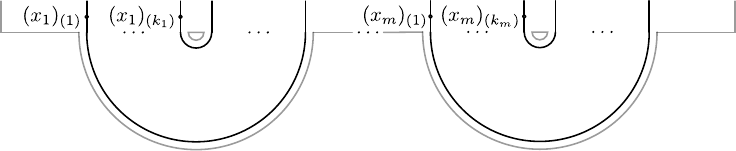}
\]
When $k_i = 0$, we add a multiplicative factor of $\varepsilon(x_{k_i})$ in front of $T$. Next, we pass to the singular version of $T$, while also inserting beads labeled by components of the R-matrix like in the case of Kirby tangles. Then, we attach a complementary 2-han\-dle canceling each 1-han\-dle of $H_m$, as shown:
\begin{align*}
 \pic{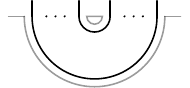} &\mapsto \pic{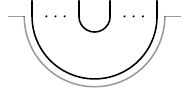}
\end{align*}
The rest of the algorithm is completely analogous to the case of Kirby tangles. In the end, we are left with a planar tangle $B(T)$ carrying at most a single bead on each of its components.

\begin{lemma}\label{L:3-dim_algorithm}
 The intertwiner $\KLFts(T,\sig) : \ad^{\otimes m} \to \ad^{\otimes m'}$ satisfies
 \[
  \KLFts(T,\sig)(x_1 \otimes \ldots \otimes x_m) = 
  \lambda(v_+)^{-\sig} \left( \prod_{j=1}^k \lambda(z_j) \right)
  x'_1 \otimes \ldots \otimes x'_{m'}
 \]
 for every $x_{1} \otimes \ldots \otimes x_m \in \ad^{\otimes m}$, where
 \begin{equation*}
  B(T) = \pic{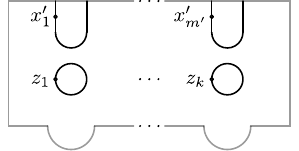}
 \end{equation*}
\end{lemma}

\begin{proof}
 First of all, we start by introducing an operation on top tangles, called the \textit{handle trick}, defined as follows:
 \begin{align*}
  \pic{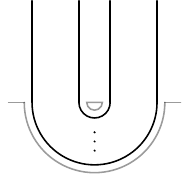} &\leftrightsquigarrow \pic{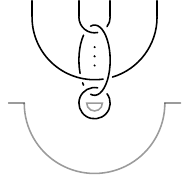}
 \end{align*}
 Remark that, if $T'$ is obtained from $T$ by a sequence of handle tricks, then 
 \[
  (T',\sig) = (T,\sig)
 \]
 for every $\sig \in \Z$. This means we can suppose $k_i = 1$ for all integers $1 \leqs i \leqs m$. Therefore, we can rewrite $B(T)$ as
 \begin{equation*}
  B(T) = \pic{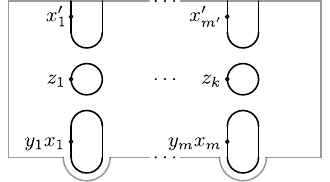}
 \end{equation*}
 Now, if 
 \begin{align*}
  &\KLFts(T,\sig)(x_1 \otimes \ldots \otimes x_m) \\*
  &\hspace*{\parindent} = 
  \lambda(v_+)^{-\sig} \left( \prod_{i=1}^m \lambda(y_i x_i) \right) 
  \left( \prod_{j=1}^k \lambda(z_j) \right)
  x'_1 \otimes \ldots \otimes x'_{m'}
 \end{align*}
 and
 \begin{align*}
  &\KLFts(T',\sig')(x'_1 \otimes \ldots \otimes x'_{m'}) \\*
  &\hspace*{\parindent} = 
  \lambda(v_+)^{-\sig'} \left( \prod_{i=1}^{m'} \lambda(y'_i x'_i) \right) 
  \left( \prod_{j=1}^{k'} \lambda(z'_j) \right)
  x''_1 \otimes \ldots \otimes x''_{m''},
 \end{align*}
 then clearly
 \begin{align*}
  &\KLFts((T',\sig') \circ (T,\sig))(x_1 \otimes \ldots \otimes x_m)  \\*
  &\hspace*{\parindent} = \KLFts(T',\sig')(\KLFts(T,\sig)(x_1 \otimes \ldots \otimes x_m)).
 \end{align*}
 Furthermore, we clearly have
 \begin{align*}
  \KLFts(\id_m)(x_1 \otimes \ldots \otimes x_m) 
  &= x_1 \otimes \ldots \otimes x_m.
 \end{align*}
 Remark that handle tricks do not affect computations, since
 \begin{align*}
  &\pic{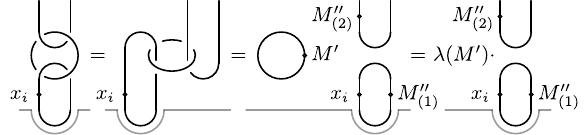} \\*[10pt]
  &\pic{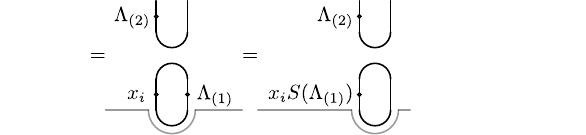}
 \end{align*}
 and we can use again \cite[Theorem~10.2.2.(c)]{R12}, like in the proof of Lemma~\ref{L:4-dim_algorithm}. This means that the algorithm indeed defines a functor with source $\ETTH$. Then, it is sufficient to check that
 \[
  \KLFts(c_{1,1}) = c_{\myuline{H},\myuline{H}},
 \]
 and that
 \begin{align*}
  \KLFts(\mu) &= \myuline{\mu}, &
  \KLFts(\eta) &= \myuline{\eta}, \\*
  \KLFts(\Delta) &= \myuline{\Delta}, &
  \KLFts(\varepsilon) &= \myuline{\varepsilon}, \\*
  \KLFts(S) &= \myuline{S}, &
  \KLFts(S^{-1}) &= \myuline{S^{-1}}, \\
  \KLFts(v_+) &= \myuline{v_+}, &
  \KLFts(v_-) &= \myuline{v_-}, \\*
  \KLFts(w_+) &= \myuline{w_+}, &
  \KLFts(w_-) &= \myuline{w_-}, \\
  \KLFts(\lambda) &= \myuline{\lambda}, &
  \KLFts(\Lambda) &= \myuline{\Lambda}.
 \end{align*}
 For instance, for the antipode we need to check that
 \begin{align*}
  \pic{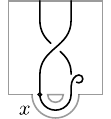} 
  &= \pic{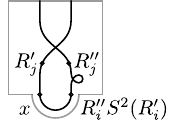} 
  = \pic{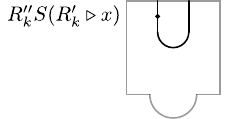} 
 \end{align*}
 where the second equality uses the identity
 \[
  R'_j \otimes R'_i \otimes R''_j R''_i = (R'_k)_{(1)} \otimes (R'_k)_{(2)} \otimes R''_k.
 \]
 The rest is left to the reader.
\end{proof}

The algorithm makes it clear that, when $\calC = \mods{H}$ for a factorizable ribbon Hopf algebra $H$, the invariant determined by the restriction of $\KLFts$ to the endomorphism set $\FRCob(0,0)$ is a special instance of the Hennings invariant of \cite[Proposition~5.2]{H96}, corresponding to the case $z = 1$.

\section{Explicit formulas for small quantum groups}\label{S:explicit_formulas}

In this section, we provide concrete examples of unimodular ribbon categories, and exhibit explicit formulas required for computations.

\subsection{General Lie algebras}\label{S:general_Lie_algebras} Let $\frakg$ be a simple complex Lie algebra of rank $n$, let $\frakh$ be a Cartan subalgebra of $\frakg$, and let $\{ \alpha_1, \ldots, \alpha_n \} \subset \Phi_+$ denote the sets of simple and positive roots respectively. Let $\langle \_,\_ \rangle$ denote the rescaled Killing form on $\frakh^*$ satisfying $\langle \alpha_i,\alpha_i \rangle = 2$ for every short simple root $\alpha_i$, and let
\[
 a_{ij} = \frac{2 \langle \alpha_i,\alpha_j \rangle}{\langle \alpha_i,\alpha_i \rangle}
\]
denote the $ij$th entry of the corresponding Cartan matrix for all $1 \leqs i,j \leqs n$. For every $\alpha \in \Phi_+$ we set
\[
 d_{\alpha} := \frac{\langle \alpha,\alpha \rangle}{2},
\]
and for every integer $1 \leqslant i \leqslant n$ we use the short notation $d_i := d_{\alpha_i}$. For instance, when $\frakg = \fsl_{n+1}$, then $d_i = 1$ for all $1 \leqs i \leqs n$. We denote with $\{ \omega_1,\ldots,\omega_n \} \subset \frakh^*$ the set of fundamental dominant weights, which are determined by the condition $\langle \alpha_i,\omega_j \rangle = d_i \delta_{ij}$ for all integers $1 \leqs i,j \leqs n$. We denote with $\Lambda_R$ the root lattice, which is the subgroup of $\frakh^*$ generated by simple roots, and we denote with $\Lambda_W$ the weight lattice, which is the subgroup of $\frakh^*$ generated by fundamental dominant weights. If $q$ is a formal parameter, then for every $\alpha \in \Phi_+$ we set $q_{\alpha} := q^{d_{\alpha}}$, for all $k \geqs \ell \in \N$ we define
\begin{gather*}
 \{ k \}_{\alpha} := q_{\alpha}^k - q_{\alpha}^{-k}, \quad [k]_{\alpha} := \frac{\{ k \}_{\alpha}}{\{ 1 \}_{\alpha}},
 \quad [k]_{\alpha}! := [k]_{\alpha}[k-1]_{\alpha}\cdots[1]_{\alpha}, \\
 \sqbinom{k}{\ell}_{\alpha} := \frac{[k]_{\alpha}!}{[\ell]_{\alpha}![k-\ell]_{\alpha}!},
\end{gather*}
and for every integer $1 \leqs i \leqs n$ we introduce the short notation 
\[
 q_i := q_{\alpha_i}, \quad \{ k \}_i := \{ k \}_{\alpha_i}, \quad [k]_i := [k]_{\alpha_i},
 \quad [k]_i! := [k]_{\alpha_i}!, \quad
 \sqbinom{k}{\ell}_i := \sqbinom{k}{\ell}_{\alpha_i}.
\]
Let $U_q \frakg$ denote the \textit{quantum group of $\frakg$}, which is the $\bbC(q)$-algebra with generators
\[
 \{ K_i,K_i^{-1},E_i,F_i \mid 1 \leqslant i \leqslant n \}
\]
and relations
\begin{gather*}
 K_i K_i^{-1} = K_i^{-1} K_i = 1, \quad [K_i,K_j] = 0, \\
 K_i E_j K_i^{-1} = q_i^{a_{ij}} E_j, \quad
 K_i F_j K_i^{-1} = q_i^{- a_{ij}} F_j, \\
 [E_i,F_j] = \delta_{ij} \frac{K_i - K_i^{-1}}{q_i-q_i^{-1}}
\end{gather*}
for all integers $1 \leqslant i,j \leqslant n$ and
\begin{align*}
 \displaystyle \sum_{k=0}^{1 - a_{ij}} (-1)^k \sqbinom{1-a_{ij}}{k}_i E_i^k E_j E_i^{1-a_{ij}-k} &= 0, \\
 \displaystyle \sum_{k=0}^{1 - a_{ij}} (-1)^k \sqbinom{1-a_{ij}}{k}_i F_i^k F_j F_i^{1-a_{ij}-k} &= 0
\end{align*}
for all integers $1 \leqslant i,j \leqslant n$ with $i \neq j$. This algebra was first introduced by Drinfeld \cite[Section~4]{D85} and Jimbo \cite[Section~1]{J85}, see also Lusztig \cite[Section~1]{L90} and De Concini--Kac \cite[Section~1]{DK90}. We can make $U_q \frakg$ into a Hopf algebra by setting
\begin{align*}
 \Delta(K_i) &= K_i \otimes K_i, & \varepsilon(K_i) &= 1, & S(K_i) &= K_i^{- 1}, \\
 \Delta(E_i) &= E_i \otimes K_i + 1 \otimes E_i, & \varepsilon(E_i) &= 0, & S(E_i) &= -E_i K_i^{-1}, \\
 \Delta(F_i) &= F_i \otimes 1 + K_i^{-1} \otimes F_i, & \varepsilon(F_i) &= 0, & S(F_i) &= - K_i F_i
\end{align*}
for all integers $1 \leqslant i \leqslant n$. Remark that both Lusztig and De Concini--Kac use a different coproduct and antipode, while we follow the conventions of \cite[Section~9.1]{CP95} and \cite[Section~5.7]{EGNO15}. For every $\mu = m_1 \alpha_1 + \ldots + m_n \alpha_n \in \Lambda_R$ we set
\[
 K_\mu := \prod_{i=1}^n K_i^{m_i} \in U_q \frakg,
\]
and, furthermore, we fix the choice of a reduced expression for the longest element of the Weyl group $W$ of $\frakg$ associated with $\frakh$, which determines an ordering of the set of positive roots $\Phi_+ = \{ \beta_1,\ldots,\beta_N \}$ allowing us to introduce, for every $\beta_k \in \Phi_+$, root vectors
\[
 E_{\beta_k} , F_{\beta_k} \in U_q \frakg,
\]
see \cite[Appendix]{DGP18} for more details.

Let us fix now an integer $r > 2 \max \{ d_1,\ldots,d_n \}$, and let us specialize 
\[
 q = e^{\frac{2 \pi i}{r}}.
\]
First of all, we consider the bilinear pairing
\begin{align*}
 \Lambda_R \times \Lambda_W &\to \bbC^\times \\*
 (\mu,\nu) &\mapsto q^{\langle \mu,\nu \rangle}.
\end{align*}
Remark that
\[
 \Lambda_R^\perp 
 := \{ \nu \in \Lambda_W \mid q^{\langle \mu,\nu \rangle} = 1
 \Forall \mu \in \Lambda_R \} 
 = \left\langle \frac{r}{\gcd(r,d_i)} \omega_i \biggm| 1 \leqs i \leqs n \right\rangle.
\]
Next, let us set 
\[
 r_\alpha := \frac{r}{\gcd(r,2 d_\alpha)}
\]
for every $\alpha \in \Phi_+$. Then, the \textit{small quantum group} $u_q \frakg$ is the $\bbC$-algebra defined by Lusztig in \cite[Section~8.2]{L90}, and obtained from $U_q \frakg$ by adding relations
\[
 K_\mu = 1, \quad E_\alpha^{r_\alpha} = F_\alpha^{r_\alpha} = 0
\]
for every $\mu \in \Lambda_R \cap \Lambda_R^\perp$ and every $\alpha \in \Phi_+$. Clearly $u_q \frakg$ inherits from $U_q \frakg$ the structure of a Hopf algebra. A Poincaré--Birkhoff--Witt basis of $u_q \frakg$ is given by
\[
 \left\{ \left( \prod_{k = 1}^N F_{\beta_k}^{c_k} \right) K_{\mu} \left( \prod_{k = 1}^N E_{\beta_k}^{b_k} \right) \Biggm| 
 \begin{array}{l}
  \mu \in \Lambda_R / (\Lambda_R \cap \Lambda_R^\perp), \\
  0 \leqs b_k,c_k < r_{\beta_k}
 \end{array} \right\},
\]
as follows from \cite[Theorem~8.3]{L90}. If we consider $D \in u_q \frakg \otimes u_q \frakg$ given by
\[
 D := \frac{1}{\left| \Lambda_R / (\Lambda_R \cap \Lambda_R^\perp) \right|} \sum_{\mu,\mu' \in \Lambda_R / (\Lambda_R \cap \Lambda_R^\perp)} q^{- \langle \mu,\mu' \rangle} K_\mu \otimes K_{\mu'}
\]
and $\Theta \in u_q \frakg \otimes u_q \frakg$ given by
\[
 \Theta := \sum_{b_1=0}^{r_{\beta_1}-1} \cdots \sum_{b_N=0}^{r_{\beta_N}-1} 
 \left( \prod_{k=1}^N \frac{\{ 1 \}_{\beta_k}^{b_k}}{[b_k]_{\beta_k}!} q^{\frac{b_k(b_k-1)}{2}} \right) 
 \left( \prod_{k=1}^N E_{\beta_k}^{b_k} \right) \otimes 
 \left( \prod_{k=1}^N F_{\beta_k}^{b_k} \right)
\]
then $R := D \Theta \in u_q \frakg \otimes u_q \frakg$ is an R-matrix for $u_q \frakg$, see \cite[Chapter~4]{L93} and \cite[Section~A.2]{L94}. Remark that Lyubashenko obtains the R-matrix $\Theta D$, because he uses Lusztig's coproduct defined by $\tilde{\Delta}(x) := D^{-1} \Delta(x) D$, compare with \cite[Proposition~A.1.3]{L94}. Then, it can be shown that $D \Theta$ is an R-matrix for $\Delta$ if and only if $\Theta D$ is an R-matrix for $\tilde{\Delta}$. A pivotal element $g \in u_q \frakg$ is given by $g = K_{2 \rho}$ where
\[
 \rho := \frac{1}{2} \sum_{k=1}^N \beta_k.
\]
Formulas appearing in Section~\ref{S:Hopf_algebras} allow us to express the ribbon element and its inverse $v_+, v_- \in u_q \frakg$ in terms of the R-matrix $R$ and the pivotal element $g$. Next, thanks to \cite[Proposition~A.5.1]{L94}, every non-zero left integral $\lambda$ of $u_q \frakg$ is given by
\[
 \lambda \left( K_\mu \left( \prod_{k = 1}^N F_{\beta_k}^{c_k} \right) \left( \prod_{k = 1}^N E_{\beta_k}^{b_k} \right) \right) = \xi \delta_{\mu,-2 \rho} 
 \prod_{k=1}^N \delta_{b_k,r_{\beta_k}-1} \prod_{k=1}^N \delta_{c_k,r_{\beta_k}-1}
\]
for some $\xi \in \bbC^\times$. Remark that our formula agrees with Lyubashenko's one, despite the use of different coproducts. Indeed, Lusztig's coproduct can be written as $\tilde{\Delta} := (\psi \otimes \psi) \circ \Delta^\op \circ \psi$ for the involutive $\bbQ$-algebra isomorphism $\psi : u_q \frakg \to u_q \frakg$ defined by $\psi(q)=q^{-1}$ and by $\psi(E_i) = E_i$, $\psi(F_i) = F_i$, $\psi(K_i) = K_i^{-1}$. Then, $\varphi$ is a left integral for $\tilde{\Delta}$ if and only if $\varphi \circ \psi$ is a right integral for $\Delta$, and $\varphi$ is a right integral for $\Delta$ if and only if $\varphi(K_{4 \rho} \_)$ is a left integral for $\Delta$. The last equivalence follows from \cite[Proposition~A.5.2]{L94}, which tells us that $u_q \frakg$ is unimodular, and that every non-zero two-sided cointegral $\Lambda$ of $u_q \frakg$ satisfying $\lambda(\Lambda) = 1$ is given by
\[
 \Lambda := \xi^{-1} \sum_{\mu \in \Lambda_R / (\Lambda_R \cap \Lambda_R^\perp)} K_\mu \left( \prod_{k = 1}^N F_{\beta_k}^{r_{\beta_k}-1} \right) \left( \prod_{k = 1}^N E_{\beta_k}^{r_{\beta_k}-1} \right)
\]
for some $\xi \in \bbC^\times$. Furthermore, thanks to \cite[Corollary~A.3.3]{L94}, $u_q \frakg$ is factorizable if and only if
\[
 \left\{ K_{2 \mu} \in u_q \frakg \mid \mu \in \Lambda_R \right\}
 = 
 \left\{ K_\mu \in u_q \frakg \mid \mu \in \Lambda_R \right\}.
\]

\subsection{Case \texorpdfstring{$\frakg = \fsl_2$}{g = sl(2)}}\label{S:small_quantum_sl2} When $\frakg = \fsl_2$, computations are easier and more explicit. In particular, as explained in \cite[Section~B.1]{L94}, the ribbon Hopf algebra $u_q \fsl_2$ is factorizable if $r \not\equiv 0 \mod 4$, it is twist non-degenerate but not factorizable if $r \equiv 4 \mod 8$, and it is twist degenerate if $r \equiv 0 \mod 8$. Indeed, if we set
\begin{align*}
 r' &:= \frac{r}{\gcd(r,2)} = 
 \begin{cases}
  \frac{r}{2} & \mbox{ if } r \equiv 0 \mod 2, \\
  r & \mbox{ if } r \equiv 1 \mod 2,
 \end{cases} \\*
 r'' &:= \frac{r}{\gcd(r,4)} = 
 \begin{cases}
 \frac{r}{4} & \mbox{ if } r \equiv 0 \mod 4, \\
  \frac{r}{2} & \mbox{ if } r \equiv 2 \mod 4, \\
  r & \mbox{ if } r \equiv 1 \mod 2,
 \end{cases}
\end{align*}
then the small quantum group $u_q \fsl_2$ has generators $\{ K,K^{-1},E,F \}$ and relations
\begin{align*}
 K K^{-1} &= K^{-1} K = 1, &
 K E K^{-1} &= q^2 E, &
 K F K^{-1} &= q^{-2} F, \\
 [E,F] &= \frac{K - K^{-1}}{q-q^{-1}} &
 K^{r'} &= 1, &
 E^{r'} &= F^{r'} = 0.
\end{align*}
The PBW basis of $u_q \fsl_2$ is
\[
 \left\{ E^a F^b K^c \mid \ 0 \leqs a,b,c \leqs r' - 1 \right\}.
\]
Specializing Lusztig's formulas from Section~\ref{S:general_Lie_algebras}, the R-matrix $R \in u_q \fsl_2 \otimes u_q \fsl_2$ and its inverse $R^{-1} \in u_q \fsl_2 \otimes u_q \fsl_2$ are given by
\begin{align*}
 R &= \frac{1}{r'} \sum_{a,b,c=0}^{r'-1} \frac{\{ 1 \}^a}{[a]!}
 q^{\frac{a(a-1)}{2} - 2bc} K^b E^a \otimes K^c F^a, \\*
 R^{-1} &= \frac{1}{r'} \sum_{a,b,c=0}^{r'-1} \frac{\{ -1 \}^a}{[a]!}
 q^{-\frac{a(a-1)}{2} + 2bc} E^a K^b \otimes F^a K^c.
\end{align*}
Our preferred non-zero left integral $\lambda$ of $u_q \frakg$ is given by
\[
 \lambda \left( E^a F^b K^c \right) = \lambda \left( F^b E^a K^c \right) = \frac{\sqrt{r''} [r'-1]!}{\{ 1 \}^{r'-1}} \delta_{a,r'-1} \delta_{b,r'-1} \delta_{c,r'-1},
\]
and our preferred non-zero two-sided cointegral $\Lambda$ of $u_q \frakg$ satisfying $\lambda(\Lambda) = 1$ is given by
\[
 \Lambda := \frac{\{ 1 \}^{r'-1}}{\sqrt{r''} [r'-1]!} \sum_{a=0}^{r'-1} E^{r'-1} F^{r'-1} K^a.
\]
Now Lemma~\ref{L:ribbon} in Appendix~\ref{A:ribbon_monodromy} immediately implies the following.

\begin{lemma}\label{L:stabilization}
 The stabilization coefficients $\lambda(v_+),\lambda(v_-) \in \Bbbk$ are given by
 \begin{align*}
  \lambda(v_+) &=
  \begin{cases}
   i^{\frac{r-1}{2}} q^{\frac{r+3}{2}} \\
   \left( \frac{2}{r'} \right) i^{\frac{r'-1}{2}} q^{\frac{r'+3}{2}} \\
   - q^{\frac{r''+3}{2}} \\
   0
  \end{cases} &
  \lambda(v_-) &=
  \begin{cases}
   i^{-\frac{r-1}{2}} q^{-\frac{r+3}{2}} & \mbox{ if } r \equiv 1 \mod 2, \\
   \left( \frac{2}{r'} \right) i^{-\frac{r'-1}{2}} q^{-\frac{r'+3}{2}} & \mbox{ if } r \equiv 2 \mod 4, \\
   q^{-\frac{r''+3}{2}} & \mbox{ if } r \equiv 4 \mod 8, \\
   0 & \mbox{ if } r \equiv 0 \mod 8,
  \end{cases}
 \end{align*}
 where $\left( \frac{2}{r'} \right)$ is the Jacobi symbol of $2$ modulo $r'$.
\end{lemma}

\begin{proof}
 If $r \equiv 1 \mod 2$, Equation~\eqref{E:ribbon_1} gives
 \begin{align*}
  \lambda(v_+) 
  &= i^{\frac{r-1}{2}} q^{-\frac{(r+2)(r-1)}{2}+\frac{r+1}{2}(-1)^2}
  = i^{\frac{r-1}{2}} q^{1+\frac{r+1}{2}} 
  = i^{\frac{r-1}{2}} q^{\frac{r+3}{2}},
 \end{align*}
 while Equation~\eqref{E:inverse_ribbon_1} gives
 \begin{align*}
  \lambda(v_-) 
  &= \overline{\lambda(v_+)}
  = i^{-\frac{r-1}{2}} q^{-\frac{r+3}{2}}.
 \end{align*}
 If $r \equiv 2 \mod 4$, then $r' \equiv 1 \mod 2$, and furthermore
 \[
  q^{r'} = -1.
 \]
 Then Equation~\eqref{E:ribbon_2} gives
 \begin{align*}
  \lambda(v_+) 
  &= \left( \frac{2}{r'} \right) i^{\frac{r'-1}{2}} q^{-\frac{(r'+2)(r'-1)}{2}+\frac{(r'+1)^2}{2}(2r'-1)^2} \\*
  &= \left( \frac{2}{r'} \right) i^{\frac{r'-1}{2}} q^{-\frac{(r'+2)(r'-1)}{2}+\frac{(r'+1)^2}{2}} \\*
  &= \left( \frac{2}{r'} \right) i^{\frac{r'-1}{2}} q^{\frac{r'+3}{2}},
 \end{align*}
 while Equation~\eqref{E:inverse_ribbon_2} gives
 \begin{align*}
  \lambda(v_-) 
  &= \overline{\lambda(v_+)}
  = \left( \frac{2}{r'} \right) i^{-\frac{r'-1}{2}} q^{-\frac{r'+3}{2}}.
 \end{align*}
 If $r \equiv 4 \mod 8$, then $r' \equiv 2 \mod 4$ and $r'' \equiv 1 \mod 2$, and furthermore
 \[
  q^{r''} = i.
 \]
 Then Equation~\eqref{E:ribbon_3} gives
 \begin{align*}
  \lambda(v_+) 
  &= (-1)^{r'-1} i^{\frac{r''-1}{2}} q^{-\frac{(r'+2)(r'-1)}{2}+\frac{(r''+1)^3}{2}(2r'-1)^2} \\*
  &= - i^{\frac{r''-1}{2}} q^{-(r''+1)(2r''-1)+\frac{(r''+1)^3}{2}} \\*
  &= - i^{\frac{r''-1}{2}} q^{-(2r''^2+r''-1)+\frac{r''+1}{2}(r''^2+2r''+1)} \\*
  &= - i^{\frac{r''-1}{2} -(2r''+1) + \frac{r''+1}{2}(r'' + 2)} q^{1+\frac{r''+1}{2}} \\*
  &= - i^{\frac{r''-1}{2} + \frac{r''^2-r''}{2}} q^{\frac{r''+3}{2}} \\*
  &= - i^{\frac{r''^2-1}{2}} q^{\frac{r''+3}{2}} \\*
  &= - q^{\frac{r''+3}{2}},
 \end{align*}
 while Equation~\eqref{E:inverse_ribbon_3} gives
 \begin{align*}
  \lambda(v_-) 
  &= - \overline{\lambda(v_+)}
  = q^{-\frac{r''+3}{2}}.
 \end{align*}
 Finally, if $r \equiv 0 \mod 8$, then Equations~\eqref{E:ribbon_4} and \eqref{E:inverse_ribbon_4} give
 \[
  \lambda(v_+) = \lambda(v_-) = 0,
 \]
 because the coefficient in front of $F^{r-1} E^{r-1} K^{r-1}$ vanishes for both $v_+$ and $v_-$.
\end{proof}

Furthermore, Lemma~\ref{L:monodromy} in Appendix~\ref{A:ribbon_monodromy} immediately implies the following.

\begin{lemma}\label{L:S^2_times_S^2}
 The Hopf link coefficient $\lambda(S(R''_j R'_i)) \lambda(R'_j R''_i) \in \Bbbk$ is given by
 \begin{align*}
  \lambda(S(R''_j R'_i)) \lambda(R'_j R''_i) &= (-1)^{r'-1}
 \end{align*}
\end{lemma}

\begin{proof}
 First of all, we have
 \begin{align*}
  \lambda(S(K^cF^bE^a)) 
  &= (-1)^{a+b}q^{(a-b)(1-a+b)}\lambda(E^aF^bK^{-a+b-c}) \\*
  &= \frac{\sqrt{r''} [r'-1]!}{\{ 1 \}^{r'-1}} \delta_{a,r'-1} \delta_{b,r'-1} \delta_{c,1}. 
 \end{align*}
 If $r \equiv 1 \mod 2$ or $r \equiv 2 \mod 4$, Equation~\eqref{E:M_1} gives
 \begin{align*}
  \lambda(S(R''_j R'_i)) \lambda(R'_j R''_i) &=
  \left( \frac{\sqrt{r'} [r'-1]!}{\{ 1 \}^{r'-1}} \right)^2 \frac{1}{r'}
  \left( \frac{\{ 1 \}^{r'-1}}{[r'-1]!} \right)^2 
  q^{(r'-1)(r'-2) - 2(r'-1)^2} \\*
  &= q^{-r'(r'-1)} 
  = 1.
 \end{align*}
 If $r \equiv 0 \mod 4$, Equation~\eqref{E:M_2} gives
 \begin{align*}
  \lambda(S(R''_j R'_i)) \lambda(R'_j R''_i) &=
  \left( \frac{\sqrt{r''} [r'-1]!}{\{ 1 \}^{r'-1}} \right)^2 \frac{1}{r''}
  \left( \frac{\{ 1 \}^{r'-1}}{[r'-1]!} \right)^2 
  q^{(r'-1)(r'-2) - 2(r'-1)^2} \\*
  &= q^{-r'(r'-1)} 
  = -1. \qedhere
 \end{align*}
\end{proof}

Let us now derive some consequences of our computations. In order to do this, we consider the trivial bundle $S^2 \times S^2$, and we denote by $S^2 \ttimes S^2$ the twisted one. Both of these closed $4$-manifolds admit a handle decomposition featuring a single $0$-handle, two $2$-handles, and a single $4$-handle. Notice that $S^2 \ttimes S^2$ can be obtained from the connected sum of $\bbC P^2$ and $\overline{\bbC P^2}$ by a single $2$-handle slide. Removing $4$-handles from $S^2 \times S^2$ and $S^2 \ttimes S^2$, we obtain two different $4$-dimensional $2$-handlebodies that can be distinguished by our invariant.

\begin{corollary}\label{C:CP2_S2_times_S2}
 The Kerler--Lyubashenko functor $\KLF : \RHB \to \mods{u_q \fsl_2}$ satisfies
 \begin{align*}
  \KLF((S^2 \times S^2) \smallsetminus \mathring{D}^4) &= (-1)^{r'-1}, \\*
  \KLF((S^2 \ttimes S^2) \smallsetminus \mathring{D}^4) &= 
  \begin{cases}
   1 & \mbox{ if } r \not\equiv 0 \mod 4, \\
   -1 & \mbox{ if } r \equiv 4 \mod 8, \\
   0 & \mbox{ if } r \equiv 0 \mod 8,
  \end{cases}
 \end{align*}
 where $r' = \frac{r}{\gcd(r,2)}$.
\end{corollary}

\begin{proof}
 The claim follows from Lemmas~\ref{L:stabilization} and \ref{L:S^2_times_S^2}, because $(S^2 \times S^2) \smallsetminus \mathring{D}^4$ is represented by an undotted Hopf link with framing $0$ on both components, while $(S^2 \ttimes S^2) \smallsetminus \mathring{D}^4$ is represented, after a \eqref{E:rel_D1} move, by an undotted unlink with two components of framing $+1$ and $-1$ respectively.
\end{proof}

\begin{remark}\label{R:boundary&signature}
 In particular, Corollary~\ref{C:CP2_S2_times_S2} implies that $\KLF : \RHB \to \mods{u_q \fsl_2}$ cannot factor through $\partial : \RHB \to \FRCob$ when $r \equiv 0 \mod 8$. Indeed, it cannot even depend exclusively on boundary, signature, and Euler characteristic of $4$-di\-men\-sion\-al $2$-handlebodies, because
 \begin{gather*}
  \partial_+((S^2 \times S^2) \smallsetminus \mathring{D}^4) \cong
  \partial_+((S^2 \ttimes S^2) \smallsetminus \mathring{D}^4) \cong D^2 \times I, \\*
  \sigma((S^2 \times S^2) \smallsetminus \mathring{D}^4) =
  \sigma((S^2 \ttimes S^2) \smallsetminus \mathring{D}^4) = 0, \\*
  \chi((S^2 \times S^2) \smallsetminus \mathring{D}^4) =
  \chi((S^2 \ttimes S^2) \smallsetminus \mathring{D}^4) = 3.
 \end{gather*}
\end{remark}

\appendix

\section{Modularity}\label{A:modularity}

In this appendix, we prove Proposition~\ref{P:modularity}, which states that a unimodular ribbon category $\calC$ is modular if and only if the copairing $w_+ : \one \to \calE \otimes \calE$ of the end $\calE$ is non-degenerate. In order to prove the claim, it is convenient to recall the definition of the coend 
\[
 \calL := \int^{X \in \calC} X^* \otimes X \in \calC,
\]
which is the universal dinatural transformation with source
\begin{align*}
 (\_^* \otimes \_) : \calC^\op \times \calC & \to \calC \\*
 (X,Y) & \mapsto X^* \otimes Y.
\end{align*}
In particular, such a coend is given by an object $\calL \in \calC$ together with a dinatural family of structure morphisms $i_X : X^* \otimes X \to \calL$ for every $X \in \calC$.

\begin{proof}[Proof of Proposition~\ref{P:modularity}]
  The strategy is to resort to \cite[Proposition~4.11]{FGR17}, which states that $\calC$ is modular if and only if the Drinfeld map $D_{\calL,\calE} : \calL \to \calE$ of \cite[Equation~(4.16)]{FGR17} is invertible. Just like in \cite[Lemma~4.12]{FGR17}, it is easy to show that the object $\calE^* \in \calC$ equipped with the dinatural family of structure morphisms
 \[
  k_X := \pic{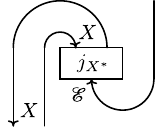} 
 \]
 is a coend
 \[
  \calE^* \cong \int^{X \in \calC} X^* \otimes X.
 \]
 The universal property of the coend $\calL$ yields a unique isomorphism $\varphi : \calL \to \calE^*$ satisfying
 \[
  k_X = \varphi \circ i_X
 \]
 for all $X \in \calC$. Then we have
 \begin{align*}
  j_Y \circ D_{\calL,\calE} \circ i_X &= \pic{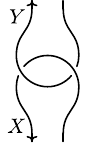} = \hspace*{-10pt} \pic{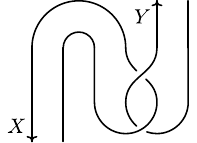} \\*
  &= j_Y \circ D_{\calE^*,\calE} \circ k_X \\*
  &= j_Y \circ D_{\calE^*,\calE} \circ \varphi \circ i_X
 \end{align*}
 for all $X,Y \in \calC$, which implies
 \[
  D_{\calL,\calE} = D_{\calE^*,\calE} \circ \varphi
 \]
 because of the universal property of $\calE$ and $\calL$. This proves that $D_{\calL,\calE}$ is invertible if and only if $D_{\calE^*,\calE}$ is. 
\end{proof}

\section{Ribbon element and monodromy matrix}\label{A:ribbon_monodromy}

In this appendix we give explicit formulas for the ribbon element $v_+ \in u_q \fsl_2$ and for the monodromy matrix $M \in u_q \fsl_2 \otimes u_q \fsl_2$, as well as their inverses. Some of these computations involve generalized quadratic Gauss sums
\begin{equation}
 G(a,b;c) := \sum_{n=0}^{c-1} e^{2 \pi i \frac{a n^2 + b n}{c}} \label{E:gqGs}
\end{equation}
for $a,b,c \in \N$. In the special case $b=0$, we use the short notation
\begin{equation}
 G(a;c) := G(a,0;c). \label{E:sqGs}
\end{equation}
For convenience, let us recall some well-known facts about generalized quadratic Gauss sums.

\begin{lemma}
 Generalized quadratic Gauss sums satisfy:
 \begin{enumerate}
  \item \textit{Square completion}: If $\gcd(a,c)=1$ and there exists $f(a) \in \N$ satisfying $f(a)a \equiv 1 \mod c$ and $f(a)b \equiv 0 \mod 2$, then
   \begin{equation}
    G(a,b;c) = G(a;c) e^{-2 \pi i \frac{a \frac{f(a)^2 b^2}{4}}{c}}. \label{E:complete_the_square}
   \end{equation}
  \item \textit{Multiplicativity}: If $\gcd(c,d)=1$, then
   \begin{equation}
    G(a,b;cd) = G(ac,b;d)G(ad,b;c). \label{E:multiplicativity}
   \end{equation}
  \item \textit{Triviality}: If $\gcd(a,c)=1$, $b \equiv 1 \mod 2$, and $c \equiv 0 \mod 4$, then
   \begin{equation}
    G(a,b;c) = 0. \label{E:triviality}
   \end{equation}
 \end{enumerate}
\end{lemma}

\begin{proof}
 For what concerns point $(i)$, as $n$ runs through all residue classes modulo $c$, so does $m = n + \frac{f(a)b}{2}$. Therefore, Equation~\eqref{E:complete_the_square} can be established by a change of variable as follows:
 \begin{align*}
  G(a,b;c) 
  &= \left( \sum_{n=0}^{c-1} e^{2 \pi i \frac{a \left( n + \frac{f(a)b}{2} \right)^2}{c}} \right) e^{-2 \pi i \frac{a \frac{f(a)^2 b^2}{4}}{c}} \\*
  &= \left( \sum_{m=0}^{c-1} e^{2 \pi i \frac{am^2}{c}} \right) e^{-2 \pi i \frac{a \frac{f(a)^2 b^2}{4}}{c}} \\*
  &= G(a;c) e^{-2 \pi i \frac{a \frac{f(a)^2 b^2}{4}}{c}}.
 \end{align*}
 Similarly, for what concerns point $(ii)$, as $k$ and $\ell$ run through all residue classes modulo $c$ and $d$ respectively, $n = ck+d\ell$ runs thorugh all residue classes modulo $cd$. Therefore, Equation~\eqref{E:multiplicativity} can be established by a change of variable as follows:
 \begin{align*}
  G(ac,b;d) G(ad,b;c) 
  &= \left( \sum_{k=0}^{d-1} e^{2 \pi i \frac{ac k^2 + b k}{d}} \right) \left( \sum_{\ell=0}^{c-1} e^{2 \pi i \frac{ad \ell^2 + b \ell}{c}} \right) \\*
  &= \sum_{k=0}^{d-1} \sum_{\ell=0}^{c-1} e^{2 \pi i \left( \frac{ac k^2 + b k}{d} + \frac{ad \ell^2 + b \ell}{c} \right)} \\*
  &= \sum_{k=0}^{d-1} \sum_{\ell=0}^{c-1} e^{2 \pi i \left( \frac{ac^2 k^2 + bc k + ad^2 \ell^2 + bd \ell}{cd} \right)} \\*
  &= \sum_{k=0}^{d-1} \sum_{\ell=0}^{c-1} e^{2 \pi i \left( \frac{a(ck+d\ell)^2 + b(ck+d\ell)}{cd} \right)} \\*
  &= \sum_{n=0}^{cd-1} e^{2 \pi i \left( \frac{an^2 + bn}{cd} \right)} \\*
  &= G(a,b;cd).
 \end{align*}
 Finally, for what concerns point $(iii)$, using point $(ii)$, it is sufficient to prove the claim for $c = 2^k$ with $k \geqs 2$. First of all, notice that $an^2 + bn$ is even for every $n \in \N$. Then, as $n$ runs through all residue classes modulo $2^k$, we claim that $an^2+bn$ runs through all even residue classes modulo $2^k$ extactly twice. Indeed, for every integer $0 \leqs m \leqs 2^{k-1}-1$, we can consider the polynomial
 \[
  P_m(X) = aX^2 + bX - 2m \in (\Z/2^k \Z)[X],
 \]
 which projects to 
 \[
  \overline{P_m}(X) = X^2 + X \in (\Z/2 \Z)[X].
 \]
 Then, by Hensel's Lemma, the unique factorization 
 \[
  \overline{P_m}(X) = X(X+1) \in (\Z/2 \Z)[X]
 \]
 lifts to a unique factorization 
 \[
  P_m(X) = a(X-n_1(m))(X-n_2(m)) \in (\Z/2^k \Z)[X].
 \] 
 In particular, $P_m(X)$ has exactly two roots $n_1(m)$ and $n_2(m)$, and these are all distinct as $m$ runs through all residue classes modulo $2^{k-1}$. This means that $\{ n_1(0), n_2(0), \ldots, n_1(2^{k-1}-1), n_2(2^{k-1}-1) \} = \Z/ 2^k \Z$. Therefore, Equation~\eqref{E:triviality} can be established by a change of variable as follows:
 \begin{align*}
  G(a,b;2^k) 
  &= \sum_{m=0}^{2^{k-1}-1} e^{2 \pi i \frac{an_1(m)^2 + bn_1(m)}{2^k}} + e^{2 \pi i \frac{an_2(m)^2 + bn_2(m)}{2^k}} \\*
  &= 2 \sum_{m=0}^{2^{k-1}-1} e^{2 \pi i \frac{2m}{2^k}} \\*
  &= 2 \sum_{m=0}^{2^{k-1}-1} e^{2 \pi i \frac{m}{2^{k-1}}} \\*
  &= 0. \qedhere
 \end{align*}
\end{proof}

We are now ready to give formulas for the ribbon element $v_+ \in u_q \fsl_2$.

\begin{lemma}\label{L:ribbon}
 The ribbon element $v_+ \in u_q \fsl_2$ and its inverse $v_- \in u_q \fsl_2$ are given for $r \equiv 1 \mod 2$ by
 \begin{align}
  v_+ 
  &= \frac{i^{\frac{r-1}{2}}}{\sqrt{r}} 
  \sum_{a,b=0}^{r-1} \frac{\{ -1 \}^a}{[a]!} 
  q^{-\frac{(a+3)a}{2}+\frac{r+1}{2}(b-1)^2} F^a E^a K^{-a-b}, \label{E:ribbon_1} \\*
  v_- 
  &= \frac{i^{-\frac{r-1}{2}}}{\sqrt{r}} 
  \sum_{a,b=0}^{r-1} \frac{\{ 1 \}^a}{[a]!} 
  q^{\frac{(a+3)a}{2}-\frac{r+1}{2}(b-1)^2} F^a E^a K^{a+b}, \label{E:inverse_ribbon_1}
 \end{align}
 for $r \equiv 2 \mod 4$ by\footnote{Here $\left( \frac{2}{r'} \right)$ denotes the Jacobi symbol of $2$ modulo $r'$.}
 \begin{align}
  v_+ 
  &= \left( \frac{2}{r'} \right) \frac{i^{\frac{r'-1}{2}}}{\sqrt{r'}} 
  \sum_{a,b=0}^{r'-1} \frac{\{ -1 \}^a}{[a]!} 
  q^{-\frac{(a+3)a}{2}+\frac{(r'+1)^2}{2}(b-1)^2} F^a E^a K^{-a-b}, \label{E:ribbon_2} \\*
  v_- 
  &= \left( \frac{2}{r'} \right) \frac{i^{-\frac{r'-1}{2}}}{\sqrt{r'}} 
  \sum_{a,b=0}^{r'-1} \frac{\{ 1 \}^a}{[a]!} 
  q^{\frac{(a+3)a}{2}-\frac{(r'+1)^2}{2}(b-1)^2} F^a E^a K^{a+b}, \label{E:inverse_ribbon_2}
 \end{align}
 for $r \equiv 4 \mod 8$ by
 \begin{align}
  v_+ 
  &= \frac{i^{\frac{r''-1}{2}}}{\sqrt{r''}}
  \sum_{a=0}^{r'-1} \sum_{b=0}^{r''-1} \frac{\{ -1 \}^a}{[a]!}
  q^{-\frac{(a+3)a}{2} + \frac{(r''+1)^3}{2} (2b-1)^2} F^a E^a K^{-a-2b}, \label{E:ribbon_3} \\*
  v_- 
  &= \frac{i^{-\frac{r''-1}{2}}}{\sqrt{r''}} \sum_{a=0}^{r'-1} \sum_{b=0}^{r''-1} \frac{\{ 1 \}^a}{[a]!}
  q^{\frac{(a+3)a}{2} - \frac{(r''+1)^3}{2} (2b-1)^2} F^a E^a K^{a+2b}, \label{E:inverse_ribbon_3}
 \end{align}
 for $r \equiv 0 \mod 8$ by
 \begin{align}
  v_+ 
  &= \frac{1-i}{\sqrt{r'}} 
  \sum_{a=0}^{r'-1} \sum_{b=0}^{r''-1} \frac{\{ -1 \}^a}{[a]!} q^{-\frac{(a+3)a}{2} + 2b^2} F^a E^a K^{-a-2b-1}, \label{E:ribbon_4} \\*
  v_- &= 
  \frac{1+i}{\sqrt{r'}} 
  \sum_{a=0}^{r'-1} \sum_{b=0}^{r''-1} \frac{\{ 1 \}^a}{[a]!} q^{\frac{(a+3)a}{2} - 2b^2} F^a E^a K^{a+2b+1}. \label{E:inverse_ribbon_4}
 \end{align}
\end{lemma}

\begin{proof}
 All these formulas rely on explicit evaluations of generalized quadratic Gauss sums. Following \cite[Chapter~1]{BEW98}, these can be expressed in terms of the Legendre symbol
 \[
  \left( \frac{a}{p} \right) = 
  \begin{cases}
   0 & \mbox{ if } a \equiv 0 \mod p, \\
   1 & \mbox{ if } a \equiv b^2 \mod p \mbox{ for some } b \not\equiv 0 \mod p, \\
   -1 & \mbox{ otherwise, }
  \end{cases}
 \]
 for $a \in \N$ arbitrary and $p \in \N$ prime, and of the Jacobi symbol
 \[
  \left( \frac{a}{b} \right) = \left( \frac{a}{p_1} \right) \cdots \left( \frac{a}{p_k} \right)
 \]
 for $b = p_1 \cdots p_k \in \N$ arbitrary and $p_1,\ldots,p_k \in \N$ prime. We also introduce for convenience the notation
 \[
  \varepsilon_a := 
  \begin{cases}
   1 & \mbox{ if } a \equiv 1 \mod 4, \\
   i & \mbox{ if } a \equiv 3 \mod 4,
  \end{cases}
 \]
 for every odd $a \in \N$.

 Formulas for $v_+$ can be obtained by computing the Drinfeld element
 \begin{align*}
  u &= \frac{1}{r'} \sum_{a,b,c=0}^{r'-1} \frac{\{ 1 \}^a}{[a]!}
  q^{\frac{a(a-1)}{2} - 2bc} S(K^c F^a) K^b E^a  \\*
  &= \frac{1}{r'} \sum_{a,b,c=0}^{r'-1} (-1)^a \frac{\{ 1 \}^a}{[a]!}
  q^{-\frac{(a+3)a}{2} - 2bc} F^a K^{a+b-c} E^a \\*
  &= \frac{1}{r'} \sum_{a,b,c=0}^{r'-1} (-1)^a \frac{\{ 1 \}^a}{[a]!}
  q^{-\frac{(a+3)a}{2} - 2bc + 2a(a+b-c)} F^a E^a K^{a+b-c},
 \end{align*}
 where the first equality follows from
 \[
  S(K^c F^a) = (-1)^a q^{-(a+1)a} F^aK^{a-c}.
 \]
 If we set $d = -2a-b+c$ and $n = a+b$ we obtain
 \begin{align}
  u 
  &= \frac{1}{r'} \sum_{a,d,k=0}^{r'-1} \frac{\{ -1 \}^a}{[a]!}
  q^{-\frac{(a+3)a}{2} - 2(n^2+dn)} F^a E^a K^{-a-d} \nonumber \\*
  &= \frac{1}{r'} \sum_{a,d=0}^{r'-1} \left( \sum_{n=0}^{r'-1} e^{-2 \pi i \frac{2n^2+2dn}{r}} \right) \frac{\{ -1 \}^a}{[a]!}
  q^{-\frac{(a+3)a}{2}} F^a E^a K^{-a-d}. \label{E:u_general}
 \end{align}

 If $r \equiv 1 \mod 2$, then $r = r'$, and Equation~\eqref{E:u_general} becomes
 \begin{align}
  u &= \frac{1}{r} \sum_{a,d=0}^{r-1} \overline{G(2,2d;r)} \frac{\{ -1 \}^a}{[a]!}
  q^{-\frac{(a+3)a}{2}} F^a E^a K^{-a-d}. \label{E:u_odd}
 \end{align}
 Completing the square using Equation~\eqref{E:complete_the_square} with $f(2)=\frac{r+1}{2}$ gives
 \[
  G(2,2d;r) = G(2;r) e^{-2 \pi i \frac{\frac{(r+1)^2}{2}d^2}{r}} .
 \]
 Now, \cite[Theorem~1.5.2]{BEW98} yields
 \begin{equation*}
  G(2;r) = \left( \frac{2}{r} \right) \varepsilon_r \sqrt{r}.
 \end{equation*}
 Notice that, since 
 \[
  \left( \frac{2}{r} \right) = 
  \begin{cases}
   1 & \mbox{ if } r \equiv 1,7 \mod 8, \\
   -1 & \mbox{ if } r \equiv 3,5 \mod 8,
  \end{cases}
 \]
 we have
 \begin{equation}\label{E:Jacobi_epsilon}
  \left( \frac{2}{r} \right) \varepsilon_r = i^{-\frac{r-1}{2}}.
 \end{equation}
 This implies
 \begin{equation}\label{E:sG_odd}
  G(2;r) = i^{-\frac{r-1}{2}} \sqrt{r}.
 \end{equation}
 Furthermore, we have
 \begin{equation*}
  e^{-2 \pi i \frac{\frac{(r+1)^2}{2}d^2}{r}} = q^{-\frac{(r+1)^2}{2}d^2} = q^{-\frac{r+1}{2}d^2}.
 \end{equation*}
 Therefore, we obtain
 \begin{equation}\label{E:gG_odd}
  G(2,2d;r) = i^{-\frac{r-1}{2}} \sqrt{r} q^{-\frac{r+1}{2}d^2}.
 \end{equation}
 Then, Equations~\eqref{E:u_odd} and \eqref{E:gG_odd} yield
 \begin{align*}
  u &= \frac{i^{\frac{r-1}{2}}}{\sqrt{r}} \sum_{a,d=0}^{r-1} 
  \frac{\{ -1 \}^a}{[a]!} q^{-\frac{(a+3)a}{2}+\frac{r+1}{2}d^2} F^a E^a K^{-a-d}.
 \end{align*}
 Equation~\eqref{E:ribbon_1} now follows from $v_+ = uK^{-1}$.

 If $r \equiv 0 \mod 2$, then $r = 2r'$, and Equation~\eqref{E:u_general} becomes
 \begin{align}
  u &= \frac{1}{r'} \sum_{a,d=0}^{r'-1} \overline{G(1,d;r')} \frac{\{ -1 \}^a}{[a]!}
  q^{-\frac{(a+3)a}{2}} F^a E^a K^{-a-d}. \label{E:u_even}
 \end{align}

 If $r \equiv 2 \mod 4$, then $r' \equiv 1 \mod 2$. Completing the square using Equation~\eqref{E:complete_the_square} with $f(1)=r'+1$ gives
 \[
  G(1,d;r') = G(1;r') e^{-2 \pi i \frac{\frac{(r'+1)^2}{4} d^2}{r'}}.
 \]
 Now, \cite[Corollary~1.2.3]{BEW98} yields
 \begin{equation*}
  G(1;r') = \varepsilon_{r'} \sqrt{r'}.
 \end{equation*}
 Notice that Equation~\eqref{E:Jacobi_epsilon} implies
 \[
  \varepsilon_{r'} = \left( \frac{2}{r'} \right) i^{-\frac{r'-1}{2}},
 \]
 Furthermore, we have
 \[
  e^{-2 \pi i \frac{\frac{(r'+1)^2}{4}d^2}{r'}} = q^{-\frac{(r'+1)^2}{2}d^2}.
 \]
 Therefore, we obtain
 \begin{equation}\label{E:gG_2_mod_4}
  G(1,d;r') = \left( \frac{2}{r'} \right) i^{-\frac{r'-1}{2}} \sqrt{r'} q^{-\frac{(r'+1)^2}{2}d^2}.
 \end{equation}
 Then, Equations~\eqref{E:u_even} and \eqref{E:gG_2_mod_4} yield
 \begin{align*}
  u &= \left( \frac{2}{r'} \right) \frac{i^{\frac{r'-1}{2}}}{\sqrt{r'}} \sum_{a,d=0}^{r'-1} 
  \frac{\{ -1 \}^a}{[a]!} q^{-\frac{(a+3)a}{2}+\frac{(r'+1)^2}{2}d^2} F^a E^a K^{-a-d}.
 \end{align*}
 Equation~\eqref{E:ribbon_2} now follows from $v_+ = uK^{-1}$.

 If $r \equiv 4 \mod 8$, then $r' = 2r''$ with $r'' \equiv 1 \mod 2$. The multiplicativity property of Equation~\eqref{E:multiplicativity} gives
 \begin{align*}
  G(1,d;r') 
  &= G(2,d;r'') G(r'',d;2) = G(2,d;r'') G(1,d;2) \\*
  &= G(2,d;r'') \left( 1+(-1)^{d+1} \right),
 \end{align*}
 and completing the square using Equation~\eqref{E:complete_the_square} with $f(2)=\frac{(r''+1)^2}{2}$ gives
 \[
  G(2,d;r'') = G(2;r'') e^{-2 \pi i \frac{\frac{(r''+1)^4}{8}d^2}{r''}}.
 \]
 Then, Equation~\eqref{E:sG_odd} yields
 \[
  G(2;r'') = i^{-\frac{r''-1}{2}} \sqrt{r''}.
 \]
 Furthermore, we have
 \[
  e^{-2 \pi i \frac{\frac{(r''+1)^4}{8}d^2}{r''}} = q^{-\frac{(r''+1)^4}{2}d^2} = q^{-\frac{(r''+1)^3}{2}d^2}.
 \]
 Therefore, we obtain
 \begin{equation}\label{E:gG_4_mod_8}
  G(1,d;r') = 
  \begin{cases}
   0 & \mbox{ if } d \equiv 0 \mod 2, \\
   2 i^{-\frac{r''-1}{2}} \sqrt{r''} q^{-\frac{(r''+1)^3}{2}d^2} & \mbox{ if } d \equiv 1 \mod 2.
  \end{cases} 
 \end{equation}
 Then, Equations~\eqref{E:u_even} and \eqref{E:gG_4_mod_8} yield
 \begin{align*}
  u &= \frac{i^{\frac{r''-1}{2}}}{\sqrt{r''}} \sum_{a=0}^{r'-1} \sum_{b=0}^{r''-1} \frac{\{ -1 \}^a}{[a]!}
  q^{-\frac{(a+3)a}{2} + \frac{(r''+1)^3}{2} (2b-1)^2} F^a E^a K^{-a-2b+1}.
 \end{align*}
 Equation~\eqref{E:ribbon_3} now follows from $v_+ = uK^{-1}$.

 If $r \equiv 0 \mod 8$, then $r' \equiv 0 \mod 4$, and we rewrite Equation~\eqref{E:u_even} as
 \begin{align}
  u &= \frac{1}{r'} \sum_{a=0}^{r'-1} \sum_{b=0}^{r''-1} \overline{G(1,2b;r')} \frac{\{ -1 \}^a}{[a]!}
  q^{-\frac{(a+3)a}{2}} F^a E^a K^{-a-2b} \nonumber \\*
  &\hspace*{\parindent} \frac{1}{r'} \sum_{a=0}^{r'-1} \sum_{b=0}^{r''-1} \overline{G(1,2b+1;r')} \frac{\{ -1 \}^a}{[a]!}
  q^{-\frac{(a+3)a}{2}} F^a E^a K^{-a-2b-1}. \label{E:u_0_mod_8}
 \end{align}
 On the one hand, completing the square using Equation~\eqref{E:complete_the_square} with $f(1)=1$ gives
 \[
  G(1,2b;r') = G(1;r') e^{-2 \pi i \frac{b^2}{r'}}.
 \]
 Now, \cite[Corollary~1.2.3]{BEW98} yields
 \begin{align*}
  G(1;r') &= (1+i) \sqrt{r'}.
 \end{align*}
 Furthermore, we have
 \[
  e^{-2 \pi i \frac{b^2}{r'}} = q^{-2b^2}.
 \]
 Therefore, we obtain
 \begin{equation}\label{E:gG_0_mod_8_even}
  G(1,2b;r') = (1+i) \sqrt{r'} q^{-2b^2}.
 \end{equation}
 On the other hand, the triviality property of Equation~\eqref{E:triviality} gives
 \begin{equation}\label{E:gG_0_mod_8_odd}
  G(1,2b+1;r') = 0.
 \end{equation}
 Then, Equations~\eqref{E:u_0_mod_8}, \eqref{E:gG_0_mod_8_even}, and \eqref{E:gG_0_mod_8_odd} yield
 \begin{align*}
  u &= 
  \frac{1-i}{\sqrt{r'}} 
  \sum_{a=0}^{r'-1} \sum_{b=0}^{r''-1} \frac{\{ -1 \}^a}{[a]!} 
  q^{-\frac{(a+3)a}{2} + 2b^2} F^a E^a K^{-a-2b}.
 \end{align*}
 Equation~\eqref{E:ribbon_4} now follows from $v_+ = uK^{-1}$.

 Similarly, formulas for $v_-$ can be obtained by computing the inverse Drinfeld element
 \begin{align*}
  u^{-1} &= \frac{1}{r'} \sum_{a,b,c=0}^{r'-1} \frac{\{ 1 \}^a}{[a]!}
  q^{\frac{a(a-1)}{2} - 2bc} K^c F^a S^2(K^b E^a) \\*
  &= \frac{1}{r'} \sum_{a,b,c=0}^{r'-1} \frac{\{ 1 \}^a}{[a]!}
  q^{\frac{(a+3)a}{2} - 2bc} K^c F^a K^b E^a \\*
  &= \frac{1}{r'} \sum_{a,b,c=0}^{r'-1} \frac{\{ 1 \}^a}{[a]!}
  q^{\frac{(a+3)a}{2} - 2bc + 2ab} F^a E^a K^{b+c}.
 \end{align*}
 where the first equality follows from
 \[
  S^2(K^b E^a) = q^{2a} K^b E^a.
 \]
 If we set $d = a-b-c$ and $n = b$ we obtain
 \begin{align}
  u^{-1} 
  &= \frac{1}{r'} \sum_{a,d,n=0}^{r'-1} \frac{\{ 1 \}^a}{[a]!}
  q^{\frac{(a+3)a}{2}+2(n^2+dn)} F^a E^a K^{a-d} \nonumber \\*
  &= \frac{1}{r'} \sum_{a,d=0}^{r'-1} \left( \sum_{n=0}^{r'-1} e^{2 \pi i \frac{2n^2+2dn}{r}} \right) \frac{\{ 1 \}^a}{[a]!}
  q^{\frac{(a+3)a}{2}} F^a E^a K^{a-d}. \label{E:u-1_general}
 \end{align}

 If $r \equiv 1 \mod 2$, then $r = r'$, and Equation~\eqref{E:u-1_general} becomes
 \begin{align}
  u^{-1} &= \frac{1}{r} \sum_{a,d=0}^{r-1} G(2,2d;r) \frac{\{ 1 \}^a}{[a]!}
  q^{\frac{(a+3)a}{2}} F^a E^a K^{a-d}. \label{E:u-1_odd}
 \end{align}
 Then, Equations~\eqref{E:u-1_odd} and \eqref{E:gG_odd} yield
 \begin{align*}
  u^{-1} &= \frac{i^{-\frac{r-1}{2}}}{\sqrt{r}} \sum_{a,d=0}^{r-1} 
  \frac{\{ 1 \}^a}{[a]!} q^{\frac{(a+3)a}{2}-\frac{r+1}{2}d^2} F^a E^a K^{a-d} \\*
  &= \frac{i^{-\frac{r-1}{2}}}{\sqrt{r}} \sum_{a,c=0}^{r-1} 
  \frac{\{ 1 \}^a}{[a]!} q^{\frac{(a+3)a}{2}-\frac{r+1}{2}c^2} F^a E^a K^{a+c}.
 \end{align*}
 Equation~\eqref{E:inverse_ribbon_1} now follows from $v_- = u^{-1}K$.

 If $r \equiv 0 \mod 2$, then $r = 2r'$, and Equation~\eqref{E:u-1_general} becomes
 \begin{align}
  u^{-1} &= \frac{1}{r'} \sum_{a,d=0}^{r'-1} G(1,d;r') \frac{\{ 1 \}^a}{[a]!}
  q^{\frac{(a+3)a}{2}} F^a E^a K^{a-d}. \label{E:u-1_even}
 \end{align}

 If $r \equiv 2 \mod 4$, then $r' \equiv 1 \mod 2$. Then, Equations~\eqref{E:u-1_even} and \eqref{E:gG_2_mod_4} yield
 \begin{align*}
  u^{-1} &= \left( \frac{2}{r'} \right) \frac{i^{-\frac{r'-1}{2}}}{\sqrt{r'}} \sum_{a,d=0}^{r'-1} 
  \frac{\{ 1 \}^a}{[a]!} q^{\frac{(a+3)a}{2}-\frac{(r'+1)^2}{2}d^2} F^a E^a K^{a-d} \\*
  &= \left( \frac{2}{r'} \right) \frac{i^{-\frac{r'-1}{2}}}{\sqrt{r'}} \sum_{a,c=0}^{r'-1} 
  \frac{\{ 1 \}^a}{[a]!} q^{\frac{(a+3)a}{2}-\frac{(r'+1)^2}{2}c^2} F^a E^a K^{a+c}.
 \end{align*}
 Equation~\eqref{E:inverse_ribbon_2} now follows from $v_- = u^{-1}K$.

 If $r \equiv 4 \mod 8$, then $r' = 2r''$ with $r'' \equiv 1 \mod 2$. Then, Equations~\eqref{E:u-1_even} and \eqref{E:gG_4_mod_8} yield
 \begin{align*}
  u^{-1} &= \frac{i^{-\frac{r''-1}{2}}}{\sqrt{r''}} \sum_{a=0}^{r'-1} \sum_{c=0}^{r''-1} \frac{\{ 1 \}^a}{[a]!}
  q^{\frac{(a+3)a}{2} - \frac{(r''+1)^3}{2} (2c+1)^2} F^a E^a K^{a-2c-1} \\*
  &= \frac{i^{-\frac{r''-1}{2}}}{\sqrt{r''}} \sum_{a=0}^{r'-1} \sum_{b=0}^{r''-1} \frac{\{ 1 \}^a}{[a]!}
  q^{\frac{(a+3)a}{2} - \frac{(r''+1)^3}{2} (2b-1)^2} F^a E^a K^{a+2b-1}.
 \end{align*}
 Equation~\eqref{E:inverse_ribbon_3} now follows from $v_- = u^{-1}K$.

 If $r \equiv 0 \mod 8$, then $r' \equiv 0 \mod 4$, and we rewrite Equation~\eqref{E:u-1_even} as
 \begin{align}
  u^{-1} &= \frac{1}{r'} \sum_{a=0}^{r'-1} \sum_{b=0}^{r''-1} G(1,2b;r') \frac{\{ 1 \}^a}{[a]!}
  q^{\frac{(a+3)a}{2}} F^a E^a K^{a-2b} \nonumber \\*
  &\hspace*{\parindent} \frac{1}{r'} \sum_{a=0}^{r'-1} \sum_{b=0}^{r''-1} G(1,2b+1;r') \frac{\{ 1 \}^a}{[a]!}
  q^{\frac{(a+3)a}{2}} F^a E^a K^{a-2b-1}. \label{E:u-1_0_mod_8}
 \end{align}
 Then, Equations~\eqref{E:u-1_0_mod_8}, \eqref{E:gG_0_mod_8_even}, and \eqref{E:gG_0_mod_8_odd} yield
 \begin{align*}
  u^{-1} &= \frac{1+i}{\sqrt{r'}} 
  \sum_{a=0}^{r'-1} \sum_{b=0}^{r''-1} \frac{\{ 1 \}^a}{[a]!} 
  q^{\frac{(a+3)a}{2} - 2b^2} F^a E^a K^{a-2b} \\*
  &= \frac{1+i}{\sqrt{r'}} 
  \sum_{a=0}^{r'-1} \sum_{c=0}^{r''-1} \frac{\{ 1 \}^a}{[a]!} 
  q^{\frac{(a+3)a}{2} - 2c^2} F^a E^a K^{a+2c}.
 \end{align*}
 Equation~\eqref{E:inverse_ribbon_4} now follows from $v_- = u^{-1}K$.
\end{proof}

Formulas for the monodromy matrix $M \in u_q \fsl_2 \otimes u_q \fsl_2$ are easier to establish.

\begin{lemma}\label{L:monodromy}
 The M-matrix $M \in u_q \fsl_2 \otimes u_q \fsl_2$ and its inverse $M^{-1} \in u_q \fsl_2 \otimes u_q \fsl_2$ are given for $r \equiv 1 \mod 2$ and $r \equiv 2 \mod 4$ by
 \begin{align}
  M &= \frac{1}{r'} \sum_{a,b,c,d=0}^{r'-1} \frac{\{ 1 \}^{a+b}}{[a]![b]!} \nonumber \\*
  &\hspace*{\parindent} q^{\frac{a(a-1)+b(b-1)}{2} - 2b^2 - (r'+1)cd} K^{-b+c} F^b E^a \otimes K^{b+d} E^b F^a, \label{E:M_1} \\
  M^{-1} &=
  \frac{1}{r'} \sum_{a,b,c,d=0}^{r'-1} \frac{\{ -1 \}^{a+b}}{[a]![b]!}  \nonumber \\*
  &\hspace*{\parindent} q^{-\frac{a(a-1)+b(b-1)}{2} + 2b^2 + (r'+1)cd} E^a F^b K^{-b+c} \otimes F^a E^b K^{b+d}, \label{E:M_inverse_1}
 \end{align}
 for $r \equiv 0 \mod 4$ by
 \begin{align}
  M &=
  \frac{1}{r''} \sum_{a,b=0}^{r'-1} \sum_{c,d=0}^{r''-1} \frac{\{ 1 \}^{a+b}}{[a]![b]!} \nonumber \\*
  &\hspace*{\parindent} q^{\frac{a(a-1)+b(b-1)}{2} - 2b^2 - 4cd} K^{-b+2c} F^b E^a \otimes K^{b+2d} E^b F^a, \label{E:M_2} \\
  M^{-1} &=
  \frac{1}{r''} \sum_{a,b=0}^{r'-1} \sum_{c,d=0}^{r''-1} \frac{\{ -1 \}^{a+b}}{[a]![b]!} \nonumber \\*
  &\hspace*{\parindent} q^{\frac{a(a-1)+b(b-1)}{2} + 2b^2 + 4cd} E^a F^b K^{-b+2c} \otimes F^a E^b K^{b+2d}. \label{E:M_inverse_2}
 \end{align}
\end{lemma}

\begin{proof}
 For what concerns $M$, we have
 \begin{align*}
  M 
  &= \frac{1}{r'^2} \sum_{a,b,c,d,e,f=0}^{r'-1} \frac{\{ 1 \}^{a+b}}{[a]![b]!} \\*
  &\hspace*{\parindent} q^{\frac{a(a-1)+b(b-1)}{2} - 2cd - 2ef} K^f F^b K^c E^a \otimes K^e E^b K^d F^a \\
  &= \frac{1}{r'^2} \sum_{a,b,c,d,e,f=0}^{r'-1} \frac{\{ 1 \}^{a+b}}{[a]![b]!} \\*
  &\hspace*{\parindent} q^{\frac{a(a-1)+b(b-1)}{2} - 2cd - 2ef + 2b(c-d)} K^{c+f} F^b E^a \otimes K^{d+e} E^b F^a \\
  &= \frac{1}{r'^2} \sum_{a,b,e,f,g,h=0}^{r'-1} \frac{\{ 1 \}^{a+b}}{[a]![b]!} \\*
  &\hspace*{\parindent} q^{\frac{a(a-1)+b(b-1)}{2} - 2(f-g)(e-h) - 2ef - 2b(f-g-e+h)} K^g F^b E^a \otimes K^h E^b F^a \\
  &= \frac{1}{r'^2} \sum_{a,b,e,g,h=0}^{r'-1} \left( \sum_{f=0}^{r'-1} e^{- \frac{2 \pi i}{r'} 2\frac{r'}{r}f(b+2e-h)} \right) \frac{\{ 1 \}^{a+b}}{[a]![b]!} \\*
  &\hspace*{\parindent} q^{\frac{a(a-1)+b(b-1)}{2} + 2(b+g)(e-h) + 2bg} K^g F^b E^a \otimes K^h E^b F^a \\
  &= \frac{1}{r'} \sum_{a,b,e,g=0}^{r'-1} \frac{\{ 1 \}^{a+b}}{[a]![b]!} \\*
  &\hspace*{\parindent} q^{\frac{a(a-1)+b(b-1)}{2} - 2(b+g)(b+e) + 2bg} K^g F^b E^a \otimes K^{b+2e} E^b F^a.
 \end{align*}

 If $r \equiv 1 \mod 2$ or $r \equiv 2 \mod 4$, then $r' \equiv 1 \mod 2$, so setting $g = -b+c$ and $e = \frac{r'+1}{2} d$ yields Equation~\eqref{E:M_1}.

 On the other hand, if $r \equiv 0 \mod 4$, then we have
 \begin{align*}
  M &= \frac{1}{r'} \sum_{a,b,g=0}^{r'-1} \sum_{e=0}^{r''-1} \frac{\{ 1 \}^{a+b}}{[a]![b]!} \left( 1 + (-1)^{b+g} \right) \\*
  &\hspace*{\parindent} q^{\frac{a(a-1)+b(b-1)}{2} - 2(b+g)(b+e) + 2bg} K^g F^b E^a \otimes K^{b+2e} E^b F^a \\*
  &= \frac{2}{r'} \sum_{a,b=0}^{r'-1} \sum_{c,e=0}^{r''-1} \frac{\{ 1 \}^{a+b}}{[a]![b]!} \\*
  &\hspace*{\parindent} q^{\frac{a(a-1)+b(b-1)}{2} - 2b^2 - 4ce} K^{-b+2c} F^b E^a \otimes K^{b+2e} E^b F^a. 
 \end{align*}
 Therefore, setting $e = d$ yields Equation~\eqref{E:M_2}.

 Similarly, for what concerns $M^{-1}$, we have
 \begin{align*}
  M^{-1} 
  &= \frac{1}{r'^2} \sum_{a,b,c,d,e,f=0}^{r'-1} \frac{\{ -1 \}^{a+b}}{[a]![b]!} \\*
  &\hspace*{\parindent} q^{-\frac{a(a-1)+b(b-1)}{2} + 2cd + 2ef}  E^a K^c F^b K^f \otimes F^a K^d E^b K^e \\
  &= \frac{1}{r'^2} \sum_{a,b,c,d,e,f=0}^{r'-1} \frac{\{ -1 \}^{a+b}}{[a]![b]!} \\*
  &\hspace*{\parindent} q^{-\frac{a(a-1)+b(b-1)}{2} + 2cd + 2ef - 2b(c-d)} E^a F^b K^{c+f} \otimes F^a E^b K^{d+e} \\
  &= \frac{1}{r'^2} \sum_{a,b,e,f,g,h=0}^{r'-1} \frac{\{ -1 \}^{a+b}}{[a]![b]!} \\*
  &\hspace*{\parindent} q^{-\frac{a(a-1)+b(b-1)}{2} + 2(f-g)(e-h) + 2ef + 2b(f-g-e+h)} E^a F^b K^g \otimes F^a E^b K^h \\
  &= \frac{1}{r'^2} \sum_{a,b,e,g,h=0}^{r'-1} \left( \sum_{f=0}^{r'-1} e^{\frac{2 \pi i}{r'} 2\frac{r'}{r}f(b+2e-h)} \right) \frac{\{ -1 \}^{a+b}}{[a]![b]!} \\*
  &\hspace*{\parindent} q^{-\frac{a(a-1)+b(b-1)}{2} - 2(b+g)(e-h) - 2bg} E^a F^b K^g \otimes F^a E^b K^h \\
  &= \frac{1}{r'} \sum_{a,b,e,g=0}^{r'-1} \frac{\{ -1 \}^{a+b}}{[a]![b]!} \\*
  &\hspace*{\parindent} q^{-\frac{a(a-1)+b(b-1)}{2} + 2(b+g)(b+e) - 2bg} E^a F^b K^g \otimes F^a E^b K^{b+2e}.
 \end{align*}

 If $r \equiv 1 \mod 2$ or $r \equiv 2 \mod 4$, then $r' \equiv 1 \mod 2$, so setting $g = -b+c$ and $e = \frac{r'+1}{2} d$ yields Equation~\eqref{E:M_inverse_1}.

 On the other hand, if $r \equiv 0 \mod 4$, then we have
 \begin{align*}
  M^{-1} &= \frac{1}{r'} \sum_{a,b,g=0}^{r'-1} \sum_{e=0}^{r''-1} \frac{\{ -1 \}^{a+b}}{[a]![b]!} \left( 1 + (-1)^{b+g} \right) \\*
  &\hspace*{\parindent} q^{-\frac{a(a-1)+b(b-1)}{2} + 2(b+g)(b+e) - 2bg} E^a F^b K^g \otimes F^a E^b K^{b+2e} \\*
  &= \frac{2}{r'} \sum_{a,b=0}^{r'-1} \sum_{c,e=0}^{r''-1} \frac{\{ -1 \}^{a+b}}{[a]![b]!} \\*
  &\hspace*{\parindent} q^{-\frac{a(a-1)+b(b-1)}{2} + 2b^2 + 4ce} E^a F^b K^{-b+2c} \otimes F^a E^b K^{b+2e}.
 \end{align*}
 Therefore, setting $e = d$ yields Equation~\eqref{E:M_inverse_2}.
\end{proof}

\section{Instability of 2-exotic pairs}

In this appendix we give a short explicit proof of the instability of (potential) $2$-exotic phenomena under boundary connected sum with $S^2 \times D^2$. We recall that the operation of boundary connected sum is denoted ${} \bcs {}$.

\begin{lemma}\label{L:instability}
 If two $4$-di\-men\-sion\-al $2$-handlebodies $X$ and $Y$ are diffeomorphic, then there exists an integer $k \geqs 0$ such that $X \bcs {} (S^2 \times D^2)^{{} \bcs k}$ and $Y \bcs {} (S^2 \times D^2)^{{} \bcs k}$ are $2$-e\-quiv\-a\-lent.
\end{lemma}

\begin{proof}
 Every diffeomorphism $f : X \to Y$ is implemented by a sequence of $2$-handle slides and of creation/removal of canceling pairs of handles of index $1/2$ and $2/3$. Suppose that $f : X \to Y$ requires $k$ creations and $k$ removals of canceling $2/3$-handle pairs. Then, up to reordering all these handle moves, we can decompose $f$ as $f''' \circ f'' \circ f'$, where: 
 \begin{enumerate}
  \item $f' : X \to X'$ creates all $k$ canceling $2/3$-handle pairs;
  \item $f'' : X' \to Y'$ is a $2$-deformation;
  \item $f''' : Y' \to Y$ removes all $k$ canceling $2/3$-handle pairs.
 \end{enumerate}
 Up to $2$-handle slides, the attaching sphere of a $2$-han\-dle that admits a canceling $3$-han\-dle is isotopic to a $0$-framed unknot, see \cite[Figure~5.15]{GS99}. Therefore, the $4$-di\-men\-sion\-al $2$-handlebody obtained from $X'$ by removing all $3$-han\-dles is $2$-equivalent to $X \bcs {} (S^2 \times D^2)^{{} \bcs k}$, and the $4$-di\-men\-sion\-al $2$-handlebody obtained from $Y'$ by removing all $3$-han\-dles is $2$-equivalent to $Y \bcs {} (S^2 \times D^2)^{{} \bcs k}$. Then $f''$ induces by restriction a $2$-deformation 
 \[
  \tilde{f} : X \bcs {} (S^2 \times D^2)^{{} \bcs k} \to Y \bcs {} (S^2 \times D^2)^{{} \bcs k}. \qedhere
 \]
\end{proof}

\end{document}